\documentclass{amsart}       

\usepackage{amsmath}
\usepackage{latexsym,amssymb,dsfont}
\usepackage{thmtools,mathtools,amsthm}
\usepackage{stmaryrd}
\usepackage{tikz-cd}
\usepackage{quiver}
\usepackage{url}

\usepackage[style=numeric,maxnames=10,backend=biber,
useprefix=true,hyperref=true]{biblatex}
\theoremstyle{definition}
\newtheorem{definition}{Definition}[section]
\newtheorem{example}[definition]{Example}
\newtheorem{remark}[definition]{Remark}

\theoremstyle{theorem}
\newtheorem{theorem}[definition]{Theorem}
\newtheorem{lemma}[definition]{Lemma}
\newtheorem{proposition}[definition]{Proposition}
\newtheorem{corollary}[definition]{Corollary}

\usepackage[mathscr]{euscript}

\usepackage{mymacros}

\usepackage[colorlinks=true]{hyperref}
\hypersetup{allcolors=[rgb]{0.1,0.1,0.4}}
\usepackage{cleveref}

\usepackage{mymacros}

\begin{document}

\title{Two-sided cartesian fibrations of synthetic $(\infty,1)$-categories}

\author[Jonathan Weinberger]{Jonathan Weinberger}
\address{Dept.~of Mathematics, Krieger School of Arts \& Sciences, Johns Hopkins University, 3400 N Charles St, Baltimore, MD 21218, USA}
\email{jweinb20@jhu.edu}

\date{\today}

\maketitle

\begin{abstract}
Within the framework of Riehl--Shulman's synthetic $(\infty,1)$-category theory, we present a theory of two-sided cartesian fibrations. Central results are several characterizations of the two-sidedness condition \`{a} la Chevalley, Gray, Street, and Riehl--Verity, a two-sided Yoneda Lemma, as well as the proof of several closure properties.

Along the way, we also define and investigate a notion of fibered or sliced fibration which is used later to develop the two-sided case in a modular fashion. We also briefly discuss \emph{discrete} two-sided cartesian fibrations in this setting, corresponding to $(\infty,1)$-distributors.

The systematics of our definitions and results closely follows Riehl--Verity's $\infty$-cosmos theory, but formulated internally to Riehl--Shulman's simplicial extension of homotopy type theory. All the constructions and proofs in this framework are by design invariant under homotopy equivalence. Semantically, the synthetic $(\infty,1)$-categories correspond to internal $(\infty,1)$-categories implemented as Rezk objects in an arbitrary given $(\infty,1)$-topos. \\[1.5\baselineskip]
\end{abstract}

\section{Introduction}
\label{intro}

\subsection{Synthetic higher category theory}

\subsubsection{$\infty$-Cosmos theory}

In recent years, there have been ongoing investigations into higher category theory using \emph{synthetic} approaches where the basic objects already have intrinsic homotopical meaning. The program of Riehl--Verity~\cite{RV21} develops $(\infty,n)$-category theory, for $0 \le n \le \infty$, using the notion of \emph{$\infty$-cosmos}\footnote{The terminology has been suggested by Peter May~\cite[Acknowledgments]{RV21}, after~\cite{Street-Cosmoi-I}.} which provides a semi-strict way to present $(\infty,2)$-categorical universes of a wide range of notions of higher categories (\eg~quasi-categories, Rezk spaces/objects, $n$-complicial sets, $\Theta_n$-spaces, \ldots)---as well as derived notions thereof such as slices, discrete objects, fibrations,~\etc. This has sparked a very powerful program to establish $\infty$-category theory in model-independent way. Intriguingly, they demonstrate how parts of the theory can even be reduced essentially to (strict) $2$-categorical arguments~\cite{RV2cat,RVscratch} which makes precise a range of very helpful intuitions that previously had lacked justification at this level of conceptionality, rigor, and comprehensiveness. Indeed, $\infty$-cosmos theory systematically employs techniques from \emph{Australian} or \emph{formal category theory} to work internal to sufficiently complete $(\infty,2)$-categories that are thought of as universes for different notions of higher categories.

\subsubsection{Simplicial type theory}

While $\infty$-cosmos theory is based on traditional set theory, there also have been approaches to higher category theory within various \emph{type theories}. A type theory is a formal system that allows for doing mathematics in a more \emph{structuralist} way, as opposed to the \emph{materialist} nature of set theory. The advent of homotopy type theory and univalent foundations (HoTT/UF) due to Voevodsky~\cite{VVTySys}, Awodey--Warren~\cite{AW05}, \emph{et al.}~has led to a new area of research yielding synthetic accounts to homotopy theory. This provides a significant generalization of the pioneering earlier work by Hofmann--Streicher~\cite{HS94,HS97} on the ($1$-)groupoid interpretation of Martin--Löf type theory. Roughly, the basic objects in HoTT are homotopy types, \ie~spaces up to homotopy \aka~weak $\infty$-groupoids, \cf~\cite{KL18,StrSimp}. In fact, after various works towards extending the class of models such as~\cite{ShuReedy,ShuInv,ShuInvEI,CisUniv} Shulman was able in 2019 to prove Awodey's conjecture that any $\inftyone$-topos (in the sense of Grothendieck--Rezk--Lurie~\cite{rezTop,LurHTT}) admits a model-categorical presentation that supports a full model of HoTT.

This does not yield an account to synthetic higher categories right away. Indeed, defining a notion of $(\infty,1)$-category in plain HoTT has been a long standing open problem.

Thus, variations of HoTT have been proposed that indeed make it possible to reason about higher categories using (complete) Segal-type formalisms such as~\emph{two-level type theory}~\cite{CapriottiPhD,2ltt} based on Voevodsky's \emph{homotopy type system (HTS)}~\cite{VV-HTS}. Another such approach is \emph{simplicial (homotopy) type theory}, introduced in \cite{RS17}.\footnote{\Cf~\cite{B19} for a high-level overview.} There, one adds another layer of strict simplicial shapes, containing \eg~the directed cubes $\I^n$, standard $n$-simplices $\Delta^n$, the $(n,k)$-horns $\Lambda_k^n$, \etc. Type families can then also depend on shapes. Another feature is the notion of \emph{extension type}: given a shape inclusion (``cofibration'') $\Phi \subseteq \Psi$, a type family $A: \Psi \to \UU$, and a partial section $a: \prod_\Phi \Phi^*A$, there exists the type
\[ \exten{\Phi}{A}{\Psi}{a}\]
of sections $b: \prod_\Psi A$ satisfying $a(t) \jdeq b(t)$ for all $t:\Phi$.

Together with the given shapes, this induces a sensible notion of (dependent) \emph{hom-type} $\hom_A:A \to A \to \UU$, as well as synthetic versions of the Segal and Rezk conditions. Riehl--Shulman's work on discrete fibrations in this setting was later generalized to the case of co-/cartesian fibrations~\cite{BW21} and lextensive \mbox{(bi-)fibrations}~\cite[Chapter~4]{jw-phd}. The text at hand presents a further generalization of the one-sided (co-/)cartesian case to the \emph{two-sided} cartesian case.

\subsection{Co-/cartesian families}

Let $\UU$ be a universe type, \cf~\cite[Section~1.3]{hottbook}. A \emph{cocartesian family} $P:B \to \UU$ over a \emph{Rezk type} $B$ (\aka~a synthetic $\inftyone$-category) precisely captures the idea of a \emph{functorial} type family valued in ($\UU$-small) synthetic $\inftyone$-categories: an arrow $u:a \to_B b$ in the type $B$ induces covariantly a functor $u_!^P \jdeq : u_! : P\,a \to P\,b$ between the fibers:
\[\begin{tikzcd}
	a & b && {P\,a} & {P\,b}
	\arrow["u", from=1-1, to=1-2]
	\arrow["\leadsto"{description}, draw=none, from=1-2, to=1-4]
	\arrow["{u_!}", from=1-4, to=1-5]
\end{tikzcd}\]
Dually, a \emph{cartesian family} $P:B \to \UU$ transforms \emph{contravariantly functorially}, \ie~an arrow $u:b \to_B a$ induces a functor $u_P^* \jdeq : u^*: P\,a \to P\,b$:
\[\begin{tikzcd}
	b & a && {P\,b} & {P\,a}
	\arrow["u", from=1-1, to=1-2]
	\arrow["{u^*}"', from=1-5, to=1-4]
	\arrow["\rightsquigarrow"{description}, draw=none, from=1-2, to=1-4]
\end{tikzcd}\]
These notions generalize the familiar $1$-categorical concepts of a Grothendieck op-/fibration~\cite{streicher2020fibered} to the synthetic higher setting. Fibrations of $\inftyone$-categories have initally been studied by Joyal~\cite{JoyNotesQcat} and Lurie~\cite{LurHTT} for quasi-categories, and consecutively very notably by Ayala--Francis~\cite{AFfib}, Barwick--Shah~\cite{BarwickShahFib}, Barwick--Dotto--Glasman--Nardin--Shah~\cite{barwick2016}, Rezk~\cite{rezk2017stuff}, Cisinski~\cite{CisInfBook}, Nguyen~\cite{NguyPhD}, Cisinski--Nguyen~\cite{cisinski2022universal}, as well as other authors. Particularly crucial for our approach are the treatments by Rasekh~\cite{rasekh2021cartesian,Ras17Cart} for Rezk spaces, and Riehl--Verity's $\infty$-cosmos theory~\cite{RV21,RVyoneda} which generalizes concepts and results due to Street~\cite{StrYon,StrBicat,StrBicatCorr} and Gray~\cite{GrayFib} to a rigorous and powerful framework of model-independent higher category theory.

In the context of simplicial type theory, Riehl--Shulman have considered \emph{covariant families}~\cite[Section~8]{RS17}, functioning as a synthetic variant of (discrete) left fibrations~\cite{JoyNotesQcat}.

\subsection{Two-sided cartesian families}
Two-sided cartesian families are type families $P:A \to B \to \UU$ which fibrationally are presented by \emph{spans}
\[\begin{tikzcd}
	& E \\
	A && B
	\arrow["\xi"', two heads, from=1-2, to=2-1]
	\arrow["\pi", two heads, from=1-2, to=2-3]
\end{tikzcd}\]
where $\xi$ is cocartesian, $\pi$ is cartesian, and some compatibility conditions between the two respective liftings are satisfied. An instructive example is given by the ``hom span'' $\partial_1: A \leftarrow A^{\Delta^1} \to A:\partial_0$ of a Rezk type $A$, and from ensuing properties one also obtains comma spans $\partial_1 : C \leftarrow \comma{f}{g} \rightarrow B : \partial_0$ induced by a span $g:C \leftarrow A \to B:f$ of functors.\footnote{In fact, these are even \emph{discrete} two-sided fibrations,~\cf~\cite[Section~8.6]{RS17}, \cite[Section~7.2]{RV21}, \cite[Theorem~2.3.3]{LorRieCatFib}.} Originally, this notion goes back to Street~\cite{StrYon,StrBicat,StrBicatCorr,LorRieCatFib}, and it has been generalized and thoroughly investigated for $\infty$-cosmoses~\cite[Chapter~7]{RV21}. Semantically, two-sided families correspond to \emph{categorical $\inftyone$-distributors},~\ie~bifunctors $A^{\Op} \times B \to \Cat$ into the $\inftyone$-category of small $\inftyone$-categories.\footnote{Even though this cannot be expressed in our theory yet,~\cf~\cite[Chapter~7]{jw-phd}.} The significance for $\infty$-cosmos theory is that the discrete variant,~\ie~the $\inftyone$-distributors or \emph{modules}, form a \emph{virtual equipment}~\cite{CruShuMulticat}, a rich double-categorical structure that presents the formal $\infty$-category theory of an $\infty$-cosmos. The Model Independence Theorem states that a biequivalence between $\infty$-cosmoses lifts to a biequivalence of the associated virtual equipments.\footnote{Note the parallel to axiomatic homotopy theory where a Quillen equivalence between ``homotopy theories'' presented through model categories lifts to an equivalence of their associated homotopy categories.} Here, however, we will deal with the categorical two-sided case. Namely, we will provide a structured analysis, leading up to characterizations and closure properties generalizing the one-sided case. This follows the thread of of~\cite[Section~7.1]{RV21}, but with a more explicit account of various (auxiliary) notions of fibered (or \emph{sliced}) fibrations, owed to the lack of categorical universes in the present theory. Our treatise nevertheless often times make use of techniques from ``formal'' category theory, by reasoning about the various conditions in terms of statements about (fibered) adjunctions, and their closure properties. We view this as a fruitful pratical effect of the $\infty$-cosmological philosophy on the synthetic theory formulated in simplicial type theory.

Our treatise ends with a two-sided Yoneda Lemma, and a (very brief) note on discrete two-sided families.

\subsection{Structure of the paper}

A two-sided cartesian family $P: A \to B \to \UU$ corresponds to a span $\xi: A \leftarrow E \rightarrow B : \pi$ with $\xi$ a cocartesian and $\pi$ a cartesian fibration such that the lifts are compatbile in a certain sense. But such a family or span, resp., can also be understood as a \emph{cartesian} functor
\[\begin{tikzcd}
	E && {A \times B} \\
	& B
	\arrow["\varphi", from=1-1, to=1-3]
	\arrow["\pi"', two heads, from=1-1, to=2-2]
	\arrow["q", two heads, from=1-3, to=2-2]
\end{tikzcd}\]
between cartesian fibrations that in addition is a \emph{fibered} or \emph{sliced cocartesian fibration} over $B$. By duality and the Chevalley criteria for (ordinary) co-/cartesian fibrations and functors, resp., there are several equivalent way to express this fact, \cf~\Cref{thm:char-two-sid}, after~\cite[Theorem~7.1.4]{RV21}.

Semantically, the condition is that $\varphi: \pi \to_B q$ is an object of the $\infty$-cosmos $\mathrm{co}{\mathcal C}\mathrm{art}(\mathcal{C}\mathrm{art}(\mathcal K))$ for the ambient $\infty$-cosmos $\mathcal K$ (which is of the form $\mathcal K=\mathscr{R}ezk_\mathscr E$, \cf~\cite[Proposition~E.3.7]{RV21} for an $\infty$-cosmos $\mathscr E$ presenting an $\inftyone$-topos.).

\Cref{sec:intro-fib-syn} is a technical introduction and a recap of the basic fibrational theory set up in~\cite{BW21}. We then work towards two-sided cartesian families by first discussing sliced cocartesian families in~\Cref{sec:sl-cocart-fams}, in particular their characterizations \`{a} la Chevalley and their closure properties, \cf~\Cref{ssec:sl-cocart-fams}. Next, we specialize to the case of cocartesian families in cartesian families, \cf~\Cref{ssec:cocart-fams-in-cart}.

\Cref{sec:2s-cart} then introduces and the notion and establishes basic properties of two-sided cartesian families and fibrations, resp. In~\Cref{ssec:two-var} we provide the basic formalism of two-variable families and bifibers. Finally, in~\Cref{ssec:2s-cart-fams} we define two-sided cartesian families and prove a number of Chevalley-like characterization criteria.

In~\Cref{sec:2s-cart-fun} we introduce two-sided cartesian functors and prove several closure properties (a subset of which corresponds to the $\infty$-cosmological closure properties, \cf~\cite[Section~7.2]{RV21}).

We proceed by proving a two-sided cartesian Yoneda Lemma in~\Cref{sec:2s-yon} and conclude by briefly discussing discrete two-sided cartesian families in~\Cref{sec:2s-disc}, namely how they can be understood as discrete objects in the $\infty$-cosmos of two-sided cartesian fibrations.

Helpful technical results about several fibered constructions and notions are provided in the Appendix. We treat fibered equivalences in~\Cref{sec:fib-equiv} and fibered (LARI) adjunctions in~\Cref{sec:fib-lari-adj}. In~\Cref{sec:sl-comma-prod} we conclude with preservation properties for sliced commas and products.

\subsection{Contribution}
We develop a theory of two-sided cartesian fibrations~\cite{StrYon,StrBicat,StrBicatCorr,RVyoneda,RV21} of synthetic $\inftyone$-category theory in simplicial homotopy type theory~\cite{RS17}. The central results are a characterization theorem (\Cref{thm:char-two-sid}), a Yoneda Lemma (\Cref{thm:dep-yon-2s,thm:abs-yon-2s}), and~several closure properties (\Cref{thm:2scart-cosm-closure}). Along the way, we also investigate other notions like sliced co-/cartesian families or cocartesian-on-the-left families.

In doing so we work completely internally in simplicial HoTT, so all the results externalize to theorems about internal $\inftyone$-categories in an arbitrary given Grothendieck--Rezk--Lurie $\inftyone$-topos.

Specifically, our work in this type-theoretic setting generalizes previous developments from~\cite{RS17} and~\cite{BW21} following ideas, concepts, and techniques from~\cite{RVyoneda,RV21} and translating into the type theory at hand.

The work presented here appears as~\cite[Chapter~5]{jw-phd}. This PhD thesis was written at TU Darmstadt under the supervision of Thomas Streicher.

\subsection{Related work}

Two-sided cartesian fibrations internal to $2$- and bicategories have been introduced by Street~\cite{StrYon,StrBicat,StrBicatCorr}, \cf~also the review in~\cite{LorRieCatFib}. They have been later generalized to the $\infty$-cosmological setting by Riehl--Verity~\cite{RVyoneda,RV21}. Our work is an adaption of the latter, in particular \cite[Chapter~7]{RV21}, to the setting of simplicial homotopy type theory~\cite{RS17}. This generalizes previous work on (one-sided) co-/cartesian fibrations~\cite{BW21} in this setting. Moreover, there exist treatments of lextensive~\cite{Weinberger2022Sums} and exponentiable \aka~Conduch\'{e} \aka~flat fibrations~\cite{martinez2022limits} in this setting.

Recently, large parts of Riehl--Shulman's initial paper~\cite{RS17}, including the $\infty$-categorical Yoneda Lemma~\cite{KRW-Yon} have been formalized~\cite{KRW-Yon} in Kudasov's new proof assistant \textsc{Rzk}~\cite{Kudasov2023rzk}. There have been various follow-up formalizations in \textsc{Rzk}~\cite{sHoTT}. These currently do not include the work at hand about two-sided cartesian (non-discrete) families but we this development to be possible depending some more groundwork formalizations.

By the semantics of simplicial HoTT~\cite{RS17,Shu19,Wei-StrExt} in $\inftyone$-toposes of the form $\mathscr E^{{\Simplex}^{\Op}}$ for an $\inftyone$-topos $\mathscr E$, the synthetic $\inftyone$-categories in our theory are interpreted as internal $\inftyone$-categories in $\mathscr E$, \ie~Rezk objects. Those have been investigated in various works and semantic settings, notably~\cite{RVyoneda,RV21,Ras17Cart,dB16segal,SteSegObj,mar-yon,mw-lim,martini2022cocartesian,martini2023internal}.

Directed univalence for a universe of synthetic left fibrations has been shown by~\cite{CRS18}. In a (bi-)cubical analogue to sHoTT, there is work by Weaver--Licata~\cite{WL19} on directed univalence for covariant discrete fibrations. There is also recent work on co-/limits~\cite{martinez2022limits} for (complete) Segal types in simplicial HoTT by Bardomiano Mart\'{i}nez.

In HoTT, univalent categories and the Rezk completion have been considered in~\cite{AKS-Rezk,BNMT-Univ} and~\cite[Section~9.9]{hottbook}. In two-level type theory (complete) (semi-)Segal types have been treated in~\cite{CK-SemiSeg,CapriottiPhD,2ltt}. Directed homotopy type theory with semantics in \eg~categories has been discussed by North~\cite{NorthDHoTT}, and \emph{bicategorical} directed type theory has been introduced by Ahrens--North--van~der~Weide~\cite{BTT}. Earlier, directed variants of type theories have been considered by Warren~\cite{War-DTT-IAS} and Licata--Harper~\cite{LH2DTT}. Synthetic notions of variance has been studied by Nuyts~\cite{NuyMSc}.

Perspectives on directed type theories in connections with synthetic higher categories and/or fibrations are given in~\cite{B19,KavQuant}.

\section{Synthetic fibered \texorpdfstring{$\inftyone$}{(∞,1)}-category theory: A quick starting guide}\label{sec:intro-fib-syn}

\subsection{Homotopy type theory}

The basic objects in dependent type theory are \emph{families} $A$ of \emph{types}, depending on \emph{contexts} $\Gamma$, indicated by the judgment $\Gamma \vdash A$. A context $\Gamma$ is a finite list of variable declarations. So if $\Gamma \jdeq [x_1 : A_1, \ldots, x_n : A_n]$, then $\Gamma \vdash A$ means that $A \defeq A(x_1, \ldots, x_n)$ is a family of types depending on the free variables $x_1 : A_1, \ldots, x_n : A_n$.

If $\Gamma \jdeq [ \, \, \,]$ is the \emph{empty} context then a dependent type $\Gamma \vdash A$ is just a single, constant type $A$. A type $A$ is inhabited by \emph{terms} or \emph{elements} $a : A$.

If $\Gamma \vdash A$, we can consider the \emph{extended context}
\[ \Gamma.A \jdeq [x_1 : A_1, \ldots, x_n : A_n, y : A(x_1,\ldots,x_n)].\]
Given families $\Gamma \vdash A$ and $\Gamma.A \vdash B$, there are two important new constructions. We can form \emph{$\Sigma$-types} or \emph{dependent sum types} $\Gamma \vdash \sum_{a:A} B(a)$ whose elements are pairs $\pair{a}{b}$ with $a:A$ and $b:B(a)$. If $A$ is a constant dependent type, then $\sum_{a:A} B(a)$ coincides with the cartesian product type $A \times B$.

Again, given dependent types $\Gamma \vdash A$ and $\Gamma.A \vdash B$, we can form the \emph{$\Pi$-type} or \emph{dependent function type} $\Gamma \vdash \prod_{a:A} B(a)$. The elements are given by dependent functions, or \emph{sections}, $x : \Gamma.A \vdash f(x) : B(x)$, which can be thought of functions with ``varying codomain.'' In case $\Gamma$ is the empty context, the $\Pi$-type $\prod_{a:A} B(a)$ coincides with the ordinary function type $A \to B$.

Given a dependent term $x : \Gamma.A \vdash f(x) : B(x)$ we denote the corresponding dependent function via the \emph{$\lambda$-constructor} $\Gamma \vdash \lambda x.f(x) : \prod_{a:A} B(a)$.

A crucial ingredient for homotopy type theory are Martin-Löf's identity types: given a type $A$ and terms $x,y:A$ one can consider the \emph{identity type} or \emph{path type} $(x=_Ay)$ (often just abbreviated as $(x=y)$) whose terms are witnesses or paths identifying the points $x$ and $y$. There always exists the canonical reflexivity path $\refl_x : (x=_A x)$ for $x:A$, and the identity types satisfy an induction principle (akin to the Yoneda Lemma) stating that the reflexivity terms inductively generate the family $x : A, y : A \vdash (x=_A y)$.

Two terms $x,y : A$ being \emph{judgmentally equal}, $x \jdeq y$, entails that there exists an identification $p : (x=_A y)$, but in an \emph{intensional} type theory (such as HoTT) the converse is not true.

We want to assume a (hierarchy of) universe(s) $\UU$ of all small types, implementing the idea of Grothendieck universes, well-known from category theory inside set theory. Though not material for most of our work, it is convenient to assume Voevodsky's \emph{univalence axiom} stating that in a canonical way the paths of the universes, as captured by its identity type, coincide with the weak equivalences, \ie, bi-invertible functions between types.

Univalence yields a straightening-unstraightening construction on the level of types, see~\Cref{thm:typal-gconstr}. Namely, we can understand a dependent type $\Gamma \vdash A$ either as family $A : \Gamma \to \UU$, where $\Gamma \jdeq A_1 \times \ldots \times A_n$ or as a ``fibration'' $\sum_{a:A} B(a) \to A$.

Indeed, by path induction, one can show that any type family $P : A \to \UU$ behaves like a fibration in the topological sense in that it supports a \emph{path transport} operation $(-)_!^P : \prod_{a,b:A} \prod_{u:(a=_Ab)} P(a) \to P(b)$. For a path $u : (a=_A b)$ and points $x : P(a)$, $y : P(b)$, the path type $(u_!^P(x) =_{P(b)} y)$ gets identified with the type $(x =^P_u y)$ (short: $(x=_uy)$) of \emph{paths over $u$}, \ie, paths $\varphi : (\pair{a}{x} =_{\sum_{x:A} B(x)} \pair{b}{y})$ such that $\pr_1(\varphi) = u$. In this way, we can think of type families also supporting \emph{path lifting}, in line with the topological intution.

In this sense the family $x,y :A \vdash (x=_A y)$ can be identified with a \emph{path space fibration} $\mathrm{Path}(A) \to A \times A$ in a model category~\cite{AW05}.

More background on homotopy type theory can be found in \eg~\cite{hottbook,RijIntro,RieHoTT,AwoTT,Grayson-UF}.

\subsection{Simplicial homotopy type theory}

Simplicial (homotopy) type theory (s(Ho)TT) has been introduced by Riehl--Shulman~\cite{RS17} as an extension of Martin--Löf type theory (MLTT) or homotopy type theory (HoTT), resp., to reason synthetically about $\inftyone$-categories. Building on this, there has been further work on directed univalence~\cite{CRS18,WL19}, fibered $\inftyone$-category theory~\cite{BW21,jw-phd}, and co-/limits~\cite{martinez2022limits}.

We give a short overview on how sHoTT can be presented, and on the basic concepts we will need in this theory. For a longer exposition of the basics of this theory we refer to the original paper~\cite{RS17}, and for a detailed discussion of the fibrational prerequisites we refer to~\cite{BW21}.

\subsubsection{Strict shapes}
In usual dependent type theories, a type family $A$ depends on a \emph{context}
\[ \Gamma \defeq [x_1 : A_1, \ldots, x_n : A_n] \vdash A.\]
This gets extended in simplicial type theory~\cite{RS17} where now types can depend on \emph{shapes} that are defined as sub-polytopes of finite-dimensional cubes. One writes
\[ \Xi \; | \; \Phi \; | \; \Gamma \vdash A \]
where $\Xi$ denotes a \emph{cube}, and $\Phi$ a (sub-)\emph{shape} of $\Xi$.

A general cube is of the form $\I^n$, where $\I^1 \jdeq \I$ denotes the directed bipointed linear order
\[ (0 \le 1).\]
From this, one can define common shapes in the expected way, such as \eg:
\begin{align*}
	& \partial \Delta^1 \defeq \{ t:\I \; | \;  t \jdeq 0 \lor t \jdeq 1\}, \quad \Lambda_1^2 \defeq \{ \pair{t}{s} : \I \times \I  \; | \;  s \jdeq 0 \lor t \jdeq 1\}, \\
	& \partial \Delta^2 \defeq \{ \pair{t}{s} : \I \times \I  \; | \;  s \jdeq t \lor s \jdeq 0 \lor t \jdeq 1\}.
\end{align*}

\emph{Strictness} of the cube and shape layer refers to the fact emethat equality of shape terms $\Phi \vdash s \jdeq t$ is reflected as definitional equality in the type level.
m
As explained in~\cite[Section~2.4]{BW21} we also coerce the shapes to be types.

\subsubsection{Extension types}\label{ssec:extn-types}

Let be a shape inclusion $\Phi \subseteq \Psi$ (in some common cube context $\Xi$). We consider a family $P:\Psi \to \UU$ and a \emph{partial} section $a:\prod_{t:\Phi} P(t)$. Then we can form the \emph{extension type}
\[ \exten{\Psi}{P}{\Phi}{a}\]
consisting of all \emph{total} sections $b:\prod_{t:\Psi} P(t)$ such that $t: \Phi \vdash a(t) \jdeq b(t)$:
\[\begin{tikzcd}
	&& P \\
	\Phi && \Psi
	\arrow[hook, from=2-1, to=2-3]
	\arrow[two heads, from=1-3, to=2-3]
	\arrow["a", dotted, from=2-1, to=1-3]
	\arrow["b"', curve={height=12pt}, dotted, from=2-3, to=1-3]
\end{tikzcd}\]
The rules for the extension type are similar to that of the usual $\Pi$-types but invoking the strict extension property, \cf~\cite[Subsection~2.2]{RS17}. The intuition is that $\Phi \hookrightarrow \Psi$ corresponds to a kind of cofibration along which we ask for extensions of a given section of the fibration $\Phi^*P \fibarr \Phi$. This idea originally goes back to unpublished work by Lumsdaine and Shulman, and it has also been used \eg~in cubical type theory~\cite{CCHM2018,OP16} to define path types~\cite{swan2018path}.

The extension types serve to define a notion of absolute and dependent hom-types, and analogous generalizations thereof to different shapes and higher cells. In particular, this gives rise to synthetic notions of $\infty$-groupoids \aka~\emph{discrete types}, pre-$\inftyone$-categories \aka~\emph{Segal types}, and $\inftyone$-categories \aka~\emph{Rezk} or \emph{complete Segal type}. We briefly introduce and explain these notions in~\Cref{ssec:syn-infty}. For a comprehensive treatment we refer to~\cite[Sections~5, 6, and 10]{RS17} and~\cite[Section~2]{BW21}.

In general, we will freely switch back and forth between type families and fibrations, in particular when defining notions of these. \Eg, defining the notion of a cocartesian \emph{family} simultaneously introduces a notion of cocartesian \emph{fibration}, and \emph{vice versa}. For a formal justification, \cf~\cite[Section~2.5]{BW21}. Specifically, for any family we have the following type-theoretic version of the Grothendieck construction \aka~straightening/unstraightening:
\begin{theorem}[Typal Grothendieck constr., cf.~{\protect\cite[Theorem~4.8.3]{hottbook}}]\label{thm:typal-gconstr}
	For the types
	\[ \Fib(\UU) \defeq \sum_{A,B:\UU} A \to B, \qquad \Fam(\UU) \defeq \sum_{B:\UU} (B \to \UU)  \]
	and maps
	\begin{align*}
		\partial_1 & \defeq \lambda A,B,\pi:A \to B.B:\Fib(\UU) \to \UU, \\
		\pr_1 		& \defeq \lambda B,P:B \to \UU.B: \Fam(\UU) \to \UU,
	\end{align*}
	there is a fiberwise quasi-equivalence
	\[\begin{tikzcd}
		{\Fib(\UU)} && {\Fam(\UU)} \\
		& \UU
		\arrow["{\partial_1}"', from=1-1, to=2-2]
		\arrow["{\pr_1}", from=1-3, to=2-2]
		\arrow["\Un", shift left=1, from=1-3, to=1-1]
		\arrow["\St", shift left=2, from=1-1, to=1-3]
	\end{tikzcd}\]
	at stage $B:\UU$ given by a pair
	\[\begin{tikzcd}
		{\UU/B} && {(B\to \UU)}
		\arrow["{\St_B}", shift left=2, from=1-1, to=1-3]
		\arrow["{\Un_B}", shift left=1, from=1-3, to=1-1]
	\end{tikzcd}\]
	with \emph{straightening}
	\[ \St_B(\pi) \defeq \lambda b.\fib_b(\pi)\]
	and \emph{unstraightening}
	\[ \Un_B(P) \defeq \pair{\totalty{P}}{\pi_P}\]
	($\pi_P: \totalty{P} \defeq \sum_{b:B} P\,b \to B$ the total space projection).
\end{theorem}

In particular, this yields a weak equivlence of maps $\Un_B \circ \St_B$ and $\id_{\UU/B}$, so for any map $E \to B$ we obtain a fibered equivalence, which we call the \emph{projection equivalence} of $\pi$:
\[\begin{tikzcd}
	E && {\sum_{b:B} b^*E} \\
	& B
	\arrow["\simeq", from=1-1, to=1-3]
	\arrow["\pi"', from=1-1, to=2-2]
	\arrow["{\pr_1}", from=1-3, to=2-2]
\end{tikzcd}\]

\subsection{Synthetic $\infty$-categories}\label{ssec:syn-infty}

Recall the notion of extension type from~\Cref{ssec:extn-types}. Importantly, this will serve to define a tractable notion of synthetic $\infty$-category.

\begin{definition}[$\hom$-types, \protect{\cite[Definition~5.1]{RS17}}]\label{def:hom}
	Let $A$ be a type with terms $x,y : A$. Then the extension type
	\[ \hom_A(x,y) \defeq \ndexten{\Delta^1}{A}{\partial \Delta^1}{[x,y]} \]
	is called the \emph{hom-type} or \emph{type of directed arrows} from $x$ to $y$.
	We also write
	\[ (x \to_A y) \defeq \hom_A(x,y),\]
	or simply $(x \to y)$ if the underlying type is understood.
\end{definition}

By the rules of extension types, a term $f : \hom_A(x,y)$ is exactly a function $f : \Delta^1 \to A$ such that $f\,0 \jdeq x$ and $f\,1 \jdeq y$.

To treat composites we have to consider diagrams of shape $\Delta^2$.

\begin{definition}[$\hom^2$-types, \protect{\cite[Definition~5.2]{RS17}}]\label{def:segal}
	Let $A$ be a type with terms $x,y,z:A$. Then, assuming we are given $f : (x\to_A y)$, $g : (y \to_A z)$, and $h : (x \to_A z)$ (and leaving the endpoints implicit in the notation), we can consider the type
	\[ \hom_A^2(f,g;h) \defeq \ndexten{\Delta^2}{A}{\partial \Delta^2}{[f,g,h]}. \]
\end{definition}

By~\Cref{def:segal}, a Segal type $A$ admits composition of arrows, uniquely up to homotopy in the following sense: given terms $x,y,z : A$, arrows $f : (x \to_A y)$ and $g : (y \to_A z)$ there is some composite arrow $h : (x \to_A z)$ together with a $2$-simplex $\sigma_{f,g,h} : \hom_A^2(f,g;h)$ witnessing that $h$ is, in fact, a composite of $f$ and $g$; and more so, the type of this data is contractible.

We are now interested in singling out those simplicial types which always admit composition of (composable) arrows, uniquely up to homotopy. These will be our synthetic pre-$\inftyone$-categories.

\begin{definition}[Segal types, \protect{\cite[Definition~5.3]{RS17}}]
	A \emph{Segal type} is a type $A$ such that
	\[ \isSegal(A) \defeq \prod_{x,y,z:A} \prod_{\substack{f:(x \to_A y) \\  g:(y \to_A z)}} \isContr \Big( \sum_{h:\hom_A(x,z)} \hom_A^2(f,g,h) \Big).\]
\end{definition}
One can show~\cite[Theorem~5.5]{RS17} that this is equivalent to being homotopically right orthogonal to the horn inclusion $\Lambda_1^2 \hookrightarrow \Delta^2$, see also~\Cref{def:orth-fam,def:isoinn-fam}.

Based on the $\hom$-types we can define a notion of \emph{isomorphism}.

\begin{definition}[Isomorphisms]
	Let $A$ be a type. Let $x,y : A$ and $f : \hom_A(x,y)$ be an arrow. We say $f$ is an \emph{isomorphism} if and only if the proposition
	\[ \isIso(f) \defeq \Big( \sum_{g:hom_A(y,x)} g\circ f = \id_x \Big) \times \Big( \sum_{h:hom_A(y,x)} f\circ h = \id_y \Big)  \]
	is satisfied, where $\id_x \jdeq \lambda t.x : \hom_A(x,x)$ is the \emph{identity} morphism of $x$.

	The \emph{type of isomorphisms} from $x$ to $y$ is then defined as
	\[ \iso_A(x,y) \defeq (x \cong_A y) \defeq \sum_{f:\hom_A(x,y)} \isIso(f).\]
\end{definition}

In particular, the identity maps can be shown to be isomorphisms. By path induction, we then obtain a comparison map
\[ \idtoiso_{A} : \prod_{x,y:A} (x =_A y) \to (x \cong_A y), \quad \idtoiso_{A,x,x}(\refl_x) \defeq \id_x.\]
This connects the new notion of isomorphism with the natively given notion of equivalence through paths. 

\subsection{Right orthogonal and LARI families}

We will be concerned in our treatment with specific notions of fibrations, based on synthetic analogues of co-/cartesian fibrations, and occasionally also (discrete) left/right fibrations. However, for both conceptual and technical purposes it will be instructive to introduce the general patterns behind those definitions: right orthogonality and LARI/RARI conditions.

We remark that the notion of iso-/inner family is not so much of conceptual interest, but rather of technical convenience, \cf~also the discussion at the beginning of~\cite[Section~4]{BW21}.

\subsubsection{Directed arrows in type families}

Let $P : B \to \UU$ be a type family with $a,b:B$ and $u: (a \to_B b)$. After~\cite{RS17}, we can then consider the type \emph{dependent homomorphisms} or \emph{dependent arrows}. For $x:P\,a$ and $y : P\,b$, the type of (dependent) arrows in $P$ over $u$ is given by
\[ \mathrm{dhom}_{P,a,b,u}(x,y) \defeq (x \to_u^P y) \defeq (x \to_u y) \defeq \exten{t:\Delta^1}{P(t)}{\partial \Delta^1}{[x,y]}.\]
Analogously, we can define dependent versions of $\hom^2$ and any other possible shape inclusion. This naturally leads to relative versions of the Segal and Rezk conditions, which we will take as the basis of all of our following theory: we will consider isoinner fibrations/families over a Rezk base type. This means, the base type behaves like a category, and so will the total type as well as all the fibers, compatible with the canonical projection from the total type to the base. But these families need not satisfy lifting properties w.r.t~directed arrows (though they do w.r.t.~the ordinary path types from standard homotopy type theory, like any type family). The precise setting for (iso)inner families is given in~\cite[Section~4]{BW21}, and we recall it briefly in the following.

\subsubsection{Fibrations and fibered maps}

We are introducing the notation $\pi : E \fibarr B$ for map between if we want to view it as a fibration of some kind. In writing so, we are somewhat imprecise in that it could mean different things: sometimes it just describes a map $\pi : E \fibarr B$ between types, but to be viewed as a fibration in a slice category over $B$, in the model. This is emphasize its role somewhat akin to to that of an (axiomatic) \emph{isofibration} in the sense of $\infty$-cosmos theory~\cite{RV21}. E.g., when considering a square 
\[\begin{tikzcd}
	E & F \\
	B & A
	\arrow["\pi"', two heads, from=1-1, to=2-1]
	\arrow["k", from=2-1, to=2-2]
	\arrow["\xi", two heads, from=1-2, to=2-2]
	\arrow["\varphi", from=1-1, to=1-2]
\end{tikzcd}\]
of maps between types, really all maps occuring could be written as $\fibarr$ (by projection equivalence). But the intuition is that a square like this encodes a fibered map $\pair{k : B \to A}{\varphi : (b :B) \to P(b) \to Q(k(b))}$ (where $P$ and $Q$, resp., denote the straightening of $E$ and $F$, resp.), so we do not regard $k$ and $\varphi$ as fibrations (but rather, as a fibered functor between fibered categories.\footnote{Though we will soon also consider certain fibered functors (between fibrations placed in the vertical direction) that can be regarded as fibrations themselves (placed in the horizontal direction).}

Sometimes $\fibarr$ might specifically mean that we are considering a (co)cartesian fibration or some related notion, which will be made explicit by the context.

\subsubsection{Right orthogonal families}

\begin{definition}[Right orthogonal family, \protect{\cite[Definition~2.1]{BW21}}]\label{def:orth-fam}
	Let $P:B \to \UU$ be a family. Given a type map or shape inclusion $j:Y \to X$, we say that $P$ is \emph{right orthogonal to $j$} or a \emph{(right) $j$-orthogonal family}, written $j \bot P$, if and only if the proposition
 	\[ \prod_{v:X \to B} \prod_{f:(t:Y) \to P(v\,t)} \isContr\Big( \sum_{g:(t:X) \to P(v\,t)} j^*(g) =^P_v f \Big)\]
 	is satisfied.
 	
 	In case $B \simeq \unit$, the family $P$ can be identified with a single type $A$. In this case we say that $A$ is \emph{right orthogonal to $j$} and write $j \bot A$.
\end{definition}

Let $\pi:E \fibarr B$ denote the projection associated to a family $P:B \to \UU$. Then $P$ is $j$-orthogonal if and only if the type of fillers $g$ in any diagram as follows is contractible:
\[\begin{tikzcd}
	\Phi && E \\
	\Psi && B
	\arrow["j"', hook, from=1-1, to=2-1]
	\arrow["v"', from=2-1, to=2-3]
	\arrow["f", from=1-1, to=1-3]
	\arrow["\pi", two heads, from=1-3, to=2-3]
	\arrow["g"{description}, dashed, from=2-1, to=1-3]
\end{tikzcd}\]
In case of a shape inclusion $j: \Phi \hookrightarrow \Psi$ a family $P$ is right orthogonal to $j$ if and only if
\[ \prod_{v:\Phi \to B} \prod_{f:(t:Y) \to P(v\,t)} \isContr\Big( \exten{\Psi}{P}{\Phi}{f}\Big), \]
\cf~\cite[Section~2.4]{BW21}.

\begin{definition}[Iso-/inner family, \protect{\cite[Definitions~4.1.1, 4.2.3]{BW21}}]\label{def:isoinn-fam}
\begin{enumerate}
	\item\label{it:inn-fam} A type family is \emph{inner} if it is right orthogonal to the horn inclusion $\Lambda_1^2 \hookrightarrow \Delta^2$. A \emph{Segal type} is a type that is inner when seen as a family over the $\unit$.
		\item\label{it:isoinn-fam} An inner family is \emph{isoinner} if it is right orthogonal to the terminal projection $\mathbb{E} \to \unit$, where $\mathbb{E}$ denotes the \emph{free bi-invertible arrow}, \cf~\cite[Subsection~4.2.1]{BW21}. A \emph{Rezk type} or \emph{complete Segal type} is a type that is isoinner (again, when seen as a family over $\unit$).
\end{enumerate}
\end{definition}

The Segal types are exactly the synthetic pre-$\inftyone$-categories\footnote{\Cf~\cite{AKS-Rezk} and \cite[Section~9.1]{hottbook} for a general discussion.} in our setting while the Rezk types are precisely the synthetic $\inftyone$-categories.

The condition of being a co- or contravariant \emph{discrete} family (meaning $\infty$-groupoid valued) can be naturally expressed in terms of a right orthogonality condition, as originally worked out in~\cite[Section~8]{RS17}. Over a point this of course specializes to the type being a synthetic $\infty$-groupoid.

\begin{definition}[Co- and contravariant families, \protect{\cite[Definition~8.2]{RS17}}]\label{def:cov-fam}
	A family is \emph{covariant} if it is right orthogonal to the shape inclusion $0: \unit \hookrightarrow \Delta^1$. Dually, it is \emph{contravariant} if it is right orthogonal to the shape inclusion $1: \unit \hookrightarrow \Delta^1$.
	
	Over a point, in either of these cases we call the ensuing type \emph{discrete}.
\end{definition}

\subsubsection{LARI families}

Co-/cartesian families can be characterized in terms of certain adjoint existence criteria. Namely, a map $\pi: E \fibarr B$ is a synthetic cocartesian fibration iff the Leibniz cotensor
\[\begin{tikzcd}
	{E^{\Delta^1}} && {\comma{\pi}{B}}
	\arrow[""{name=0, anchor=center, inner sep=0}, curve={height=12pt}, from=1-1, to=1-3]
	\arrow[""{name=1, anchor=center, inner sep=0}, curve={height=12pt}, dotted, from=1-3, to=1-1]
	\arrow["\dashv"{anchor=center, rotate=-90}, draw=none, from=1, to=0]
\end{tikzcd}\]
has a \emph{left adjoint right inverse (LARI)}, \cf~\cite[Section~11]{RS17} and \cite[Appendix~B.1]{BW21}. If the latter map exists, it exactly functions as the \emph{lifting map}, mapping a pair $\pair{u:b \to_B b'}{e:P\,b}$ to its cocartesian lift $P_!(u,e): e \cocartarr^P_u e'$.

Formally generalizing this condition to any shape inclusion leads to the notion of \emph{LARI family} or \emph{LARI fibration}, \ie, maps $\pi: E\to B$ such that for a certain shape inclusion or type map $j:Y \to X$ the following induced map has a LARI:
\[\begin{tikzcd}
	{E^X} \\
	&& {B^X \times_{B^Y} E^Y} && {E^Y} \\
	&& {B^X} && {B^Y}
	\arrow["{B^j}"{description}, from=3-3, to=3-5]
	\arrow["{\pi^Y}", from=2-5, to=3-5]
	\arrow[from=2-3, to=2-5]
	\arrow[from=2-3, to=3-3]
	\arrow["\lrcorner"{anchor=center, pos=0.125}, draw=none, from=2-3, to=3-5]
	\arrow["{\pi^X}"{description}, curve={height=18pt}, from=1-1, to=3-3]
	\arrow["{E^j}"{description}, curve={height=-24pt}, from=1-1, to=2-5]
	\arrow[""{name=0, anchor=center, inner sep=0}, from=1-1, to=2-3]
	\arrow[""{name=1, anchor=center, inner sep=0}, curve={height=12pt}, dotted, from=2-3, to=1-1]
	\arrow["\dashv"{anchor=center, rotate=-135}, draw=none, from=1, to=0]
\end{tikzcd}\]

LARI adjunctions in simplicial HoTT have been studied in~\cite[Appendix~B]{BW21}, extending the work on adjunctions in simplicial HoTT by Riehl--Shulman~\cite[Section~11]{RS17}. This in turn extends the study of adjunctions in plain HoTT~\cite[Chapter~4]{hottbook}. In the semantic picture, Riehl--Shulman's theory of adjunction can be understood to internalize parts of of Riehl--Verity's theory of homotopy coherent adjunctions~\cite{RVCohAdjMnd}. In particular, they show in \emph{op.~cit.}~that a $2$-categorical adjunction extends to a fully homotopy coherent one. In~\Cref{sec:fib-lari-adj} we extend the setup of~\cite[Appendix~B]{BW21} by a fibered version of LARI adjunction.

\begin{definition}[$j$-LARI family, \cf~\protect{\cite[Cor.~6.3.8]{RV21}, \cite[Def.~3.2.2]{BW21}}]\label{def:lari-fam}
Let $j:Y \to X$ be a shape inclusion or type map. A family $P: B \to \UU$ is a \emph{$j$-LARI family} if for its unstraightening $\pi \defeq \pi_P : E \to B$ the following Leibniz cotensor map has a LARI:
\[\begin{tikzcd}
	{E^X} && {E^Y\times_{B^Y} B^X}
	\arrow[""{name=0, anchor=center, inner sep=0}, "{j \cotens \pi}"', curve={height=12pt}, from=1-1, to=1-3]
	\arrow[""{name=1, anchor=center, inner sep=0}, curve={height=12pt}, dotted, from=1-3, to=1-1]
	\arrow["\dashv"{anchor=center, rotate=-90}, draw=none, from=1, to=0]
\end{tikzcd}\]
\end{definition}

\subsection{Co-/cartesian families}

\subsubsection{Cocartesian families}

We only treat the covariant case. The contravariant case replaces the initial vertex $i_0 : \unit \hookrightarrow \Delta^1$ with the terminal one $i_1: \unit \hookrightarrow \Delta^1$ and left adjoint inverses (LARIs) with right adjoint inverses (RARIs).

\begin{definition}[Cocartesian arrow, \protect{\cite[Def.~5.1.1]{BW21}, \cf~\cite[Def.~5.4.1]{RV21}}]
Let $P:B \to \UU$ be an isoinner family over a Rezk type $B$. An arrow $f:e \to_u^P e'$ over an arrow $u: b \to_B b'$ is \emph{$P$-cocartesian} if and only if
\[\isCocartArr_u^P(f) \defeq \prod_{b'':B} \prod_{v:b' \to b''} \prod_{e'':P\,b''} \prod_{h:e \to^P_{vu} e''} \isContr\Big( \sum_{g: e' \to^P_v e''} h =_{vu}^P gf \Big).\]
\end{definition}

\begin{definition}[Cocartesian family, \protect{\cite[Def.~5.2.2]{BW21}}, \cf~\protect{\cite[Def.~5.4.2]{RV21}}]
	An isoinner family over a Rezk type $B$ is \emph{cocartesian} if and only if for $u:b \to_B b'$ and $e:P\,b$  the type
	\[ \sum_{e':P\,b'} \sum_{f:e \to^P_u e'} \isCocartArr_u^P(f)  \]
	is inhabited.
\end{definition}

Given $P,u,e,e'$ as above, we will denote the subtype of cocartesian arrows as $(e\cocartarr^P_u e') \hookrightarrow (e \to^P_u e')$, and likewise for the actual arrows in the illustrating figures.

One can show that the inhabitedness condition from the above definition implies contractibility. In other words, cocartesian lifts of a given arrow in the base are determined uniquely up to homotopy by their source vertex.

Connecting back with~\Cref{def:lari-fam} we can express being a cocartesian family as a LARI condition.
\begin{theorem}[Chevalley criterion: Cocartesian families via lifting, \protect{\cite[Def.~5.2.6]{BW21}} \cf~\protect{\cite[Prop.~5.2.8(ii)]{RV21}}]\label{thm:cocart-fams-intl-char}
	Let $B$ be a Rezk type, $P: B \to \UU$ be an isoinner family, and denote by $\pi: E \to B$ the associated projection map. The family $P$ is cocartesian if and only if the Leibniz cotensor map $i_0 \cotens \pi: E^{\Delta^1} \to \comma{\pi}{B}$ has a left adjoint right inverse:
	\[\begin{tikzcd}
		{E^{\Delta^1}} & {} \\
		&& {\pi \downarrow B} && E \\
		&& {B^{\Delta^1}} && B
		\arrow[two heads, from=2-3, to=3-3]
		\arrow["{\partial_0}"', from=3-3, to=3-5]
		\arrow[from=2-3, to=2-5]
		\arrow["\pi", two heads, from=2-5, to=3-5]
		\arrow["\lrcorner"{anchor=center, pos=0.125}, draw=none, from=2-3, to=3-5]
		\arrow["{\partial_0}", shift left=2, curve={height=-18pt}, from=1-1, to=2-5]
		\arrow[""{name=0, anchor=center, inner sep=0}, "\chi"', curve={height=12pt}, dashed, from=2-3, to=1-1]
		\arrow[""{name=1, anchor=center, inner sep=0}, "{i_0\widehat{\pitchfork}\pi}"', curve={height=12pt}, from=1-1, to=2-3]
		\arrow["{\pi^{\Delta^1}}"', shift right=2, curve={height=18pt}, two heads, from=1-1, to=3-3]
		\arrow["\dashv"{anchor=center, rotate=-118}, draw=none, from=0, to=1]
	\end{tikzcd}\]
\end{theorem}

In~\Cref{thm:cocart-fams-intl-char} the stipulated LARI acts as the \emph{lifting map} $\chi: \comma{\pi}{B} \to E^{\Delta^1}$ with $\chi(u,e) \defeq P_!(u,e): e \cocartarr^P_u u_!\,e$. It is also possible to state cocartesianness in terms of a \emph{transport map} $\tau: \comma{\pi}{B} \to_B E$, $\tau(u,e) \defeq u_!\,e : P(\partial_1 \,u)$. This is the content of the next theorem.

\begin{theorem}[Cocart.~families via transport, \protect{\cite[Thm.~5.2.7]{BW21}}, \cf~\protect{\cite[Pr~~5.2.8(ii)]{RV21}}]\label{thm:cocartfams-via-transp}
	Let $B$ be a Rezk type, and $P:B \to \UU$ an isoinner family with associated total type projection $\pi:E \to B$.
	
	Then, $P$ is cocartesian if and only if the map
	\[ \iota \defeq \iota_P : E \to \comma{\pi}{B}, \quad \iota \, \pair{b}{e} :\jdeq \pair{\id_b}{e}  \]
	has a fibered left adjoint $\tau \defeq \tau_P: \comma{\pi}{B} \to E$ as indicated in the diagram:
	\[\begin{tikzcd}
		{} && E && {\pi \downarrow B} \\
		\\
		&&& B
		\arrow[""{name=0, anchor=center, inner sep=0}, "\iota"', curve={height=12pt}, from=1-3, to=1-5]
		\arrow["\pi"', two heads, from=1-3, to=3-4]
		\arrow["{\partial_1'}", two heads, from=1-5, to=3-4]
		\arrow[""{name=1, anchor=center, inner sep=0}, "\tau"', curve={height=12pt}, from=1-5, to=1-3]
		\arrow["\dashv"{anchor=center, rotate=-90}, draw=none, from=1, to=0]
	\end{tikzcd}\]
\end{theorem}

The adjoint criteria from~\Cref{thm:cocart-fams-intl-char,thm:cocartfams-via-transp} are of central importance for a variety of reasons. First, they imply a lot of closure properties, \cf~\cite[Corollary~6.3.8, Proposition~6.3.14]{RV21} and \cite[Section~3, Proposition~5.3.17]{BW21}. Next, they are generalized in our present study to fibered and two-sided situations. Finally, from a broader perspective in model-independent higher category theory the adjoint criteria \`{a} la Chevalley, Gray, Street, are taken by Riehl--Verity as ``reference definitions'' for co-/cartesian fibrations internal to an arbitrary $\infty$-cosmos. It turns out that there is also an equivalent definition~\cite[Definition~5.2.1]{RV21} directly on $2$-cells but the adjoint criteria are more fit to readily yield an internal definition.

In fact, recently Riehl--Verity showed that co-/cartesian arrows in the $\infty$-cosmological sense can also be characterized by Chevalley-type criteria, \cf~\cite[Theorem~5.17]{RV21}. This does not play a role in the present text but has been discussed in~\cite[Appendinx~A.2]{jw-phd} within simplicial HoTT.\footnote{Generalizing to arbitrary shape inclusions or type maps again this gives rise to the notion of \emph{$j$-LARI cell}.}

\subsubsection{Cartesian families}

Dually to the cocartesian case and the left adjoint (right inverse) characterizations, we can consider cartesian fibrations and right adjoint (right inverse) characterizations. Similarly, this can be expressed as a (contravariant) lifting condition w.r.t.~\emph{cartesian arrows}. 

Given a cartesian family $P:B \to \UU$ over $u:b \to_B a$, for $e:P\,b$ and $d:P\,a$, we denote by $(d \cartarr^P_u e) \hookrightarrow (d \to^P_u e)$ the subtype of cartesian arrows. The cartesian lift \wrt~$e:P\,a$ is denoted by $P^*(u,e): u^*_P\,e \cartarr^P_u e$, where $u^*_P\,e$ is the cartesian transport of $e$ by $u$.

\subsubsection{Cocartesian functors}

We conclude our overview with a brief consideration of cocartesian functors in this setting. Those also turn out to be characterizable in terms of Chevalley criteria.\footnote{We obtain a formally analogue characterization via Chevalley criteria in the context of $j$-LARI cells and $j$-LARI-cell-preserving functors, \cf~\cite[Theorem~A.2.6]{jw-phd}.}

\begin{definition}[Cocart.~functors, \protect{\cite[Def.~5.3.2]{BW21}}, \cf~\protect{\cite[Def.~5.3.2]{RV21}}]\label{def:cocart-fun}
	If $P$ and $Q$ are cocartesian families, and furthermore the map
	\[ \Phi: \totalty{P} \to \totalty{Q}, \quad \Phi \, b \, e :\jdeq \pair{j(b)}{\varphi_j(e) } \]
	preserves cocartesian arrows, then we call $\Phi$ a \emph{cocartesian map}:\footnote{This is a proposition because being a cocartesian arrow is a proposition.}
	\[ \isCocartFun_{P,Q}(\Phi) :\jdeq \prod_{f:\Delta^1\to \widetilde{P}} \isCocartArr_P(f) \to \isCocartArr_Q(\varphi f).\]
	In particular, if $B$ and $C$ in addition are Segal (or Rezk) types we speak of a \emph{cocartesian functor}.
	
	We define
	\[ \CocartFun_{B,C}(P,Q) :\jdeq \sum_{F:\FibMap_{B,C}(P,Q)} \isCocartFun_{P,Q}(F).\]
	
	Given families $P: B \to \UU$ to $Q: B \to \UU$, a \emph{fibered functor} from $P \to Q$ is a section $\varphi : \prod_{x:B} Px \to Qx$. It is cocartesian if
	\[ \isCocartFun_{P,Q}(\pair{\id_B}{\varphi}) \equiv \prod_{\substack{u:\Delta^1 \to B \\ f: \Delta^1 \to u^*P}} \isCocartArr_P(f) \to \isCocartArr_Q(\varphi f).\]
	
	We define
	\[ \CocartFun_{B}(P,Q) :\jdeq \sum_{\varphi:\prod_B P \to Q} \isCocartFun_{P,Q}(\pair{\id_B}{\varphi}).\]
\end{definition}

\begin{theorem}[Characterizations of cocart.~functors, \protect{\cite[Thm.~5.3.19]{BW21}}, \cf~\protect{\cite[Thm.~5.3.4]{RV21}}]\label{thm:char-cocart-fun}
	Let $A$ and $B$ be Rezk types, and consider cocartesian families $P:B \to \UU$ and $Q:A \to \UU$ with total types $E\defeq \totalty{P}$ and $F\defeq \totalty{F}$, resp.
	
	For a fibered functor $\Phi\defeq \pair{j}{\varphi}$ giving rise to a square
	\[
	\begin{tikzcd}
		F \ar[r, "\varphi"] \ar[d, "\xi" swap, two heads] & E \ar[d, "\pi", two heads] \\
		A \ar[r, "j" swap] & B
	\end{tikzcd}
	\]
	the following are equivalent:
	\begin{enumerate}
		\item The fiberwise map $\Phi$ is a cocartesian functor.
		\item The mate of the induced canonical fibered natural isomorphism is invertible, too:
		\[\begin{tikzcd}
			{F} && {E} & {} & {F} && {E} \\
			{\xi \downarrow A} && {\pi \downarrow B} & {} & {\xi \downarrow A} && {\pi \downarrow B}
			\arrow["{i}"', from=1-1, to=2-1]
			\arrow["{\varphi}", from=1-1, to=1-3]
			\arrow["{i'}", from=1-3, to=2-3]
			\arrow[Rightarrow, "{=}", from=2-1, to=1-3, shorten <=7pt, shorten >=7pt]
			\arrow["{\rightsquigarrow}" description, from=1-4, to=2-4, phantom, no head]
			\arrow["{\kappa}", from=2-5, to=1-5]
			\arrow["{\varphi}", from=1-5, to=1-7]
			\arrow["{\varphi \downarrow j}"', from=2-5, to=2-7]
			\arrow["{\kappa'}"', from=2-7, to=1-7]
			\arrow[Rightarrow, "{=}"', from=2-7, to=1-5, shorten <=7pt, shorten >=7pt]
			\arrow["{\varphi \downarrow j}"', from=2-1, to=2-3]
		\end{tikzcd}\]
		
		\item The mate of the induced canonical natural isomorphism is invertible, too:
		\[\begin{tikzcd}
			{F^{\Delta^1}} && {E^{\Delta^1}} & {} & {F^{\Delta^1}} && {E^{\Delta^1}} \\
			{\xi \downarrow A} && {\pi \downarrow B} & {} & {\xi \downarrow A} && {\pi \downarrow B}
			\arrow["{r}"', from=1-1, to=2-1]
			\arrow["{\varphi \downarrow j}"', from=2-1, to=2-3]
			\arrow["{\varphi^{\Delta^1}}", from=1-1, to=1-3]
			\arrow["{r'}", from=1-3, to=2-3]
			\arrow["{\rightsquigarrow}" description, from=1-4, to=2-4, phantom, no head]
			\arrow[Rightarrow, "{=}", from=2-1, to=1-3, shorten <=7pt, shorten >=7pt]
			\arrow["{\ell}", from=2-5, to=1-5]
			\arrow["{\varphi \downarrow j}"', from=2-5, to=2-7]
			\arrow["{\varphi^{\Delta^1}}", from=1-5, to=1-7]
			\arrow["{\ell'}"', from=2-7, to=1-7]
			\arrow[Rightarrow, "{=}"', from=2-7, to=1-5, shorten <=7pt, shorten >=7pt]
		\end{tikzcd}\]
	\end{enumerate}
	
\end{theorem}

\subsubsection{Closure properties of co-/cartesian families and functors}

\begin{proposition}[Cosmological closure properties of cocartesian families, \protect{\cite[Proposition~5.3.17]{BW21}}]\label{prop:cocart-cosm-closure}
	Over Rezk bases, it holds that:
	
	Cocartesian families are closed under composition, dependent products, pullback along arbitrary maps, and cotensoring with maps/shape inclusions. Families corresponding to equivalences or terminal projections are always cocartesian.
	
	Between cocartesian families over Rezk bases, it holds that:
	Cocartesian functors are closed under (both horizontal and vertical) composition, dependent products, pullback, sequential limits,\footnote{all three objectwise limit notions satisfying the expected universal properties \wrt~to cocartesian functors} and Leibniz cotensors.
	
	Fibered equivalences and fibered functors into the identity of $\unit$ are always cocartesian.
\end{proposition}

\section{Sliced cocartesian families}\label{sec:sl-cocart-fams}

The main object of study of this paper are two-sided cartesian fibrations $E \fibarr A \times B$. It turns out to be instructive to analyze this notion by means of \emph{fibered} or \emph{sliced fibrations}, \ie, fibered functors $\varphi$ between fibrations
\[\begin{tikzcd}
	F && E \\
	& B
	\arrow["\varphi", from=1-1, to=1-3]
	\arrow["\xi"', two heads, from=1-1, to=2-2]
	\arrow["\pi", two heads, from=1-3, to=2-2]
\end{tikzcd}\]
where $\varphi$ is regarded itself as a fibration of some kind, over the base $\pi$ (itself a fibration). This is the point of view developed by Riehl--Verity in~\cite[Chapter~7]{RV21}, and we adapt it here. For this purpose, we develop in this section the theory of fibrations fibered or sliced over a fibration. While the development undoubtedly presents itself as technical, it will prove instrumental in our main theorems about the characterization of two-sided cartesian families and their closure properties, see~\Cref{thm:char-two-sid,thm:2scart-cosm-closure}.

\subsection{Sliced constructions}\label{ssec:sl-constr}

We introduce here type-theoretic versions of a few sliced constructions, \cf~\cite[Proposition~1.2.22]{RV21}.

\begin{definition}[Sliced cotensor, \protect{\cite[Proposition~1.2.22(vi)]{RV21}}]
	Let $\pi:E\fibarr B$ be a map, and $X$ be a type or shape. The \emph{sliced exponential (over $B$) of $\pi$ by $X$}  is given by the map $(X \iexp E) \to B$ defined as:
	\[\begin{tikzcd}
		{X \slcotens E} && {E^X} \\
		B && {B^X}
		\arrow[two heads, from=1-1, to=2-1]
		\arrow["{\mathrm{cst}}"' description, from=2-1, to=2-3]
		\arrow[from=1-1, to=1-3]
		\arrow["{\pi^X}", two heads, from=1-3, to=2-3]
		\arrow["\lrcorner"{anchor=center, pos=0.125}, draw=none, from=1-1, to=2-3]
	\end{tikzcd}\]
	This means $X \slcotens E \simeq \sum_{b:B} X \to P\,b$.
	In particular, for $X \jdeq \Delta^1$, we obtain the \emph{vertical arrow} type $\VertArr_\pi(E) \to B$. We often identify vertical arrows $f: e \to^P_{\id_b} e'$ with $f: e \to_{P\,b} e'$.
\end{definition}

\begin{definition}[Sliced product, \protect{\cite[Proposition~1.2.22(vi)]{RV21}}]
	Let $I$ and $B$ be types, and consider maps $\pi_i:E_i \fibarr B$ for $i:I$. The \emph{sliced product} over the $\pi_i$ is defined by pullback:
	\[\begin{tikzcd}
		{\times_{i:I}^B E_i} && {\prod_{i:I} E_i} \\
		B && {B^I}
		\arrow["{\times_{i:I}^B \pi_i}"', two heads, from=1-1, to=2-1]
		\arrow[from=2-1, to=2-3, "{\mathrm{cst}}"' description]
		\arrow[from=1-1, to=1-3]
		\arrow["{\prod_{i:I} \pi_i}", two heads, from=1-3, to=2-3]
		\arrow["\lrcorner"{anchor=center, pos=0.125}, draw=none, from=1-1, to=2-3]
	\end{tikzcd}\]
\end{definition}

\begin{definition}[Sliced comma, \protect{\cite[Proposition~1.2.22(vi)]{RV21}}]
	Consider a cospan $\psi:F \to_B G \leftarrow_B E: \varphi$ of fibered functors, giving rise to the sliced comma type $\relcomma{B}{\varphi}{\psi}$:
	\[\begin{tikzcd}
		{\varphi \downarrow_B \psi} && {\VertArr(G)} \\
		{F \times E} && {G \times G} \\
		& B
		\arrow[from=1-1, to=2-1]
		\arrow["{\psi \times \varphi}"{description}, from=2-1, to=2-3]
		\arrow[from=1-1, to=1-3]
		\arrow["{\langle \partial_1,\partial_0\rangle}", from=1-3, to=2-3]
		\arrow[two heads, from=2-1, to=3-2]
		\arrow[two heads, from=2-3, to=3-2]
		\arrow["\lrcorner"{anchor=center, pos=0.125}, draw=none, from=1-1, to=2-3]
	\end{tikzcd}\]
\end{definition}

\subsection{Sliced cocartesian families}\label{ssec:sl-cocart-fams}

\begin{definition}[Sliced cocartesian families]
	Let $B$ be a Rezk type.  A \emph{sliced cocartesian family over $B$} is given by the following data:
	\begin{itemize}
		\item an isoinner family $P: B \to \UU$,
		\item an isoinner family $K: \totalty{P} \to \UU$,
		\item and, writing $Q \defeq \Sigma_{\totalty{P}} K$, a witness for the proposition\footnote{recall the uniqueness of cocartesian lifts, \cf~\cite[Proposition~5.1.3]{BW21}}
		\[ \prod_{b:B} \prod_{\substack{f:\Delta^1 \to P\,b \\ x:Q(b,f0)}} \sum_{\substack{x':Q(b,f1) \\ k:x \to^Q_{b,f} x'}} \isCocartArr^Q_f(k).\]
	\end{itemize}
	We call $K$ a \emph{cocartesian family sliced over $B$ with base $P$}, and denote the ensuing cocartesian lifts as
	\[ K_!^b(f,x) \defeq K_!(b,f,x):x \cocartarr^K_{\pair{b}{f}} f_!^{K,b}\,x.\]
\end{definition}

Perhaps more familiarly, in fibrational terms, a \emph{cocartesian fibration sliced over $B$ with base $\pi$} is given by a fibered functor $\varphi: \xi \to_B \pi$, where $\xi: F \fibarr B$ and $\pi:E \fibarr B$ are isoinner fibrations over $B$, visualized through
\[\begin{tikzcd}
	F && E \\
	& B
	\arrow["\varphi", from=1-1, to=1-3]
	\arrow["\xi"', two heads, from=1-1, to=2-2]
	\arrow["\pi", two heads, from=1-3, to=2-2]
\end{tikzcd}\]
moreover satisfying the analogous lifting property: any $\pi$-vertical arrow has a $\varphi$-cocartesian lift.

As previously with ordinary cocartesian families, we will often bring in the fibrational viewpoint and reason diagrammatically.

This is a type-theoretic formulation of what, more generally in $\infty$-cosmos theory, defines for any $\infty$-cosmos $\mathcal K$ and an object $B \in \mathcal K$ a cocartesian family in the slice-$\infty$-cosmos $\mathcal K/B$. This is captured internally by the following theorem, which shows that the above condition precisely amounts to the sliced version of the familiar LARI condition for cocartesian families.


\begin{theorem}[Characterization of sliced cocartesian families]\label{thm:sl-cocart-fam-char}
	Given a Rezk type $B$, let $P:B \to \UU$ and $K: \totalty{P} \to \UU$ be isoinner families. We write $Q \jdeq \Sigma_{\totalty{P}} K:B \to \UU$, and denote
	\[ \pi  \defeq \Un_B(P) :  E \defeq \totalty{P} \fibarr B, \quad \xi \defeq \Un_B(Q) :  F \defeq \totalty{Q} \fibarr B, \]
	\[  \varphi \defeq \Un_E(K): F \to E  :\]
	\[\begin{tikzcd}
		F && E && {\widetilde{Q} \simeq \widetilde{K}} && {\widetilde{P}} \\
		& B & {} && {} & B
		\arrow["\varphi", from=1-1, to=1-3]
		\arrow["\xi"', two heads, from=1-1, to=2-2]
		\arrow["\pi", two heads, from=1-3, to=2-2]
		\arrow["{\pi_K}", from=1-5, to=1-7]
		\arrow["{\pi_Q \jdeq \pi_{\Sigma_{\totalty{P}}\,K}}"{description}, two heads, from=1-5, to=2-6]
		\arrow["{\pi_P}", two heads, from=1-7, to=2-6]
		\arrow[squiggly, tail reversed, from=1-3, to=1-5]
	\end{tikzcd}\]
	Then the following are equivalent propositions:
	\begin{enumerate}
		\item\label{it:sl-cocart-fam-char-i} The family $K: E \to \UU$ is a cocartesian family sliced over $B$.
		\item\label{it:sl-cocart-fam-char-ii} The sliced Leibniz cotensor~$i_0 \cotens_B \varphi: \VertArr_{\xi}(F) \to_B \relcomma{B}{\varphi}{E}$ has a fibered LARI:
		\[\begin{tikzcd}
			{\VertArr_{\xi}(F) } &&  {\varphi \downarrow_B E} \\
			\\
			& B
			\arrow[""{name=0, anchor=center, inner sep=0}, "{i_0 \widehat{\pitchfork}_B \varphi}"', curve={height=6pt}, from=1-1, to=1-3]
			\arrow[two heads, from=1-1, to=3-2]
			\arrow[two heads, from=1-3, to=3-2]
			\arrow[""{name=1, anchor=center, inner sep=0}, "{\chi_B}"', curve={height=6pt}, dashed, from=1-3, to=1-1]
			\arrow["\dashv"{anchor=center, rotate=-90}, draw=none, from=1, to=0]
		\end{tikzcd}\]
		\item\label{it:sl-cocart-fam-char-iii} The fibered inclusion map~$\iota_\varphi: F \to_E \relcomma{B}{\varphi}{E}$ has a fibered left adjoint:
		\[\begin{tikzcd}
			F &&&& {\varphi \downarrow_B E} \\
			&& E \\
			\\
			&& B
			\arrow["\varphi"{description}, from=1-1, to=2-3]
			\arrow["{\partial_1}"{description}, from=1-5, to=2-3]
			\arrow["\pi"{description}, two heads, from=2-3, to=4-3]
			\arrow["\xi"{description}, two heads, from=1-1, to=4-3]
			\arrow["{\partial_1'}"{description}, two heads, from=1-5, to=4-3]
			\arrow[""{name=0, anchor=center, inner sep=0}, "{\iota_\varphi}"{description}, from=1-1, to=1-5]
			\arrow[""{name=1, anchor=center, inner sep=0}, "{\tau_\varphi}"{description}, curve={height=18pt}, dashed, from=1-5, to=1-1]
			\arrow["\dashv"{anchor=center, rotate=-91}, draw=none, from=1, to=0]
		\end{tikzcd}\]
	\end{enumerate}
	
\end{theorem}

\begin{proof}
	\begin{description}
		\item[$\ref{it:sl-cocart-fam-char-ii} \implies \ref{it:sl-cocart-fam-char-i}$]
		We abbreviate $r\defeq i_0 \widehat{\pitchfork}_B \varphi$. After the usual projection equivalence, we can identify it as the fiberwise map with components
		\[ r_b \big( f:e \to_{P\,b} e', k:x \to^Q_f x' \big) \defeq \langle e, f, x \rangle \jdeq \langle \partial_0 \,f, f, \partial_0 k \rangle \]
		for $b:B$.
		
		Assume, the stated fibered LARI condition is satisfied.
		First we note that the invertible unit, for every $b:B$, exhibits $\chi_{B,b}$ as a (strict) section of $r_b$, ~~given $e:P\,b$, $f:\comma{e}{P\,b}$ and $x:Q(b,e)$, we can assume $\chi_{B,b}(e,f,x):x \to^Q_f x'$ for some $x':Q(b,\partial_1\,f)$.
		We have a fibered equivalence
		\begin{align*}  & \prod_{\substack{b,b':B \\ u:b \to_B b'}} \prod_{\substack{e,e':P\,b \\ f:e \to_{P\,b} e' \\ x:Q(b,e)}} \prod_{\substack{d,d':P(b') \\ g:d \to_{P\,b'} d'}}
			\prod_{\substack{y:Q(b',d), y':Q(b',d') \\ m:y \to^Q_g y'}} \cdots \\
			& \cdots ( \chi_B(e,f,x) \to_u \langle g,m \rangle) \stackrel{\simeq}{\longrightarrow} ( \langle e,f,x \rangle \to_u \langle d, g,y \rangle).
		\end{align*}
		Just as in the second part of the proof of \Cref{thm:cocart-fams-intl-char}, by specializing to the case that $m= \id_{x''}$ and $g=\id_{e''}$ for some $e'':P(b')$ and $x'':Q(b',e'')$, we find that the lift $\chi_{B,b}(e,f,x)$ is a $K$-cocartesian lift of the \emph{$P$-vertical} arrow $f:e \to_{P\,b} e'$.
		
		\item[$\ref{it:sl-cocart-fam-char-i} \implies \ref{it:sl-cocart-fam-char-ii}$]
		On the other hand, suppose that $K$-cocartesian lifts of all $P$-cocartesian maps exist, \wrt~to a given initial vertex. Accordingly, we define $\chi_{B,b}(e,f,x)\defeq K^b_!(f,x) : x \cocartarr_f^K x'$. Again, analogously to the first part of the proof, we define a pair of maps
		\[\begin{tikzcd}
			{\left( \chi_{B,b}(e,f,x) \to_u \langle g,m\rangle\right)} && {\left( \langle e,f,x\rangle \to_u \langle d,g,y\rangle\right)}
			\arrow[""{name=0, anchor=center, inner sep=0}, "\Phi", shift left=2, from=1-1, to=1-3]
			\arrow[""{name=1, anchor=center, inner sep=0}, "\Psi", shift left=2, from=1-3, to=1-1]
			\arrow["\simeq"{description}, shorten <=1pt, shorten >=1pt, Rightarrow, from=0, to=1]
		\end{tikzcd}\]
		by
		\begin{align*}
			& \Phi(h:e \to_u^P d, h': e' \to^P_u d', k: x \to^K_h y, k': x' \to^K_{h'} y') \defeq \langle h, h', k \rangle  \\
			& \Psi(h:e \to_u^P d, h': e' \to^P_u d', k:x \to_h^K y) \defeq \langle h, h', k, \tyfill^K_{\chi_{B,b}(e,f,x)}(m \circ k) \rangle.
		\end{align*}
		Due to cocartesianness of $\chi_{B,b}(e,f,x)$ these are quasi-inverse to each other. In particular, the components of $\Phi$ are defined by applying the right adjoint $i_0 \cotens_B \varphi$. For the unit of the adjunction, we take reflexivity, and taken together this defines a fibered LARI adjunction.
		
		\item[$\ref{it:sl-cocart-fam-char-i} \implies \ref{it:sl-cocart-fam-char-iii}$] The fiberwise map $\iota_\varphi: F \to_E \relcomma{B}{\varphi}{E}$ is given by
		\[ \iota_\varphi(b,e,x)\defeq \angled{b,e,e,\id_e,x}. \]
		Because of the preconditions we can define the candidate fibered left adjoint $\tau_\varphi: \relcomma{B}{\varphi}{E} \to_E F$ by
		\[ \tau_\varphi(b,e',e,f:e \to_{P\,b} e', x:Q(b,e)) \defeq \angled{b,e',f^Q_!(x):Q(b,e')}, \]
		as we would expect analogously to~\Cref{thm:cocartfams-via-transp}.
		To obtain a fibered adjunction as desired, recalling~\Cref{thm:char-fib-adj},~\Cref{it:fib-ladj-sliced}, we want to define a family of equivalences
		\[ \Phi: \prod_{\substack{b:B \\ e':P\,b}} \prod_{\substack{e:P\,b, x:K\,b\,e\\ f:e \to_{P\,b} e'}} \prod_{x':K\,b\,e'} \Big( \underbrace{\tau_\varphi(e,f,x)}_{\jdeq f_!x} \longrightarrow_{K(b,e')} x' \Big) \stackrel{\simeq}{\to} \Big( x \longrightarrow^{Q}_{b,f} x'\Big) : \Psi, \]
		generalizing~\Cref{thm:cocartfams-via-transp}, by\footnote{Note that for the codomain of the equivalence we have identified the type of morphisms $\Big(\angled{e,f,x} \longrightarrow \angled{e',\id_{e'}, x'} \Big)$ (in the fiber $\sum_{g:\comma{P\,b}{e'}} Q(b,\partial_0\,g)$) with simply $(x \to^{b^*Q}_f x')$.}
		\begin{align*} \Phi_{b,e'}\big( m:f_!\,x \to_{Q\,b\,e'} x'\big) & \defeq m \circ Q^b_!(f,e), \\ \Psi_{b,e'}\big( k:x \to_f^{b^*Q} x\big) & \defeq \tyfill_{Q^b_!(f,x)}^\varphi(k),
		\end{align*}
		see~\Cref{fig:sl-cocart-transp} for an illustration. By \emph{$Q$-cocartesianness} of the lifts of $P$-vertical arrows in $K$ one can show---analogously to the proof of~\Cref{thm:cocartfams-via-transp}---that the maps are quasi-inverse to one another.
		\item[$\ref{it:sl-cocart-fam-char-iii} \implies \ref{it:sl-cocart-fam-char-i}$] By assumption, there exists a fibered functor $\tau_\varphi: \relcomma{B}{\varphi}{E} \to_E F$ and a fibered natural transformation
		\[ \eta:\big( \id_{\relcomma{B}{\varphi}{E}}\Rightarrow^K_E \iota_\varphi \circ \tau_\varphi \big) \simeq \prod_{\substack{b:B \\e:P\,b}} \prod_{\substack{d:P\,b \\ f:d \to_{P\,b}e \\ x:Q(b,d)}} \angled{d,f,x} \longrightarrow^{\relcomma{B}{K}{E}}_{\pair{b}{e}} \angled{e,\id_e,f_!x},  \]
		where we write $f_!x$ for the respective component, for the sake of foreshadowing.
		Here,
		\[ \relcomma{B}{K}{E} \defeq \lambda b,e.\sum_{\substack{d:P\,b \\ f:d \to_{P\,b}e}} Q(b,d):E \to \UU \]
		is the straightening of $\partial_1: \relcomma{B}{\varphi}{E} \fibarr E$. Note that there is an equivalence
		\[ \Big( \angled{d,f,x} \to_{(\relcomma{B}{K}{E})(b,e)} \angled{e,\id_e,f_!\,x}  \Big) \simeq \Big( x \to^{b^*Q}_{\pair{b}{f}} f_!\,x\Big) ,\]
		as illustrated by:
		\[\begin{tikzcd}
			F & x && {f_!\,x} \\
			& d && e \\
			E & e && e \\
			B && b
			\arrow["f"', from=2-2, to=3-2]
			\arrow[Rightarrow, no head, from=3-2, to=3-4]
			\arrow["f", from=2-2, to=2-4]
			\arrow[Rightarrow, no head, from=2-4, to=3-4]
			\arrow["{\eta_x}", from=1-2, to=1-4, cocart]
			\arrow[Rightarrow, dotted, no head, from=1-2, to=2-2]
			\arrow[Rightarrow, dotted, no head, from=1-4, to=2-4]
			\arrow[Rightarrow, dotted, no head, from=3-2, to=4-3]
			\arrow[Rightarrow, dotted, no head, from=3-4, to=4-3]
			\arrow[two heads, from=3-1, to=4-1]
			\arrow[two heads, from=1-1, to=3-1]
		\end{tikzcd}\]
		Furthermore, by the assumption, the induced transposing map\footnote{again, identifying $\angled{d,f,x} \longrightarrow_{\pair{v}{g}} \angled{e',\id_{e'}, x'}$ with $x \to^K_{\pair{v}{gf}} x'$} is a family of equivalences:
		\begin{align*}
			& \Phi : \prod_{\substack{b,b':B \\e:P\,b \\ e':P\,b'}} \prod_{\substack{v:b \to_B b' \\ g:e \to^P_v e'}} \prod_{\substack{d:P\,b \\ f:d \to_{P\,b} e \\ x:Q(b,d)}} \prod_{x':Q(b',e')} \Big( f_!\,x \longrightarrow_{\pair{v}{g}}^K x' \Big) \stackrel{\simeq}{\to} \Big( x \longrightarrow^K_{\pair{v}{gf}} x' \Big), \\
			& \Phi\big(m: f_!\,x \to^K_{\pair{v}{g}} x' \big) \defeq \big( m \circ \eta_x: x \to_{Q\,b\,e}  f_!\,x \to^K_{\pair{v}{g}} x' \big)
		\end{align*}
		Now, $\Phi$ being a fiberwise equivalence means the proposition
		\[ \prod_{\substack{b,b':B \\e:P\,b \\ e':P\,b'}} \prod_{\substack{v:b \to_B b' \\ g:e \to^P_v e'}} \prod_{\substack{d:P\,b \\ f:d \to_{P\,b} e \\ x:Q(b,d)}} \prod_{x':Q(b',e')} \prod_{k: x \to^K_{\angled{v,gf}} x'} \isContr \Big( \sum_{m:f_!\,x \to^K_{\pair{v}{g}} x'} m \circ \eta_x = k \Big) \]
		is satisfied, cf.~\Cref{fig:sl-vert-fill}. This exhibits $\eta_x: x \to^K_{\pair{b}{f}} \tau_\varphi(f,x)$ as $K$-cocartesian lift of the $P$-vertical arrow $f:d \to_{P\,b} e$ (starting at $x:Q(b,d)$), as claimed.
	\end{description}
\end{proof}

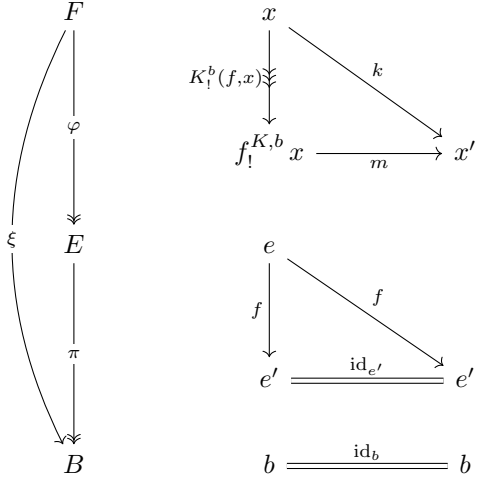
\begin{figure}
	\centering
	\[\begin{tikzcd}
		F && x \\
		\\
		&& {f_!^{K,b}\,x} && {x'} \\
		E && e \\
		\\
		&& {e'} && {e'} \\
		B && b && b
		\arrow["k", from=1-3, to=3-5]
		\arrow["f"', from=4-3, to=6-3]
		\arrow["{\id_{e'}}", Rightarrow, no head, from=6-3, to=6-5]
		\arrow["f", from=4-3, to=6-5]
		\arrow["\pi"{description}, two heads, from=4-1, to=7-1]
		\arrow["\varphi"{description}, two heads, from=1-1, to=4-1]
		\arrow["\xi"{description}, curve={height=30pt}, from=1-1, to=7-1]
		\arrow["{K^b_!(f,x)}"', from=1-3, to=3-3, cocart]
		\arrow["m"', from=3-3, to=3-5]
		\arrow["{\id_b}", Rightarrow, no head, from=7-3, to=7-5]
	\end{tikzcd}\]
	\caption{Cocartesian transport for sliced cocartesian families}
	\label{fig:sl-cocart-transp}
\end{figure}

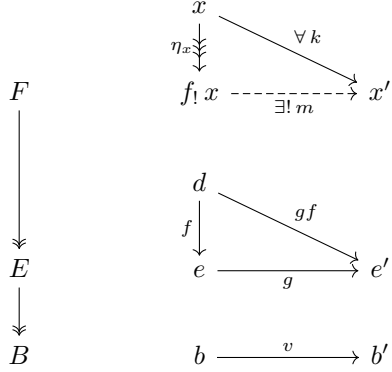
\begin{figure}
	\[\begin{tikzcd}
		&& x \\
		F && {f_!\,x} && {x'} \\
		&& d \\
		E && e && {e'} \\
		B && b && {b'}
		\arrow["{\eta_x}"', from=1-3, to=2-3, cocart]
		\arrow["{\exists!\, m}"', dashed, from=2-3, to=2-5]
		\arrow["{\forall\,k}", from=1-3, to=2-5]
		\arrow["f"', from=3-3, to=4-3]
		\arrow["g"', from=4-3, to=4-5]
		\arrow["gf", from=3-3, to=4-5]
		\arrow["v", from=5-3, to=5-5]
		\arrow[from=4-1, to=5-1, two heads]
		\arrow[from=2-1, to=4-1, two heads]
	\end{tikzcd}\]
	\caption{Cocartesian filling in sliced cocartesian families}
	\label{fig:sl-vert-fill}
\end{figure}

\begin{figure}
	\[\begin{tikzcd}
		x & {x'} & x & {x'} \\
		{g'_!\,x} & {f_!\,x'} &&& {g'_!\,x} & {f_!\,x'} \\
		e & {e'} & e & {e'} & e & {e'} \\
		{e''} & {e'''} & {e''} & {e'''} & {e''} & {e'''} \\
		b & {b'} & b & {b'} & b & {b'} \\
		{\VertArr_\xi(F)} && {\relcomma{B}{\varphi}{E}} && {} & F
		\arrow["v", from=5-3, to=5-4]
		\arrow["{f'}"', from=4-3, to=4-4]
		\arrow["{g'}"', from=3-3, to=4-3]
		\arrow["f", from=3-4, to=4-4]
		\arrow["v", from=5-5, to=5-6]
		\arrow["{g'}"', from=3-5, to=4-5]
		\arrow["{f'}"', from=4-5, to=4-6]
		\arrow["g", from=3-5, to=3-6]
		\arrow["f", from=3-6, to=4-6]
		\arrow["k", from=1-3, to=1-4]
		\arrow["{k'}", dashed, from=2-5, to=2-6]
		\arrow[Rightarrow, dotted, no head, from=1-3, to=3-3]
		\arrow[Rightarrow, dotted, no head, from=1-4, to=3-4]
		\arrow["g", from=3-3, to=3-4]
		\arrow["{\tau_B}", curve={height=-12pt}, maps to, from=3-4, to=3-5]
		\arrow["g", from=3-1, to=3-2]
		\arrow["{g'}"', from=3-1, to=4-1]
		\arrow["f", from=3-2, to=4-2]
		\arrow["{f'}"', from=4-1, to=4-2]
		\arrow["v", from=5-1, to=5-2]
		\arrow["{\chi_B}"', curve={height=12pt}, maps to, from=3-3, to=3-2]
		\arrow["{k'}", dashed, from=2-1, to=2-2]
		\arrow["B", two heads, from=6-3, to=6-1]
		\arrow["{m'}"', from=1-1, to=2-1, cocart]
		\arrow["k", from=1-1, to=1-2]
		\arrow["m", from=1-2, to=2-2, cocart]
		\arrow[curve={height=-18pt}, Rightarrow, dotted, no head, from=2-6, to=4-6]
		\arrow[curve={height=18pt}, Rightarrow, dotted, no head, from=2-5, to=4-5]
		\arrow["E"', two heads, from=6-3, to=6-6]
	\end{tikzcd}\]
	\caption{Action on arrows of lifting and transport of sliced cocartesian families}
	\label{fig:actn-lift-transp-sl-cocart}
\end{figure}
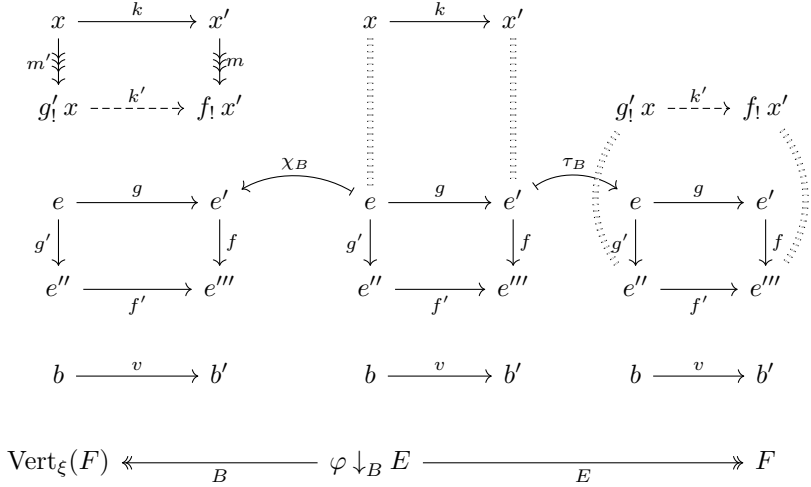

\begin{remark}[Sliced cocartesian families: actions on arrows of the induced functors]\label{rem:sl-cocart-fam-actn-on-arrs}
	For future reference, we record here the actions on arrows of the induced (fibered) lifting and transport functors as established in the proof of~\Cref{thm:sl-cocart-fam-char}. Informally, these can be described as follows. An arrow in $\relcomma{B}{\varphi}{E}$ is given by a dependent square $\sigma$ in $P$ whose vertical edges $f \defeq \lambda t.\sigma(1,t)$, $g' \defeq \lambda t.\sigma(0,t)$ are $P$-vertical, and, moreover, an arrow $k$ in $Q$ over the edge $\pair{t}{0}$. The lifting functor $\chi_B$ maps this to the ensuing dependent square in $Q$ produced by adding the $Q$-cocartesian lifts of the $P$-vertical arrows and the induced filling edge. The functor $\tau_B$ yields only the filling edge of this square.
	
	Formally, this reads as\footnote{suppressing the repeated data in the lower layers}
	{\small \begin{align}
		\chi_{B,v}(\angled{g,f,g',f'}, k:x \to_g x') & \defeq \angled{k,m:x' \cartarr_{f}^Q f_!\,x', m': x \cartarr_{f}^Q g'_!\,x, k': g_!'x \to_{f'}^Q f_!\,x},\label{eq:lift-sl-cocart}\\
		\tau_{B,f'}(\angled{g,f,g',f'}, k:x \to_g x') & \defeq (k':g_!\,x \to_{f'}^Q f_!\,x'). \label{eq:trans-sl-cocart}
	\end{align}}
\end{remark}

The following proposition reflects the known fact that, given a cocartesian functor between cocartesian fibrations, it is a sliced cocartesian fibration if and only if it is a cocartesian fibration in the usual sense.
\begin{proposition}
	Let $B$ be a Rezk type. Assume $\xi:F \fibarr B$ and $\pi:E \fibarr B$ are cocartesian fibrations, and $\varphi:F \to_B E$ is a cocartesian functor:
	\[\begin{tikzcd}
		F && E \\
		& B
		\arrow["\varphi", from=1-1, to=1-3]
		\arrow["\xi"', two heads, from=1-1, to=2-2]
		\arrow["\pi", two heads, from=1-3, to=2-2]
	\end{tikzcd}\]
	Then $\varphi:F \to_B E$ is a cocartesian fibration sliced over $B$ if and only if it is a cocartesian fibration $F \fibarr E$ in the usual sense.
\end{proposition}

\begin{proof}
	In case $\varphi: F \to E$ is a cocartesian fibration it is also a cocartesian fibration sliced over $B$ since it automatically satisfies the weaker existence condition for lifts.
	
	For the converse, we fibrantly replace the given diagram, considering the straightenings $P \defeq \St_B(\pi): B \to \UU$, $Q \defeq \St_B(\pi): B \to \UU$, $K \defeq  \St_E(\varphi): E \to \UU$.
	
	We assume $K$ to be a cocartesian family sliced over $B$, and want to show that it is also a cocartesian family in the usual sense. For an illustration of what follows, \cf~\Cref{fig:abs-from-sliced-cocart-fibs}. Consider an arrow $\pair{u:b \to_B b'}{f:e \to_u^\pi e'}$ in $E$ together with a point $x:K(b,e)\jdeq Q(b,e)$. First, consider the $P$-cocartesian lift of $u:b\to_B b'$ \wrt~$e:P\,b$, given by~$g \defeq P_!(u,e) : e \cocartarr^P_{u} u^P_!\,x$. This induces a filler $h \defeq \tyfill^P_g(f): u^P_!\,x \to_{P\,b'} e'$ that is in particular vertical. Since $K:E \to \UU$ is a sliced cocartesian family we have a lift \wrt~to the \emph{$Q$-cocartesian} transport of the point $x:Q(b,e)$, namely an arrow $m \defeq K_!(h,u_!^Q\,x): u_!^Q\,x \to^K_{\pair{u}{f}} x'$ to some point $x':Q(b,e')$. But by assumption, $m$ (together with its $\varphi$-image $h$) is also a $Q$-cocartesian arrow, hence so is the composite $k \defeq m \circ Q_!(u,x) : x \longrightarrow^K_{g \circ \varphi(Q_!(u,x))} x'$.
	The functor $\varphi$ being cocartesian means $\varphi(Q_!(u,x))$ is identified with $P_!(u,e) \jdeq g$, so up to homotopy, the dependent arrow $k:x \to^Q_u x'$ lies over the composite $h \circ g = f$---hence we can assume it does so strictly. Now, $k$ (together with its projection $f$) being a $Q$-cocartesian arrow means it is in particular $K$-cocartesian (cf.~\Cref{fig:abs-from-sliced-cocart-fibs} for illustration).
\end{proof}

\begin{figure}
	\centering
	\[\begin{tikzcd}
		F & x && {x'} \\
		&&& {u_!^Q\,x} \\
		E & e && {e'} \\
		&&& {u^P_!\,e} \\
		B & b && {b'}
		\arrow["k", from=1-2, to=1-4]
		\arrow["m"', from=2-4, to=1-4]
		\arrow["f", from=3-2, to=3-4]
		\arrow["{g \jdeq P_!(u,e)}"', from=3-2, to=4-4, cocart]
		\arrow["h"', dashed, from=4-4, to=3-4]
		\arrow["u", from=5-2, to=5-4]
		\arrow["\varphi"{description}, from=1-1, to=3-1]
		\arrow["\pi"{description}, two heads, from=3-1, to=5-1]
		\arrow["{\ell \jdeq Q_!(u,x)}"', from=1-2, to=2-4, cocart]
		\arrow["\xi"{description}, curve={height=24pt}, two heads, from=1-1, to=5-1]
	\end{tikzcd}\]
	\caption{Absolute from sliced cocartesian fibrations}\label{fig:abs-from-sliced-cocart-fibs}
\end{figure}
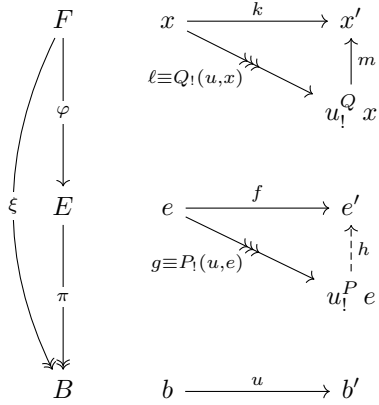

\begin{proposition}\label{prop:sliced-comma-is-cocart}
	Consider a cospan $\psi:F \to_B G \leftarrow_B E: \varphi$ of fibered functors between isoinner fibrations over a Rezk type $B$, giving rise to the sliced comma type:
	
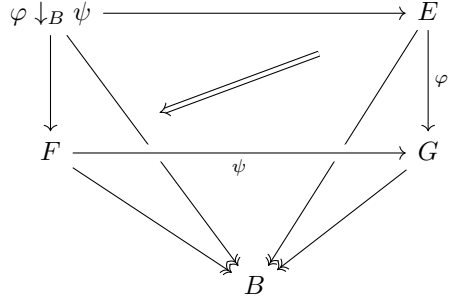
\[\begin{tikzcd}
		{\varphi \downarrow_B \psi} &&&& E \\
		\\
		F &&&& G & {} \\
		\\
		&& B
		\arrow[from=1-1, to=3-1]
		\arrow[two heads, from=3-1, to=5-3]
		\arrow[two heads, from=3-5, to=5-3]
		\arrow[from=1-1, to=1-5]
		\arrow["\varphi", from=1-5, to=3-5]
		\arrow[two heads, from=1-1, to=5-3]
		\arrow[two heads, from=1-5, to=5-3]
		\arrow[shorten <=35pt, shorten >=35pt, Rightarrow, from=1-5, to=3-1]
		\arrow["\psi"', from=3-1, to=3-5, crossing over]
	\end{tikzcd}\]
	The codomain projection
	\[ \partial_1: \relcomma{B}{\varphi}{\psi} \fibarr F \]
	is a cocartesian fibration.
\end{proposition}

\begin{proof}
	Consider the families $R:B \to \UU$, $P:G \to \UU$, and $Q:G \to \UU$ associated to $G \fibarr B$, $\varphi: E \fibarr G$, and $\psi: F \fibarr G$, resp. By projection equivalence we have
	\[ \partial_1: \relcomma{B}{\varphi}{\psi} \simeq \sum_{b:B} \sum_{\stackrel{x,x':R\,b}{k:x \to^R_u x'}} P\,b\,x \times Q\, b\,x' \fibarr F \simeq \sum_{b:B} \sum_{x':R\,b} Q\,b\,x'. \]
	For an arrow in $F$, given by data $u:b \to_B b'$, $m:x' \to^R_u x''$, $f:d \to^Q_m d'$, we posit the cocartesian lift w.r.t.~the starting vertex~$\angled{b,k:x \to_{P\,b} x', d,e}$ to be the ``tautological extension''
	\[ \angled{\pair{\id_x}{m}: k \rightrightarrows^R_{u} mk, \id_e: e =_{P\,x} e,f:d \to^Q_m d'} \]
	as illustrated:
	{\small 
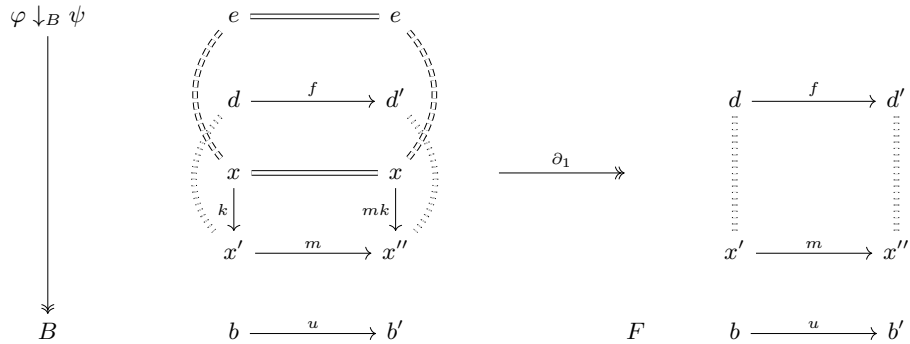
	\[\begin{tikzcd}
		{\varphi \downarrow_B \psi} && e && e \\
		&& d && {d'} &&&& d && {d'} \\
		&& x && x & {} && {} \\
		&& {x'} && {x''} &&&& {x'} && {x''} \\
		B && b && {b'} &&& F & b && {b'}
		\arrow[Rightarrow, no head, from=1-3, to=1-5]
		\arrow["f", from=2-3, to=2-5]
		\arrow["k"', from=3-3, to=4-3]
		\arrow["m", from=4-3, to=4-5]
		\arrow[Rightarrow, no head, from=3-3, to=3-5]
		\arrow["mk", from=3-5, to=4-5, swap]
		\arrow["u", from=5-3, to=5-5]
		\arrow[two heads, from=1-1, to=5-1]
		\arrow["{\partial_1}", two heads, from=3-6, to=3-8]
		\arrow["u", from=5-9, to=5-11]
		\arrow["m", from=4-9, to=4-11]
		\arrow["f", from=2-9, to=2-11]
		\arrow[curve={height=18pt}, Rightarrow, dotted, no head, from=2-3, to=4-3]
		\arrow[curve={height=-18pt}, Rightarrow, dotted, no head, from=2-5, to=4-5]
		\arrow[curve={height=18pt}, Rightarrow, dashed, no head, from=1-3, to=3-3]
		\arrow[curve={height=-18pt}, Rightarrow, dashed, no head, from=1-5, to=3-5]
		\arrow[Rightarrow, dotted, no head, from=2-9, to=4-9]
		\arrow[Rightarrow, dotted, no head, from=2-11, to=4-11]
	\end{tikzcd}\]}
	In the picture the right hand side indicates the action of the projection $\partial_1 : \relcomma{B}{\varphi}{\psi} \fibarr F$
	That this arrow in the sliced comma $\relcomma{B}{\varphi}{\psi}$ is in fact cocartesian is seen as follows. A postcomposing arrow in $F$ consists of data $v:b' \to b''$, $\ell: x'' \to_v^R x'''$, $g:d' \to^Q_\ell d''$. A dependent arrow in $\relcomma{B}{\varphi}{\psi}$ over the composite arrow in $F$ is given by
	\[ \angled{\pair{\ell'}{\ell m}: k \rightrightarrows^R_{vu} k', r:e \to^P_{\ell'} e', g \circ f: d \to^Q_{\ell \circ m} d''}\]
	where $\ell':x \to_v^R y$, $k':y \to_{R\,b''} x'''$:
	\[\begin{tikzcd}
		{\varphi \downarrow_B \psi} && e && e && {e'} \\
		&& d && {d'} && {d''} \\
		&& x && x && y & {} \\
		&& {x'} && {x''} && {x'''} \\
		B && b && {b'} && {b''}
		\arrow[Rightarrow, no head, from=1-3, to=1-5]
		\arrow["f", from=2-3, to=2-5]
		\arrow["k"', from=3-3, to=4-3]
		\arrow["m", from=4-3, to=4-5]
		\arrow[Rightarrow, no head, from=3-3, to=3-5]
		\arrow["mk", from=3-5, to=4-5]
		\arrow["u", from=5-3, to=5-5]
		\arrow[two heads, from=1-1, to=5-1]
		\arrow["{\ell'}", from=3-5, to=3-7]
		\arrow["\ell", from=4-5, to=4-7]
		\arrow["r", from=1-5, to=1-7]
		\arrow["g", from=2-5, to=2-7]
		\arrow["{k'}", from=3-7, to=4-7]
		\arrow["v", from=5-5, to=5-7]
	\end{tikzcd}\]
	The filler is constructed just by repeating the missing data $\ell':x \to y$, $r:e \to e'$ which also shows uniqueness up to contractibility~\wrt~the given data.
\end{proof}

\begin{corollary}
	For a cospan of maps between Rezk types $g: C \to A \leftarrow B:f$, the codomain projection $\partial_1: \comma{f}{g} \fibarr C$ is a cocartesian fibration.
\end{corollary}

\subsection{Cocartesian families in cartesian families}\label{ssec:cocart-fams-in-cart}

We now introduce the concept of \emph{cocartesian fibrations in cartesian fibrations}, which semantically, really translate to cocartesian fibrations in $\infty$-cosmoses of cartesian fibrations.

Their characterization will be of relevance for the characterization of two-sided cartesian fibrations in~\Cref{thm:char-two-sid}.

\begin{definition}[Cocartesian fibrations in cartesian fibrations]\label{def:cocart-fibs-in-cart-fibs}
	Let $\varphi: \xi \to_B \pi$ be a cartesian functor between cartesian fibrations. In addition, let $\varphi$ be a \emph{cocartesian} fibration sliced over $B$:
	\[\begin{tikzcd}
		F && E \\
		& B
		\arrow["\varphi", from=1-1, to=1-3]
		\arrow["\xi"', two heads, from=1-1, to=2-2]
		\arrow["\pi", two heads, from=1-3, to=2-2]
	\end{tikzcd}\]
	Then we call $\varphi$ a \emph{cocartesian fibration in cartesian fibrations}.\footnote{For the official naming we prefer the fibrational variant since it is closer to its semantic counterpart, but of course by the typal Grothendieck construction there exists an indexed variant as well.}
\end{definition}

\begin{proposition}[Char.~of cocartesian fibration in cartesian fibrations]
	\label{prop:char-cocart-fib-in-cart-fib}
	Let $\varphi:\xi \to_B \pi$ be a cocartesian fibration in cartesian fibrations, as in~\Cref{def:cocart-fibs-in-cart-fibs}. Then the following are equivalent:
	\begin{enumerate}
		\item\label{it:gen-lift-comm-i}
		The sliced cocartesian lifting map, ~~~the fibered LARI
		\[\begin{tikzcd}
			{\mathrm{Vert}_\xi(F)} && {\varphi \downarrow_B E } \\
			& B
			\arrow["{\overline{\xi}}"', two heads, from=1-1, to=2-2]
			\arrow["{\overline{\varphi}}", two heads, from=1-3, to=2-2]
			\arrow[""{name=0, anchor=center, inner sep=0}, "{\chi_B}"', curve={height=12pt}, dotted, from=1-3, to=1-1]
			\arrow[""{name=1, anchor=center, inner sep=0}, "{i_0 \cotens \pi}"', curve={height=6pt}, from=1-1, to=1-3]
			\arrow["\dashv"{anchor=center, rotate=-90}, draw=none, from=0, to=1]
		\end{tikzcd}\]
		is a \emph{cartesian functor} between cartesian fibrations over $B$
		\item\label{it:gen-lift-comm-ii} 
		The sliced cocartesian transport map, \ie, the fibered left adjoint
		\[\begin{tikzcd}
			F && {\varphi \downarrow_B E} \\
			& E \\
			& B
			\arrow["\varphi"{description}, from=1-1, to=2-2]
			\arrow[""{name=0, anchor=center, inner sep=0}, "{\iota_B}"{description}, curve={height=6pt}, from=1-1, to=1-3]
			\arrow["{\partial_1}"{description}, from=1-3, to=2-2]
			\arrow["\xi"{description}, two heads, from=1-1, to=3-2]
			\arrow["{\overline{\varphi}}"{description}, two heads, from=1-3, to=3-2]
			\arrow["\pi"{description}, two heads, from=2-2, to=3-2]
			\arrow[""{name=1, anchor=center, inner sep=0}, "{\tau_B}"{description}, curve={height=12pt}, dashed, from=1-3, to=1-1]
			\arrow["\dashv"{anchor=center, rotate=-93}, draw=none, from=1, to=0]
		\end{tikzcd}\]
		is a cartesian functor (from $\overline{\varphi}$ to $\xi$).
		\item\label{it:gen-lift-comm-iii} For all elements $b,b':B$, arrows $v:b' \to_B v$, vertical arrows $f:e' \to_{P\,b}e$ and $x:Q(b,e')$, let us make the following abbreviations:\footnote{Note that all the cocartesian lifts exist because they are over $P$-vertical arrows.}
		{ \begin{align}\label{eq:abbrv-cocart-in-cart-gen}
			g 	& \defeq P^*(v,e'): v^*(b'e') \cartarr^P_{v} e', & f' & \defeq P^*(v,e'): v^*(b,e) \cartarr^P_{v} e, \\
			g ' & \defeq \cartFill_{f'}^P(fg), & k & \defeq Q^*(g,x): x \cartarr^Q_g g^*x, \\
			k' & \defeq \cocartFill^Q_{m'}(mk): g'_!g^*x \to^Q_{f'} f_!\,x, & k'' & \defeq Q^*(f',f_!\,x): (f')^*f_!\,x \cartarr^Q_{f'} f_!x, \\
			m 	& \defeq Q_!(f,x): x \cocartarr^Q_f f_!\,x. & m' & \defeq Q_!(g',g^*x): g^*x \cocartarr^Q_{g'} g_!'g^*x, \\
			m'' & \defeq \cartFill^Q_{k''}(mk): g^*x \to (f')^*f_!\,x. && 
		\end{align}}
		Then there is a homotopy $r$ such that:
		\[\begin{tikzcd}
			{(f')^*f_!\,x} && {f_!\,x} \\
			& {g'_!g^*\,x}
			\arrow["{k'}", from=1-1, to=1-3]
			\arrow["r"', Rightarrow, no head, from=1-1, to=2-2]
			\arrow["{k''}"', from=2-2, to=1-3, cart]
		\end{tikzcd}\]
		\item\label{it:gen-lift-comm-iv} With the notation from~\Cref{it:gen-lift-comm-iii} there is a homotopy $r$ such that:
		\[\begin{tikzcd}
			{g^*x\,} && {g_!'g^*\,x} \\
			& {(f')^*f_!\,x}
			\arrow["{m''}", from=1-1, to=1-3, cocart]
			\arrow["{m'}"', from=1-1, to=2-2]
			\arrow["r"', Rightarrow, no head, from=2-2, to=1-3]
		\end{tikzcd}\]
	\end{enumerate}
	
\end{proposition}

\begin{proof}
	We prove the equivalence of these four conditions by first explicating~\Cref{it:gen-lift-comm-i}. We will readily see that it is equivalent to either of the three remaining condition.
	
	Recall the action of the fibered lifting $\chi_B$ and transport functor $\tau_B$, resp, \Cref{thm:sl-cocart-fam-char,rem:sl-cocart-fam-actn-on-arrs}. In the first case, assume we have an identity
	\begin{align}\label{eq:cart-sl-transp} 
		\chi_B(\overline{\varphi}^*(v,\pair{f}{x})) = \overline{\xi}^*(v,\chi_B(f,x))
	\end{align}
	for all $v:b'\to_B b$, $f:e' \to_{P\,b} e$, $x:Q(b,e')$. Consider the abbreviations from~(\ref{eq:abbrv-cocart-in-cart-gen}).
	Specifically, the case for $\chi_B$ will involve
	\begin{align}\label{eq:abbrv-cocart-in-cart-chi}
		m' & \jdeq Q_!(g',g^*x): g^*x \cocartarr^Q_{g'} g_!'g^*x, & & k' \jdeq \cocartFill^Q_{m'}(mk): g'_!g^*x \to^Q_{f'} f_!\,x,
	\end{align}
	whereas for $\tau_B$ we will need:
	\begin{align}\label{eq:abbrv-cocart-in-tau}
		k'' & \jdeq Q^*(f',f_!\,x): (f')^*f_!\,x \cartarr^Q_{f'} f_!x, & & m'' \jdeq \cartFill^Q_{k''}(mk): g^*x \to (f')^*f_!\,x.	
	\end{align}
	As detailed in~\cite[Subsection~5.2.3]{BW21}, lifts of co-/cartesian families are fiberwise. Hence, we find for the left hand side in~(\ref{eq:cart-sl-transp}):
	\begin{align}
		& \chi_B(\overline{\varphi}^*(v,\pair{f}{x})) = \chi_B(v,\angled{g,f,g',f'},k:g^*x \cartarr_{g}^Q x) \\
		= & \angled{v,\angled{g,f,g',f'},\angled{k,m,m',k'}} \label{eq:abbrv-cocart-in-cart-lhs}
	\end{align}
	and for the right hand side:
	\begin{align}
		& \overline{\xi}^*(v,\chi_B(f,x)) = \overline{\xi}^*(v,\angled{g,f,g',f'},\angled{k,m,m'',k''}) \\
		=	& \angled{v,\angled{g,f,g',f'},\angled{k,m,m'',k''}}. \label{eq:abbrv-cocart-in-cart-rhs}
	\end{align}
	Recall that a functor being cartesian is a proposition.
	A path between (\ref{eq:abbrv-cocart-in-cart-lhs}) and (\ref{eq:abbrv-cocart-in-cart-rhs}) amounts to an isomorphism $r:(f')^*f_!\,x =_{v^*(b,e)} g'_!g^*x$ such that the entire following diagram commutes:
	\[\begin{tikzcd}
		{g^*\,x} && x \\
		{(f')^*f_!\,x} && {f_!\,x} \\
		& {g_!'g^*x}
		\arrow["{m'}"', from=1-1, to=2-1, cocart]
		\arrow["k"{pos=0.6}, from=1-1, to=1-3]
		\arrow["m", from=1-3, to=2-3, cocart]
		\arrow["r"', Rightarrow, no head, from=2-1, to=3-2]
		\arrow["{k''}"', from=3-2, to=2-3, cart]
		\arrow["{k'}"'{pos=0.6}, from=2-1, to=2-3]
		\arrow["{m''}"{pos=0.3}, from=1-1, to=3-2, crossing over]
	\end{tikzcd}\]
	More generally, it can be shown that there exists a filler $r:(f')^*f_!\,x \to_{v^*(b,e)} g_!'g^*x$ s.t.~$m'' = rm'$ and~$k' = k''r$. Hence, this propositional condition is equivalent to this induced arrow being invertible. But moreover, we can see by universality that this is equivalent to the existence of either identification $m'=m''$ or $k'=k''$. In particular, the action by the transport functor $\tau_B$ yields just the latter. Hence, all the four conditions claimed are equivalent.
\end{proof}

\begin{proposition}[Closure of sliced cocartesian fibrations under product]\label{prop:clos-sl-cocart-fib-prod}
	For a small indexing type $I:\UU$, let $B:I \to \UU$ be a family of Rezk types. Let $P:\prod_{i:I}(B_i \to \UU)$ be a family and $K: \prod_{i:I} (\totalty{P}_i \to \UU)$ be another family. We define $Q\defeq \lambda i.\Sigma_{\totalty{P_i}}K_i.\prod_{i:I} (B_i \to \UU)$. For every $i:I$, we denote
	\begin{align*}
		& \pi_i \defeq \Un_{B_i}(P_i) : E_i \defeq \totalty{P_i} \fibarr B_i,& \quad \xi_i \defeq \Un_{B_i}(Q_i) : F_i \defeq \totalty{Q_i} \fibarr B_i,  \\ 
		& \varphi \defeq \Un_{E_i}(K_i): F_i \to E_i &
	\end{align*}
	giving rise to diagrams:
	\[\begin{tikzcd}
		{F_i} && {E_i} \\
		& {B_i} & {} && {}
		\arrow["{\xi_i}"', two heads, from=1-1, to=2-2]
		\arrow["{\pi_i}", two heads, from=1-3, to=2-2]
		\arrow["{\varphi_i}", from=1-1, to=1-3]
	\end{tikzcd}\]
	If each $\varphi_i$ is a sliced cocartesian fibration, then so is the product:
	\[\begin{tikzcd}
		{\prod_{i:I} F_i} && {\prod_{i:I} E_i} \\
		& {\prod_{i:I} B_i} & {} && {}
		\arrow["{\prod_{i:I} \xi_i}"', two heads, from=1-1, to=2-2]
		\arrow["{\prod_{i:I} \pi_i}", two heads, from=1-3, to=2-2]
		\arrow["{\prod_{i:I} \varphi_i}", from=1-1, to=1-3]
	\end{tikzcd}\]
	Moreover, if $\varphi_i$ is a cocartesian fibration in cartesian fibrations in the sense of~\ref{def:cocart-fibs-in-cart-fibs}, then so is $\prod_{i:I} \varphi_i$.
\end{proposition}

\begin{proof}
	Since dependent products commute with sliced commas by~\Cref{prop:dep-prod-comm-sl-commas} we find
	\begin{align}
		\prod_{i:I} \VertArr_{\xi_i}(F_i) & \equiv \prod_{i:I} \relcomma{B_i}{F_i}{F_i} \equiv   \relcomma{\prod_{i:I} B_i}{\big(\prod_{i:I}F_i\big)}{\big(\prod_{i:I}F_i\big)} \\
		\prod_{i:I} \relcomma{B_i}{\varphi_i}{E_i} & \equiv \relcomma{\prod_{i:I} B_i}{\big(\prod_{i:I} \varphi_i\big)}{\big(\prod_{i:I} E_i\big)}.
	\end{align}
	Since sliced LARIs are preserved by the dependent product~\Cref{prop:fib-lari-pres-by-sl-prod} we obtain an induced fibered LARI between these commas, establishing that $\prod_{i:I} \varphi$ is sliced cocartesian by~\Cref{thm:sl-cocart-fam-char}.
	
	Moreover, since cartesian fibrations and co-/cartesian functors are preserved under the dependent product the analogous closure statement for the stronger notion of cocartesian fibration in cartesian fibrations follows readily.
\end{proof}

\begin{lemma}[Pullback of fibered cocartesian sections]\label{lem:cocart-sect-pb}
	For a Rezk type $B$, consider cocartesian families $P:B \to \UU$, $Q:\totalty{P} \to \UU$, and a fiberwise map $\varphi : \prod_{b:B} P\,b \to Q\,b$. We write the unstraightenings as $\pi: E \defeq \totalty{P} \fibarr B$, $\xi: F \defeq \totalty{Q} \fibarr B$. Consider the following diagram, induced by a section $\ell$ of (the totalization of) $\varphi$ over $B$, and $A$ and a map $k:A \to B$ between Rezk types:
	\[\begin{tikzcd}
		{F'} &&& F \\
		& {E'} &&& E \\
		A &&& B
		\arrow["\varphi"{description}, from=1-4, to=2-5]
		\arrow[from=1-1, to=1-4]
		\arrow["{\varphi'}"{description}, from=1-1, to=2-2]
		\arrow["\xi"{description, pos=0.3}, two heads, from=1-4, to=3-4]
		\arrow["\pi"{description}, two heads, from=2-5, to=3-4]
		\arrow["{\xi'}"{description, pos=0.3}, two heads, from=1-1, to=3-1]
		\arrow["{\pi'}"{description}, two heads, from=2-2, to=3-1]
		\arrow["k", from=3-1, to=3-4]
		\arrow["\ell"{description}, curve={height=12pt}, dotted, from=2-5, to=1-4]
		\arrow[from=2-2, to=2-5, crossing over]
		\arrow["\lrcorner"{anchor=center, pos=0.125}, draw=none, from=2-2, to=3-4]
		\arrow["\lrcorner"{anchor=center, pos=0.125}, shift right=5, draw=none, from=1-1, to=3-4]
		\arrow["{\ell'}"{description, pos=0.3}, curve={height=18pt}, dotted, from=2-2, to=1-1, crossing over]
	\end{tikzcd}\]
	If $\ell$ is a cocartesian functor, then the induced section $\ell'$ is, too.
\end{lemma}

\begin{proof}
	First, projection equivalence yields:
	\[E \defeq \sum_{b:B} P\,b, \quad F \defeq \sum_{\substack{b:B \\ e:P\,b}} Q\,b\,e,
	E' \defeq \sum_{a:A} P\,ka, \quad F' \defeq \sum_{\substack{a:B \\ d:P\,ka}} Q\,ka\,d.  \]
	The section $\ell$ is then taken to be
	\[ \ell(b,e) \defeq \angled{b,e,\widehat{\ell}(b,e)}\]
	for $b:B$, $e:P\,b$.
	Cocartesianness means that there is a path
	\begin{align*}
		\ell(u,P_!(u,e)) & \jdeq \angled{u,P_!(u,e),\widehat{\ell}(u,P_!(u,e))} \\
		& \jdeq \angled{u,P_!(u,e),Q_!(u, \pair{u^P_!e}{\widehat{\ell}(b,e)})}
	\end{align*}
	for $u:b\to_B b'$, $e:P\,b$.
	The induced section $\ell'$ arises as $\ell'(a,d) \jdeq \ell(ka,d) \jdeq \angled{ka,d \widehat{\ell}(ka,d)}$ for $a:A$, $d:P'\,a \simeq P\,ka$. Applying this to the $P'$-cocartesian lift of $v:a \to a'$ \wrt~$d:Q\,a$ yields
	\begin{align*}
		\ell'(v,P'_!(v,d)) & \jdeq \ell(kv,P'_!(kv,d)) = \ell(kv,P_!(v,d)) \defeq \angled{kv,P_!(v,d),\widehat{\ell}(kv,P_!(v,d))} \\
		& = \angled{kv,P_!(kv,d),Q_!(kv, \pair{(kv)^P_!(d)}{\widehat{\ell}(kv,P_!(kv,d))})} \\
		& = \angled{kv,P'_!(v,d),Q'_!(v, \pair{v^{P'}_!(d)}{\widehat{\ell'}(v,P'_!(v,d))})} \\
	\end{align*}
	confirming the claim.
\end{proof}

\begin{proposition}[Closure of sliced cocartesian fibrations under composition]\label{prop:clos-sl-cocart-fib-comp}
	Let $P,Q,R:B \to \UU$ be isoinner families over a Rezk type $B$ with unstraightenings
	\[ \xi \defeq \Un_B(Q) : F \fibarr B, \pi \defeq \Un_B(P) : E \fibarr B, \rho \defeq \Un_B(R) : G \fibarr B. \]
	Furthermore, assume we have fibered functors $\varphi:F \to_B E$, $\psi: E \to_B G$ that are sliced cocartesian over $B$. Then, so is their composite $\kappa:F \to_B G$:
	\[\begin{tikzcd}
		F && E && G \\
		\\
		&& B
		\arrow["\varphi", from=1-1, to=1-3]
		\arrow["\pi"{description}, two heads, from=1-3, to=3-3]
		\arrow["\psi", from=1-3, to=1-5]
		\arrow["\xi"{description}, two heads, from=1-1, to=3-3]
		\arrow["\rho"{description}, two heads, from=1-5, to=3-3]
		\arrow["\kappa"{description}, curve={height=-24pt}, from=1-1, to=1-5]
	\end{tikzcd}\]
	Moreover, if $\varphi$ and $\psi$ are cocartesian fibrations in cartesian fibrations, then so is $\psi \circ \varphi$.
\end{proposition}

\begin{proof}
	This proof works analogously to the one for the absolute situation in~\cite[Proposition~2.3.7]{BW21}. First of all, we fibrantly replace the objects at play (with some abbreviation for the term declarations):
	\begin{align*}
		G & \equiv \sum_{b:B} Q\,b, & E & \equiv \sum_{\substack{b:B \\ x:R\,b}} P\,b\,x, \\
		F & \equiv \sum_{\substack{b:B \\ x:R\,b \\ e:P\,b\,x}} Q\,b\,\,x\,e, & \relcomma{B}{\psi}{G} & \equiv \sum_{b,x,e} \sum_{x':R\,b} (x \to_{R\,b} x'), \\
		\relcomma{B}{\varphi \psi}{G} & \equiv \sum_{b,x,x',e} \sum_{d:Q\,b\,x\,e} (x \to_{R\,b} x'), & \relcomma{B}{\varphi}{E} & \equiv \sum_{b,x,x',e,d} \sum_{u:x \to_{R\,b} x'} (e \to_u^P e'), \\
		\VertArr_\pi(E) & \equiv \sum_{b,x,x',u,e,e'} (e \to_u^P e'). & 	& & 
	\end{align*}
	We are to construct from the given fibered LARI adjunctions
	\[\begin{tikzcd}
		{\VertArr(F)} && {\relcomma{B}{\varphi}{E}} && {\VertArr(E)} && {\relcomma{B}{\psi}{G}} \\
		& B &&&& B
		\arrow[""{name=0, anchor=center, inner sep=0}, "r"{description}, curve={height=12pt}, from=1-1, to=1-3]
		\arrow[""{name=1, anchor=center, inner sep=0}, "{r'}"{description}, curve={height=12pt}, from=1-5, to=1-7]
		\arrow[""{name=2, anchor=center, inner sep=0}, "\ell"{description}, curve={height=18pt}, dotted, from=1-3, to=1-1]
		\arrow[""{name=3, anchor=center, inner sep=0}, "{\ell'}"{description}, curve={height=18pt}, dotted, from=1-7, to=1-5]
		\arrow[two heads, from=1-1, to=2-2]
		\arrow[two heads, from=1-3, to=2-2]
		\arrow[two heads, from=1-5, to=2-6]
		\arrow[two heads, from=1-7, to=2-6]
		\arrow["\dashv"{anchor=center, rotate=-90}, draw=none, from=2, to=0]
		\arrow["\dashv"{anchor=center, rotate=-90}, draw=none, from=3, to=1]
	\end{tikzcd}\]
	a fibered LARI adjunction:
	\[\begin{tikzcd}
		{\VertArr(E)} && {\relcomma{B}{\varphi \psi}{G}} \\
		& B
		\arrow[""{name=0, anchor=center, inner sep=0}, "{r''}"{description}, curve={height=12pt}, from=1-1, to=1-3]
		\arrow[""{name=1, anchor=center, inner sep=0}, "{\ell''}"{description}, curve={height=18pt}, dotted, from=1-3, to=1-1]
		\arrow[two heads, from=1-1, to=2-2]
		\arrow[two heads, from=1-3, to=2-2]
		\arrow["\dashv"{anchor=center, rotate=-90}, draw=none, from=1, to=0]
	\end{tikzcd}\]
	Using the projection equivalences, indeed we find the diagram analogous to the proof in~\cite[Proposition~2.3.7]{BW21}, all fibered over $B$:
	\[\begin{tikzcd}
		{\VertArr_\xi(F)} \\
		&& {\relcomma{B}{\varphi}{E}} && {\relcomma{B}{\varphi \psi}{G}} && F \\
		&& {\VertArr_\pi(E)} && {\relcomma{B}{\psi}{G}} && E \\
		&&&& B
		\arrow[from=2-7, to=3-7]
		\arrow[from=3-5, to=3-7]
		\arrow[from=2-5, to=3-5]
		\arrow[from=2-5, to=2-7]
		\arrow[""{name=0, anchor=center, inner sep=0}, curve={height=6pt}, from=2-3, to=2-5]
		\arrow[from=2-3, to=3-3]
		\arrow[""{name=1, anchor=center, inner sep=0}, "r"', curve={height=6pt}, from=3-3, to=3-5]
		\arrow[from=3-3, to=4-5]
		\arrow[from=3-5, to=4-5]
		\arrow[from=3-7, to=4-5]
		\arrow["\lrcorner"{anchor=center, pos=0.125}, draw=none, from=2-5, to=3-7]
		\arrow["\lrcorner"{anchor=center, pos=0.125}, draw=none, from=2-3, to=3-5]
		\arrow[curve={height=-18pt}, from=1-1, to=2-7]
		\arrow[curve={height=12pt}, from=1-1, to=3-3]
		\arrow[""{name=2, anchor=center, inner sep=0}, "\ell"{description}, curve={height=12pt}, dotted, from=3-5, to=3-3]
		\arrow[""{name=3, anchor=center, inner sep=0}, curve={height=12pt}, dotted, from=2-5, to=2-3]
		\arrow[""{name=4, anchor=center, inner sep=0}, "{r'}"{description}, curve={height=6pt}, from=1-1, to=2-3]
		\arrow[""{name=5, anchor=center, inner sep=0}, "{\ell'}"{description}, curve={height=12pt}, dotted, from=2-3, to=1-1]
		\arrow["\dashv"{anchor=center, rotate=-90}, draw=none, from=2, to=1]
		\arrow["\dashv"{anchor=center, rotate=-90}, draw=none, from=3, to=0]
		\arrow["\dashv"{anchor=center, rotate=-144}, draw=none, from=5, to=4]
	\end{tikzcd}\]
	As before, the proclaimed fibered LARI arises by pulling back and then composing.
	
	The closure property descends to cocartesian fibrations in cartesian fibrations by closedness under composition of cartesian functors, and pullback stability of fibered cartesian sections by the dual of~\Cref{lem:cocart-sect-pb}.
\end{proof}

\begin{proposition}[Closure of sliced cocartesian fibrations under pullback]\label{prop:clos-sl-cocart-fib-pb}
	Let $\varphi:F \to_B E$ be a sliced cocartesian fibration over a Rezk type $B$. For any map $k:A \to B$ consider the pullback:
	\[\begin{tikzcd}
		{k^*F} && F \\
		& {k^*E} && E \\
		A && B
		\arrow["\xi"{description, pos=0.25}, two heads, from=1-3, to=3-3]
		\arrow["\varphi", curve={height=-6pt}, from=1-3, to=2-4]
		\arrow["\pi"{description}, two heads, from=2-4, to=3-3]
		\arrow[from=1-1, to=1-3]
		\arrow["{k^*\varphi}"{pos=0.8}, curve={height=-6pt}, dashed, from=1-1, to=2-2]
		\arrow["{k^*\xi}"{description, pos=0.3}, two heads, from=1-1, to=3-1]
		\arrow["{k^*\pi}"{description}, two heads, from=2-2, to=3-1]
		\arrow["k"', from=3-1, to=3-3]
		\arrow["\lrcorner"{anchor=center, pos=0.125}, draw=none, from=2-2, to=3-3]
		\arrow["\lrcorner"{anchor=center, pos=0.125}, shift right=2, draw=none, from=1-1, to=3-3]
		\arrow[crossing over, from=2-2, to=2-4]
	\end{tikzcd}\]
	Then the induced fibered functor $k^*\varphi: k^*F \to_A k^*E$ is a sliced cocartesian fibration over $A$.
	
	In particular, the analogous statement is true if $\varphi$ is assumed to be a cocartesian fibration in cartesian fibrations.
\end{proposition}

\begin{proof}
	Since pullback commutes with sliced commas\footnote{as can \eg~be checked by projection equivalence} we get the following square:
	\[\begin{tikzcd}
		{k^*\VertArr_{k^*\xi}(F)} && {\VertArr_\xi(F)} \\
		& {\relcomma{A}{k^*\varphi}{k^*E}} && {\relcomma{B}{\varphi}{E}} \\
		A && B
		\arrow[two heads, from=1-3, to=3-3]
		\arrow[""{name=0, anchor=center, inner sep=0}, "{\varphi'}", curve={height=-6pt}, from=1-3, to=2-4]
		\arrow[two heads, from=2-4, to=3-3]
		\arrow[from=1-1, to=1-3]
		\arrow["{k^*\varphi'}"{pos=0.8}, curve={height=-6pt}, dashed, from=1-1, to=2-2]
		\arrow[two heads, from=1-1, to=3-1]
		\arrow[two heads, from=2-2, to=3-1]
		\arrow["k"', from=3-1, to=3-3]
		\arrow["\lrcorner"{anchor=center, pos=0.125}, draw=none, from=2-2, to=3-3]
		\arrow["\lrcorner"{anchor=center, pos=0.125}, shift right=2, draw=none, from=1-1, to=3-3]
		\arrow[""{name=1, anchor=center, inner sep=0}, curve={height=-6pt}, dotted, from=2-4, to=1-3]
		\arrow["\dashv"{anchor=center, rotate=12}, draw=none, from=1, to=0]
		\arrow[from=2-2, to=2-4, crossing over]
	\end{tikzcd}\]
	The fibered LARI on the right gets pulled back to define a fibered LARI on the left, as desired, hence the pulled back functor $k^*\varphi$ is sliced cocartesian as well.
	
	If $\varphi:F \to_B E$ is a cartesian functor between cartesian functors, then $\varphi'$ pulls back to define a cartesian functor between cartesian fibrations by (the dual of)~\cite[Proposition~5.3.21]{BW21}, \cf~also \cite[Lemma~5.3.5]{RV21}. In case the fibered LARI is cartesian, the induced fibered LARI is as well, as one can see by the dual of~\Cref{lem:cocart-sect-pb}.

\end{proof}

\section{Two-sided cartesian families}\label{sec:2s-cart}

Two-sided cartesian families can be presented by spans
\[\begin{tikzcd}
	& E \\
	A && B
	\arrow["\xi"', two heads, from=1-2, to=2-1]
	\arrow["\pi", two heads, from=1-2, to=2-3]
\end{tikzcd}\]
where $\xi$ is a cocartesian fibration and $\pi$ is a cartesian fibration, satisfying some compatibility conditions between the lifts of opposite variance.

We will develop this notion in a systematic manner, closely following Riehl--Verity's treatment in $\infty$-cosmoses~\cite[Chapter~7]{RV21}.

Semantically, they correspond to $\inftyone$-functors $A^{\Op} \times B \to \Cat_\infty$, and their \emph{discrete} variant models $\inftyone$-functors $\inftyone$-functors $A^{\Op} \times B \to \mathrm{Space}$ factoring through the $\inftyone$-category of spaces \aka~$\inftyone$-distributors or $\inftyone$-profunctors (sometimes also called correspondence or bifibration, though the latter terminology can also be used to mean a fibration that is simultaneously cartesian and cocartesian).

(Discrete) two-sided cartesian fibrations have a long history of study in various contexts~\cite{yoneda1960ext,StrYon,StrBicat,BournPenon,benaboudistributeurs,StrDist,LorRieCatFib,bakovic2011fibrations,vonGlehn2018polynomials,mccloskeyRelWfs,di2019unicity,CigoliSpan,metere2017distributors}.

In the higher categorical context they are used in the theory of fibrations of $(\infty,1)$-categories~\cite{LurHTT,AFfib,BarwickShahFib,Haugseng2023TwoVar,stevenson2018model}, the formal category theory of $\infty$-categories~\cite{RV21,RVKan,RuitEquip}, and higher algebra and geometry~\cite{Lur17,haugseng2017higher,galvez2020decomposition}.

\subsection{Two-variable families and bifibers}\label{ssec:two-var}

For types $A,B: \UU$ consider a family
\[ P:A \to B \to \UU.\]

For $a:A$ and $b:B$, the type $P(a,b)$ is called the \emph{bifiber} of $P$ at $a$ and $b$. Fixing one of the elements gives rise to the definitions
\[ P_b \defeq \lambda a.P(a,b): A \to \UU, \quad P^a \defeq \lambda b.P(a,b): B \to \UU.\]
The two ``legs'' of the family $P$ are given by the families
\[ P_B \defeq \lambda a.\sum_{b:B} P_b(a):A \to \UU, \quad  P^A \defeq \lambda b.\sum_{a:B} P^a(b): B \to \UU.\]

We obtain the following version of the typal Grothendieck construction. By transposition we have a chain of fiberwise equivalences:
\[{\small \begin{tikzcd}
	{\sum_{A,B:\UU} A \to B \to \UU} && {\sum_{A,B:\UU} (A \times B) \to \UU} && {\sum_{A,B:\UU} \Fib_\UU(A \times B)} \\
	&& {\UU \times \UU}
	\arrow[from=1-1, to=2-3]
	\arrow[from=1-3, to=2-3]
	\arrow["\simeq", from=1-3, to=1-5]
	\arrow[from=1-5, to=2-3]
	\arrow["\simeq", from=1-1, to=1-3]
\end{tikzcd}}\]
Hence two-sided families $P:A \to B \to \UU$ correspond to maps over $\pair{\xi}{\pi}:\widetilde{P} \to A \times B$.

Diagrammatically, this manifests as follows:
\[\begin{tikzcd}
	& {\sum_{\substack{a:A \\ b:B}} P\,a\,b } \\
	& {A \times B} \\
	A && B
	\arrow["\varphi"', two heads, from=1-2, to=2-2]
	\arrow[two heads, from=2-2, to=3-1]
	\arrow[two heads, from=2-2, to=3-3]
	\arrow["{\xi \defeq \Un_A(P_B)}"{description}, curve={height=18pt}, two heads, from=1-2, to=3-1]
	\arrow["{ \Un_B(P^A) \defeq \pi}"{description}, curve={height=-18pt}, two heads, from=1-2, to=3-3]
\end{tikzcd}\]

The fibers at $a:A$ or $b:B$, resp., are obtained as follows
\[\begin{tikzcd}
	{E_b} && E && {E^a} && E \\
	A && {A \times B} && B && {A \times B}
	\arrow["{\xi_b}"', from=1-1, to=2-1]
	\arrow["{\langle \id_A,b\rangle}"', from=2-1, to=2-3]
	\arrow[from=1-1, to=1-3]
	\arrow["{\langle\xi, \pi\rangle}", from=1-3, to=2-3]
	\arrow["{\pi^a}"', from=1-5, to=2-5]
	\arrow["{\langle a,\id_B\rangle}"', from=2-5, to=2-7]
	\arrow[from=1-5, to=1-7]
	\arrow["{\langle\xi, \pi\rangle}", from=1-7, to=2-7]
	\arrow["\lrcorner"{anchor=center, pos=0.125}, draw=none, from=1-5, to=2-7]
	\arrow["\lrcorner"{anchor=center, pos=0.125}, draw=none, from=1-1, to=2-3]
\end{tikzcd}\]
where the projections arise as unstraightenings
\[ \xi_b \defeq \Un_A(P_b): E_b \to A, \quad \pi^a \defeq \Un_B(P^a): E^a \to B. \]
The notation comes from the convention of denoting the components of the projection $E \to A \times B$ conceived as a fibered functor in two different ways
\[\begin{tikzcd}
	E && {A \times B} && E && {A \times B} \\
	& B &&&& A
	\arrow["{\langle \xi,\pi\rangle}", from=1-1, to=1-3]
	\arrow["\pi"', from=1-1, to=2-2]
	\arrow["q", from=1-3, to=2-2]
	\arrow["{\langle \xi,\pi\rangle}", from=1-5, to=1-7]
	\arrow["\xi"', from=1-5, to=2-6]
	\arrow["p", from=1-7, to=2-6]
\end{tikzcd}\]
as
\[ \xi_b \defeq \pair{\xi}{\pi}_b :E_b \to (A\times B)_b \simeq A, \quad  \pi^a \defeq  \pair{\xi}{\pi}^a : E^a \to (A\times B)^a \simeq B. \]

\subsection{Cocartesianness on the left}\label{ssec:cocart-on-left}

\begin{definition}[Cocartesian on the left]
	A two-sided family $P:A \to B \to \UU$ is \emph{cocartesian on the left} if the family $P_B:A \to \UU$ is cocartesian, and every $P_B$-cocartesian arrow in $\totalty{P}$ is $P^A$-vertical.
\end{definition}

\begin{example}[Cocartesian families]\label{ex:cocart-fams-cocart-left}
	A family $P:A \to \UU$ is cocartesian if and only if it is cocartesian on the left, seen as a family $P: (A \to \UU) \simeq (A \to 1 \to \UU)$.
\end{example}

%

\begin{proposition}[Characterizations of cocartesianness on the left, {\protect\cite[Lemma~7.1.1]{RV21}}]\label{prop:char-fib-cocart-left}
	For a two-sided family $P:A \to B \to \UU$, corresponding to $\pair{\xi}{\pi}:E \to A \times B$, the following are equivalent:
	\begin{enumerate}
		\item\label{it:coc-left-i} The fibered functor
		\[\begin{tikzcd}
			E && {A \times B} \\
			& B
			\arrow["{\langle \xi, \pi \rangle}", two heads, from=1-1, to=1-3]
			\arrow["\pi"', from=1-1, to=2-2]
			\arrow["q", from=1-3, to=2-2]
		\end{tikzcd}\]
		is a cocartesian fibration sliced over $B$.
		\item\label{it:coc-left-ii} The fibered functor
		\[\begin{tikzcd}
			E && {A \times B} \\
			& A
			\arrow["{\langle \xi, \pi \rangle}", from=1-1, to=1-3]
			\arrow["\xi"', two heads, from=1-1, to=2-2]
			\arrow["p", two heads, from=1-3, to=2-2]
		\end{tikzcd}\]
		is a cocartesian functor between cocartesian fibrations.
		\item\label{it:coc-left-iii} The fibered functor given by
		\[ \iota_\xi: E \to_{A \times B} \comma{\xi}{A}, \quad \iota_\xi(a,b,e) \defeq \angled{a,a,\id_a,b,e} \]
		has a fibered left adjoint $\tau_\xi$:
		\[\begin{tikzcd}
			E &&&& {\xi \downarrow A} \\
			&& {A \times B}
			\arrow[""{name=0, anchor=center, inner sep=0}, "{\iota_\xi}"{description}, from=1-1, to=1-5]
			\arrow["{\langle \xi,\pi \rangle}"', from=1-1, to=2-3]
			\arrow["{\langle \partial_1,\pi \circ \partial_0\rangle}", from=1-5, to=2-3]
			\arrow[""{name=1, anchor=center, inner sep=0}, "{\tau_\xi}"{description}, curve={height=18pt}, dashed, from=1-5, to=1-1]
			\arrow["\dashv"{anchor=center, rotate=-90}, draw=none, from=1, to=0]
		\end{tikzcd}\]
		\item\label{it:coc-left-iv} The two-sided family $P:A \to B \to \UU$ is cocartesian on the left.
	\end{enumerate}
\end{proposition}

\begin{proof} We abbreviate $\varphi \defeq \pair{\xi}{\pi}: E \to_B A \times B$.
	\begin{description}
		\item[$\ref{it:coc-left-i} \iff \ref{it:coc-left-iv}$] The fibered functor $\varphi$ being a cocartesian fibration sliced over $B$ is equivalent to the condition that for all $q$-vertical maps exist a $\varphi$-cocartesian lift (\wrt~to a given initial vertex). By Rezk-completeness, this is equivalent to any arrow $\pair{u:a \to a'}{\id_b}$ having a $\varphi$-cocartesian lift $f$ with prescribed initial vertex $e:P(a,b)$. Note that $f$ is $\pi$-vertical. Projecting away the $B$-component, we obtain that $f$ is $\xi$-cocartesian. 
		\item[$\ref{it:coc-left-ii} \iff \ref{it:coc-left-iv}$] Since $p:A \times B \fibarr A$ is a cocartesian fibration in any case, the assumption is equivalent to $\xi:E \fibarr A$ being a cocartesian fibration and every $\xi$-cocartesian arrow being mapped to $p$-cocartesian arrows under $\varphi$. But $\varphi$ is just the projection pairing $\pair{\xi}{\pi}$, and $p$-cocartesian arrows are exactly given by arrows whose $B$-component is an identity. Projecting down to $B$ this means exactly that the $\xi$-cocartesian arrows are $\pi$-vertical.
		\item[$\ref{it:coc-left-iii} \implies\ref{it:coc-left-iv}$] We denote
		\[ \tau_\xi(a',b,u:a\to a',e) \defeq  \angled{a',b,\widehat{\tau}_{\xi,a',b}(u,e)}. \]
		Again, similarly to the considerations in the proof of~\Cref{thm:cocartfams-via-transp}, the unit is a family of arrows
		\[ \eta:\prod_{a',b:A} \prod_{a:A, u:a \to_A a'} \prod_{e:P(a,b)} \angled{a',b,a,u,e} \to \angled{a',b,a',\id_{a'}, \widehat{\tau}_{\xi,a',b}(u,e)},\]
		illustrated as follows:
		\[\begin{tikzcd}
			e && {\widehat{\tau}_{\xi,a',b}(u,e)} \\
			a && {a'} \\
			{a'} && {a'} \\
			{a'} && {a'} \\
			b && b
			\arrow["{\eta_{a',b,a,u,e}}", from=1-1, to=1-3, cocart]
			\arrow["u"', from=2-1, to=3-1]
			\arrow["u", from=2-1, to=2-3]
			\arrow["{\id_{a'}}", Rightarrow, no head, from=3-1, to=3-3]
			\arrow["{\id_{a'}}", Rightarrow, no head, from=2-3, to=3-3]
			\arrow["{\id_{a'}}", Rightarrow, no head, from=4-1, to=4-3]
			\arrow["{\id_{b}}", Rightarrow, no head, from=5-1, to=5-3]
			\arrow[curve={height=12pt}, Rightarrow, dotted, no head, from=1-1, to=2-1]
			\arrow[curve={height=-12pt}, Rightarrow, dotted, no head, from=1-3, to=2-3]
			\arrow[curve={height=12pt}, Rightarrow, dotted, no head, from=3-1, to=4-1]
			\arrow[curve={height=-12pt}, Rightarrow, dotted, no head, from=3-3, to=4-3]
		\end{tikzcd}\]
		By assumption, the transposing map induced by $\eta$ is an equivalence:
		{\small \begin{align*}
			\Phi_\eta & : \prod_{\substack{a',a'':A \\ b,b':B}} \prod_{\substack{u':a' \to_A a'' \\ v:b \to_B b'}} \prod_{\substack{a:A \\ u:a \to_A a' \\ e:P(a,b)}} \prod_{\substack{a'':A \\ b':B \\ d:P(a'',b')}} \left( \angled{a',b,\widehat{\tau}_{\xi,a',b}(u,e):P(a',b)} \to_{\pair{u'}{v}} \angled{a'',b',d} \right) \\
			\stackrel{\simeq}{\longrightarrow} & \left( \angled{a',b,a,u:a \to_A a', e:P(a,b)} \to_{\pair{u'}{v}}  \angled{a'',b',a'', \id_{a''}, d:P(a'',b')} \right), \\
			\Phi_\eta & \defeq \lambda u',v,e,d,g.\iota_\xi(g) \circ \eta_{u,e} : \big(e \to_{\pair{v}{g}}^P d\big)
		\end{align*}}
		After contracting away redundant data, this is equivalent to the proposition
		\[ \prod_{h:e \to^P_{\pair{u'u}{v}} d} \isContr\Big( \sum_{g:\widehat{\tau}_{\xi,a',b}(u,e) \to^P_{\pair{u'}{v}} d} g \circ^P_{\eta_{u,e}} = h \Big), \]
		\cf~\Cref{fig:lift-cocart-left-fib-adj} for illustration. But this precisely means that $\xi:E \fibarr A$ is a cocartesian fibration whose cocartesian lifts all are $\pi$-vertical, namely the components of the fibered unit $\eta$.
		\item[$\ref{it:coc-left-iv} \implies\ref{it:coc-left-iii}$] We can strictify the diagram as follows, including the fibered left adjoint to be defined:
		\[{\small \begin{tikzcd}
			{E \simeq \sum_{\substack{a':A \\ b:B}} P(a',b)} &&&& {\sum_{\substack{a,a':A \\ b:B}} (a \to_A a') \times P(a,b) \simeq \xi \downarrow A} \\
			\\
			&& {A \times B}
			\arrow["{\langle \partial_1, \pi \circ \partial_0\rangle}", two heads, from=1-5, to=3-3]
			\arrow["{\langle \xi,\pi \rangle}"', two heads, from=1-1, to=3-3]
			\arrow[""{name=0, anchor=center, inner sep=0}, "{\iota_\xi}"{description}, curve={height=12pt}, from=1-1, to=1-5]
			\arrow[""{name=1, anchor=center, inner sep=0}, "{\tau_\xi}"', curve={height=6pt}, from=1-5, to=1-1]
			\arrow["\dashv"{anchor=center,rotate=-90}, draw=none, from=1, to=0]
		\end{tikzcd}}\]
		The fibered ``inclusion'' map is defined as
		\[ \iota_\xi(a',b,e) \defeq \angled{a',b,a',\id_{a'},e}.\]
		By the preconditions from~\Cref{it:coc-left-iii}, the map $\xi:E \fibarr A$ is a cocartesian fibration, moreover whose cocartesian lifts are all $\pi$-vertical. We let
		\[ \tau_\xi(a',b,a,u:a \to a',e) \defeq \angled{a',b,u_!\,e}.\]
		Now, similarly, as in the proof of~\Cref{thm:cocartfams-via-transp}, we exhibit the fibered adjunction as given by a fiberwise equivalence\footnote{Here, we write $\xi_b \defeq (\id_A \times b)^*\pair{\xi}{\phi}: E_b \simeq \sum_{a:A} P(a,b) \fibarr A$, giving rise to the comma object $\sum_{a:A} (a \to_A a') \times P(a,b)$} 
		\begin{align*} 
		\prod_{\substack{a':A \\ b:B}} \prod_{\substack{a:A \\ u:a \to_A a' \\ e:P(a,b)}} & \prod_{e':P(a',b)} \big( \tau_{\xi,a',b}(a,u,u_!\,e) \to_{P(a',b)} e'\big) \\ & \stackrel{\simeq}{\longrightarrow} \big( \angled{a,u,e} \longrightarrow_{\comma{\xi_b}{A}}  \angled{a',\id_{a'}, e'} \big)
		\end{align*} 
		as follows: Over a point $\pair{a'}{b}:A \times B$ in the base, fix $\angled{a:A, u:a \to_A a', e:P(a,b)}$, $e':P(a',b)$, and define maps between the transposing hom types, in opposite directions,
		\begin{align*}
			& \Phi(g:u_!e \to_{P(a',b)} e') \defeq g \circ \xi_!(u,e), \\
			& \Psi(u:a \to_A a', u:u \to_{\comma{A}{a}} \id_{a'}, h:e \to_u^P e') \defeq \tyfill_{\xi_!(u,e)}^\xi(h). 
		\end{align*}
		Again, by the universal property of cocartesian lifts it can be checked that the two maps are quasi-inverses. Note that by assumption, the $\xi$-cocartesian lifts are $\pi$-vertical, so everything stays in ``the fibers over $b$'', or, more precisely, in the pullback type $(\id_A \times b)^*E$.

	\end{description}
	
	\begin{figure}
		\centering
		\[\begin{tikzcd}
			e && d \\
			{\widehat{\eta}(u,e)} && d \\
			a && {a''} \\
			{a'} && {a''} \\
			{a'} && {a''} \\
			b && {b'}
			\arrow["{\eta_{u,e}}"', swap, description, from=1-1, to=2-1, cocart]
			\arrow["u"', from=3-1, to=4-1]
			\arrow["{u'u}", from=3-1, to=3-3]
			\arrow["{u'}", from=4-1, to=4-3]
			\arrow["{\id_{a''}}"', Rightarrow, no head, from=3-3, to=4-3]
			\arrow["{u'}", from=5-1, to=5-3]
			\arrow["v", from=6-1, to=6-3]
			\arrow["{\forall \,h}", from=1-1, to=1-3]
			\arrow[Rightarrow, no head, from=1-3, to=2-3]
			\arrow[curve={height=30pt}, Rightarrow, dotted, no head, from=1-1, to=3-1]
			\arrow[curve={height=-30pt}, Rightarrow, dotted, no head, from=1-3, to=3-3]
			\arrow["{\exists! \, g}", from=2-1, to=2-3, dashed]
			\arrow[curve={height=30pt}, Rightarrow, dotted, no head, from=2-1, to=4-1, crossing over]
			\arrow[curve={height=-30pt}, Rightarrow, dotted, no head, from=2-3, to=4-3, crossing over]
		\end{tikzcd}\]
		\caption{Fibered adjunction criterion for cocartesian-on-the-left families}
		\label{fig:lift-cocart-left-fib-adj}
	\end{figure}
	
	\begin{figure}
		\centering
		\[\begin{tikzcd}
			&& e && {e'} \\
			E && {u_!'e} \\
			&& {a''} && {a'} \\
			&& a && {a'} \\
			{A \times B} && b && {b'}
			\arrow["{\xi_!(u',e)}"', from=1-3, to=2-3, cocart]
			\arrow["h", from=1-3, to=1-5]
			\arrow["g"', dashed, from=2-3, to=1-5]
			\arrow["{u'}"', from=3-3, to=4-3]
			\arrow["u", from=4-3, to=4-5]
			\arrow["{uu'}", dashed, from=3-3, to=3-5]
			\arrow[Rightarrow, no head, from=3-5, to=4-5]
			\arrow["v", from=5-3, to=5-5]
			\arrow[two heads, from=2-1, to=5-1]
		\end{tikzcd}\]
		\caption{Lifts in a family that is cocartesian on the left}
		\label{fig:lift-cocart-left}
	\end{figure}
	
\end{proof}

By dualization, one obtains the notion of a two-sided family \emph{cartesian on the right}. As a corollary, we obtain a characterization of the conjunction of both properties.

\begin{corollary}[{\protect\cite[Corollary~7.1.3]{RV21}}]\label{cor:cocart-left-and-cart-right}
	A two-variable family $P:A \to B \to \UU$ is cocartesian on the left and cartesian on the right if and only if the following equivalent propositions are true.
	\begin{enumerate}
		\item The fibered functor
		\[\begin{tikzcd}
			E && {A \times B} \\
			& B
			\arrow["{\langle \xi,\pi \rangle}", two heads, from=1-1, to=1-3]
			\arrow["\pi"', two heads, from=1-1, to=2-2]
			\arrow["q", two heads, from=1-3, to=2-2]
		\end{tikzcd}\]
		is both a cartesian functor between cartesian fibrations and a relative cocartesian fibration over $B$.
		\item The fibered functor
		\[\begin{tikzcd}
			E && {A \times B} \\
			& A
			\arrow["{\langle \xi,\pi \rangle}", two heads, from=1-1, to=1-3]
			\arrow["\xi"', two heads, from=1-1, to=2-2]
			\arrow["p", two heads, from=1-3, to=2-2]
		\end{tikzcd}\]
		is both a cocartesian functor between cocartesian fibrations and a relative cartesian fibration over $B$.
	\end{enumerate}
\end{corollary}

In the case that $P:A \to B \to \UU$ is cocartesian on the left and cartesian on the right, we denote the ensuing lifting operations as follows. Given $a:A$, $b:B$, $e:P\,a\,b$, for arrows $u:a \downarrow A$, $v:B \downarrow b$, there are lifts
\[ P_!(u,e): e \cocartarr_{\pair{u}{b}} u_!\,e, \quad P^*(v,e): v^*\,e \cartarr_{\pair{a}{v}} e,\]
where in the notation we identify elements and identity maps.

The notion of \emph{two-sided cartesian fibration} adds on top a certain compatibility condition between the lifts of different variance. Before defining two-sided cartesian fibrations in~\Cref{ssec:2s-cart-fams}, we first investigate the compatibility condition in the following lemma.

\begin{lemma}[Comparing cartesian and cocartesian transport]\label{lem:comp-cart-cocart-transp}
	For Rezk types $A$ and $B$, let $P: A \to B \to \UU$ be a two-variable family which is cocartesian on the left and cartesian on the right. Denoting its unstraightening as~$\varphi\jdeq \pair{\xi}{\pi}:E \fibarr A \times B$, for arrows $u:a \to_A a'$, $v:b' \to_B b$ and a point $e:P\,a\,b$ we abbreviate:
	\begin{align*}
		k & \defeq \xi_!(u,e):e \cocartarr^\xi_{\pair{u}{b}} u_!\,e  &  k' & \defeq \pi^*(v,e):v^*\,e \cartarr^\pi_{\pair{a}{v}} e \\
		f & \defeq \xi_!(u,v^*\,e) : v^*\,e \cocartarr^\xi_{\pair{u}{b'}} u_!\,v^*\,e & f' & \defeq \pi^*(v,u_! \,e) : v^*\,u_!\,e \cartarr^\pi_{\pair{a'}{v}} u_!\,e \\
		g & \defeq \tyfill_f^\xi(k'k): u_!\,v^*\,e \to_{\pair{a'}{v}}^\xi u_!\,e & g' & \defeq \tyfill_{f'}^\pi(k'k): v^*\,e \to^\pi_{\pair{u}{b'}} v^*\, u_!\,e
	\end{align*}
	We claim that there is an identification between the following two induced morphisms $h,h':u_!\,v^*\,e \to_{P(a',b')} v^*\,u_!\,e$~ (\cf~\ref{fig:comp-cart-cocart-transport}):
	\[\begin{tikzcd}
		{v^*\,e} && e && {u_!\,e} \\
		{u_!v^*\,e} &&&& {v^* u_!\,e} \\
		{v^*\,e} && e && {u_!\,e} \\
		{u_!v^*\,e} &&&& {v^* u_!\,e}
		\arrow["f"', from=1-1, to=2-1]
		\arrow["{k'}", from=1-1, to=1-3, cart]
		\arrow["k", from=1-3, to=1-5, cocart]
		\arrow["{f'}"', from=2-5, to=1-5, cart]
		\arrow["{h:\defeq \tyfill^\pi_{f'}(g)}"', dashed, from=2-1, to=2-5]
		\arrow["g"{description}, dashed, from=2-1, to=1-5]
		\arrow["{k'}", from=3-1, to=3-3, cart]
		\arrow["k", from=3-3, to=3-5, cocart]
		\arrow["f"', from=3-1, to=4-1, cocart]
		\arrow["{h':\defeq \tyfill^\xi_{f}(g')}"', dashed, from=4-1, to=4-5]
		\arrow["{f'}"', from=4-5, to=3-5, cart]
		\arrow["{g'}"{description}, dashed, from=3-1, to=4-5]
	\end{tikzcd}\]
\end{lemma}

\begin{proof}
	It is sufficient to provide an identification $g=f'h'$. For this, it is sufficient---and necessary---to provide a witness for $f'(h'f) = k'k$. But this follows from $f'h=g$, since $gf=k'k$.
\end{proof}

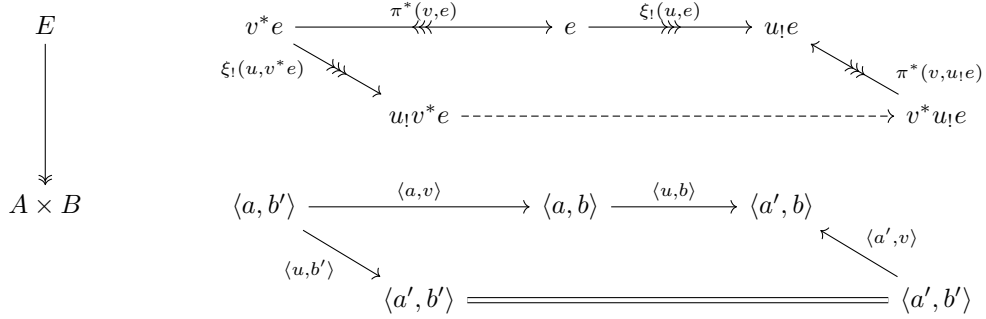
\begin{figure}
	\centering
	\[\begin{tikzcd}
		E && {v^*e} && e && {u_!e} \\
		&&& {u_!v^*e} &&&& {v^*u_!e} \\
		{A \times B} && {\langle a,b'\rangle} && {\langle a,b \rangle} && {\langle a',b\rangle} \\
		{} &&& {\langle a',b' \rangle} &&&& {\langle a',b'\rangle}
		\arrow["{\xi_!(u,v^*e)}"'{pos=0.2}, from=1-3, to=2-4, cocart]
		\arrow["{\pi^*(v,e)}", from=1-3, to=1-5, cart]
		\arrow["{\xi_!(u,e)}", from=1-5, to=1-7, cocart]
		\arrow["{\pi^*(v,u_!e)}"'{pos=0.1}, from=2-8, to=1-7, cart]
		\arrow[dashed, from=2-4, to=2-8]
		\arrow["{\langle u,b' \rangle}"', from=3-3, to=4-4]
		\arrow["{\langle a,v \rangle}", from=3-3, to=3-5]
		\arrow["{\langle u,b \rangle}", from=3-5, to=3-7]
		\arrow[two heads, from=1-1, to=3-1]
		\arrow["{\langle a',v \rangle}"', from=4-8, to=3-7]
		\arrow[Rightarrow, no head, from=4-4, to=4-8]
	\end{tikzcd}\]
	\caption{Comparing cartesian and cocartesian transport}\label{fig:comp-cart-cocart-transport}
\end{figure}

\subsection{Two-sided cartesian families}\label{ssec:2s-cart-fams}

\subsubsection{Definition and examples of two-sided cartesian fibrations}

\begin{definition}[Two-sided cartesian families, \protect{\cite[Section~7.1]{RV21}}]\label{def:2s-cart}
	Let $P:A \to B \to \UU$ be a (two-sided) family, where $A$ and $B$ are Rezk types. We call $P$ a \emph{two-sided cartesian family} (short: \emph{two-sided family}) if
	\begin{enumerate}
		\item $P$ is cocartesian on the left and cartesian on the right,
		\item\label{it:lifts-commute} and $P$ satisfies the condition that \emph{cocartesian and cartesian lifts commute}: for any $a:A$, $b:B$, $e:P\,a\,b$ and arrows $u:a \to_A a'$, $v:b' \to_B b$, the filler $\kappa:u_!(v^*e) \to v^*(u_!e)$ from \Cref{lem:comp-cart-cocart-transp} is an isomorphism, hence there is an identity $u_!v^*e =_{P(a',b')} v^*u_!e$.
	\end{enumerate}
\end{definition}

For any Rezk type $A$ the $\hom$-span
\[\begin{tikzcd}
	& {A^{\Delta^1}} \\
	A && A
	\arrow["{\partial_1}"', from=1-2, to=2-1]
	\arrow["{\partial_0}", from=1-2, to=2-3]
\end{tikzcd}\]
is a two-sided cartesian fibration: the $\partial_1$-cocartesian lift are exactly the $\partial_0$-vertical squares, and vice versa, and the transports satisfy the necessary compatibility condition (\cf~\cite[Proposition~5.2.15 and Corollary~5.2.16]{BW21}). In fact, this gives an example of a \emph{discrete} two-sided family, since all the fibers are discrete types, \cf~\cite[Proposition~8.29]{RS17}, as well as \Cref{def:2s-disc-cart} and \Cref{prop:2s-disc}.

More generally, comma fibrations of the form
\[\begin{tikzcd}
	& {\comma{f}{g}} \\
	A && B
	\arrow["{\partial_1}"', from=1-2, to=2-1]
	\arrow["{\partial_0}", from=1-2, to=2-3]
\end{tikzcd}\]
for functors $f : B \to A$ and $g : C \to A$ yield examples of two-sided discrete cartesian fibrations, by the analogous arguments.

Furthermore, given two-sided cartesian fibrations $E \fibarr A \times B$ and $F \fibarr B \times C$, we will see that their \emph{span composition}
\[\begin{tikzcd}
	&& {E \times_B Q } \\
	& E && Q \\
	A && B && C
	\arrow[two heads, from=2-2, to=3-1]
	\arrow[two heads, from=2-2, to=3-3]
	\arrow[two heads, from=2-4, to=3-3]
	\arrow[two heads, from=2-4, to=3-5]
	\arrow[two heads, from=1-3, to=2-2]
	\arrow[two heads, from=1-3, to=2-4]
	\arrow["\lrcorner"{anchor=center, pos=0.125, rotate=-45}, draw=none, from=1-3, to=3-3]
\end{tikzcd}\]
is a two-sided cartesian fibration, too. For example, composing the $\hom$-span $A^{\Delta^1} \fibarr A \times A$, which is a discrete two-sided fibration, with itself yields a two-sided cartesian fibration
\[\begin{tikzcd}
	&& {A^{\Delta^2}} \\
	& {A^{\Delta^1}} && {A^{\Delta^1}} \\
	A && A && A
	\arrow["{\partial_1}"', two heads, from=2-2, to=3-1]
	\arrow["{\partial_0}", two heads, from=2-2, to=3-3]
	\arrow["{\partial_1}"', two heads, from=2-4, to=3-3]
	\arrow["{\partial_0}", two heads, from=2-4, to=3-5]
	\arrow[two heads, from=1-3, to=2-2]
	\arrow[two heads, from=1-3, to=2-4]
	\arrow["\lrcorner"{anchor=center, pos=0.125, rotate=-45}, draw=none, from=1-3, to=3-3]
\end{tikzcd}\]
whose total type is equivalent to $A^{\Delta^2}$ due to $A$ being Segal. Its fiber at $\langle c,a \rangle : A \times A$ is $\sum_{b:A} \hom(a,b) \times \hom(b,c)$, which is in general not discrete. Thus, $A^{\Delta^2} \fibarr A \times A$ is an example of a not necessarily discrete two-sided cartesian fibration.\footnote{We are grateful to Emily Riehl for pointing out this example.}

Furthermore, any family $P : A \times B \to \UU$ with associated span $\xi : A \leftarrow E \rightarrow B : \pi$ of functors gives rise to the \emph{free-two sided fibration}, see~\cite{Shu:2s-free}:
\[\begin{tikzcd}
	&&& {L^{\mathrm{2s}}(E)} \\
	&& {\comma{\xi}{A}} && {\comma{B}{\pi}} \\
	& {A^{\Delta^1}} && E && {B^{\Delta^1}} \\
	A && A && B && B
	\arrow[two heads, from=1-4, to=2-3]
	\arrow[two heads, from=1-4, to=2-5]
	\arrow[two heads, from=2-3, to=3-4]
	\arrow[two heads, from=2-5, to=3-4]
	\arrow[from=2-3, to=3-2]
	\arrow[from=2-5, to=3-6]
	\arrow["{{\partial_0}}", two heads, from=3-2, to=4-3]
	\arrow["\xi"', from=3-4, to=4-3]
	\arrow["\pi", from=3-4, to=4-5]
	\arrow["{{\partial_1}}"', two heads, from=3-6, to=4-5]
	\arrow["{{\partial_1}}"', two heads, from=3-2, to=4-1]
	\arrow["{{\partial_0}}", two heads, from=3-6, to=4-7]
	\arrow["\lrcorner"{anchor=center, pos=0.125, rotate=-45}, draw=none, from=2-3, to=4-3]
	\arrow["\lrcorner"{anchor=center, pos=0.125, rotate=-45}, draw=none, from=2-5, to=4-5]
	\arrow["\lrcorner"{anchor=center, pos=0.125, rotate=-45}, draw=none, from=1-4, to=3-4]
\end{tikzcd}\]
That this is a two-sided cartesian fibration satisfying the corresponding universal property can be proven using the results about the free cocartesian fibration from~\cite[Subsection~5.2.17]{BW21} and their dual analogue for the free cartesian fibration.

The fiber of $L^{\mathrm{2s}}(E) \fibarr A \times B$ at $\pair{a}{b} : A \times B$ is given, up to equivalence, by the type $\sum_{a' : A} \sum_{b' : B} (a' \to_A a) \times (b \to_B b') \times P(a',b')$.

\subsubsection{Characterization of two-sided cartesian families}

\begin{proposition}[Commutation of cocartesian and cartesian lifts]\label{prop:comm-lifts}
	Let $P:A \to B \to \UU$ be a family with both $A$ and $B$ Rezk which is cocartesian on the left and cartesian on the right. We denote by $\xi:F \fibarr A$ and $\pi:E \fibarr B$, resp., the unstraightenings.
	Then cocartesian and cartesian lifts commute if and only if the following property is satisfied:
	Given $u:a \to_A a'$, $v:b' \to_B b$, $e:P\,a\,b$, and a diagram
	\[\begin{tikzcd}
		{v^*\,e} && e \\
		d && {u_!\,e}
		\arrow["f"', from=1-1, to=2-1]
		\arrow["{k'\defeq \pi^*(v,e)}", from=1-1, to=1-3, cart]
		\arrow["g"', from=2-1, to=2-3]
		\arrow["{k \defeq \xi_!(u,e)}", from=1-3, to=2-3, cocart]
	\end{tikzcd}\]
	where $g$ (and necessarily $k'$) is $\xi$-vertical and $f$ (and necessarily $k$) is $\pi$-vertical. Then $f$ is $\xi$-cocartesian if and only if $g$ is $\pi$-cartesian.
\end{proposition}

\begin{proof}
	In light of~\Cref{lem:comp-cart-cocart-transp}, the commutation condition is equivalent to the type of paths $h:u_!\,v^*\,e =_{P(a',b')} v^*\,u_!\,e$ together with witnesses, necessarily propositional, that the following diagram commutes:
	\[\begin{tikzcd}
		&& {v^*\,e} &&&& e \\
		\\
		&& {u_!\,v^*\,e} &&&& {u_!\,e} \\
		\\
		{v^*\,u_!\,e}
		\arrow["{f \defeq \xi_!(u,v^*\,e)}", from=1-3, to=3-3, cocart]
		\arrow["{g \defeq \tyfill_f^\xi(kk')}", dashed, from=3-3, to=3-7]
		\arrow["{k' \defeq \pi^*(v,e)}", from=1-3, to=1-7, cart]
		\arrow["{k \defeq \xi_!(u,e)}", from=1-7, to=3-7, cocart]
		\arrow["{g' \defeq \tyfill_{f'}^\pi(kk')}"', dashed, from=1-3, to=5-1]
		\arrow["{f' \defeq \pi^*(v,u_!\,e)}"', from=5-1, to=3-7, cart]
		\arrow["h"{description}, Rightarrow, no head, from=5-1, to=3-3]
	\end{tikzcd}\]
	But since these diagrams commute in any case by the assumptions\\ (recall~\Cref{lem:comp-cart-cocart-transp}) said type is equivalent to the proposition that the filler $h$ is an isomorphism:
	\[\begin{tikzcd}
		&& {u_!\,v^*\,e} \\
		{v^*\,e} && {v^*\,u_!\,e} && {u_!\,e}
		\arrow["h"', dashed, from=1-3, to=2-3]
		\arrow["{f'}"', from=2-3, to=2-5, cart]
		\arrow["g", from=1-3, to=2-5]
		\arrow["{g'}"', from=2-1, to=2-3]
		\arrow["f", from=2-1, to=1-3, cocart]
	\end{tikzcd}\]
	Finally due to the commutation of both of the ``completed squares'' above, this is equivalent to the new alternative criterion: up to identification, $f$ is $\xi$-cocartesian if and only if $g$ is $\pi$-cartesian.
\end{proof}

\begin{figure}
	\centering
	\[\begin{tikzcd}
		{\mathrm{Vert}_\pi(E)} & {v^*\,e} && e & {E \times_{A \times B} \mathrlap{\mathrm{Vert}_q(A \times B)}} \\
		{} & {v^*\,e'} && {e'} && {v^*e} & e & {} & {} \\
		& a && a && a & a && {} \\
		& {a'} && {a'} && {a'} & {a'} && {} \\
		B & b && {b'} & B & b & {b'}
		\arrow["{\pi^*(v,e)}", from=1-2, to=1-4, cocart]
		\arrow["{\mathrm{fill}}"', dashed, from=1-2, to=2-2]
		\arrow["{\pi^*(v,e')}"', from=2-2, to=2-4, cocart]
		\arrow["f", from=1-4, to=2-4]
		\arrow["u"', from=3-2, to=4-2]
		\arrow[Rightarrow, no head, from=3-2, to=3-4]
		\arrow["u", from=3-4, to=4-4]
		\arrow[Rightarrow, no head, from=4-2, to=4-4]
		\arrow["v", from=5-2, to=5-4]
		\arrow["u"', from=3-6, to=4-6]
		\arrow["\psi"{description}, from=1-5, to=5-5]
		\arrow["{\pi'}"{description}, two heads, from=1-1, to=5-1]
		\arrow["{\pi^*(v,e)}", from=2-6, to=2-7, cocart]
		\arrow[Rightarrow, no head, from=3-6, to=3-7]
		\arrow["u", from=3-7, to=4-7]
		\arrow[Rightarrow, no head, from=4-6, to=4-7]
		\arrow["v", from=5-6, to=5-7]
	\end{tikzcd}\]
	\label{fig:lifts-two-sided-fib-lari}\caption{Cartesian lifts in the induced fibrations~$\pi'$~and~$\psi$}
\end{figure}

The following theorem finally contains several characterizations of a two-variable family being two-sided cartesian.\footnote{We are indebted to Emily Riehl for helpful explanations and discussions about~\cite[Thm.~7.1.4]{RV21}.} This consists in several sliced Chevalley/fibered adjoint criteria and a criterion formulated directly on the level of two-variable families.

\begin{theorem}[Char.~of two-sided families, {\protect\cite[Thm.~7.1.4]{RV21}}]\label{thm:char-two-sid}
	For a family $P:A \to B \to \UU$, corresponding to $\varphi \defeq \pair{\xi}{\pi}: E \defeq \sum_{a:A,b:B} P(a,b) \to A \times B$, the following are equivalent:
	\begin{enumerate}
		\item\label{it:char-two-sid-i} The two-variable family $P$ is two-sided.
		\item\label{it:char-two-sid-ii} Considering
		\[\begin{tikzcd}
			E && {A \times B} \\
			& B
			\arrow["\varphi", from=1-1, to=1-3]
			\arrow["\pi"', two heads, from=1-1, to=2-2]
			\arrow["q", two heads, from=1-3, to=2-2]
		\end{tikzcd}\]
		the map $\pi$ 
		is a cartesian fibration, the functor $\varphi$ is cartesian and a \emph{cocartesian} fibration sliced over $B$. Furthermore, the fibered LARI
		\[ \chi_B: \VertArr_q(A \times B) \times_{A\times B} E \to_B \VertArr_\pi(E)\]
		is a \emph{cartesian} functor.
		\item\label{it:char-two-sid-iii'} Considering
		\[\begin{tikzcd}
			E && {A \times B} \\
			& B
			\arrow["\varphi", from=1-1, to=1-3]
			\arrow["\pi"', two heads, from=1-1, to=2-2]
			\arrow["q", two heads, from=1-3, to=2-2]
		\end{tikzcd}\]
		the map $\pi$ 
		is a cartesian fibration, the functor $\varphi$ is cartesian and a \emph{cocartesian} fibration sliced over $B$. Furthermore, the fibered LARI
		\[ \tau_B: \VertArr_q(A \times B) \times_{A\times B} E \to_B E\]
		is a \emph{cartesian} functor.
		\item\label{it:char-two-sid-iii} Considering
		\[\begin{tikzcd}
			E && {A \times B} \\
			& A
			\arrow["\varphi", from=1-1, to=1-3]
			\arrow["\xi"', two heads, from=1-1, to=2-2]
			\arrow["p", two heads, from=1-3, to=2-2]
		\end{tikzcd}\]
		the map $\xi$
		is a cocartesian fibration, the functor $\varphi$ is cocartesian and a \emph{cartesian} fibration sliced over $A$. Furthermore, the fibered RARI
		\[ \chi^A: \VertArr_p(A \times B) \times_{A\times B} E \to_A \VertArr_\xi(E)\]
		is a \emph{cocartesian} functor.
		\item\label{it:char-two-sid-iv}
		The fibered adjoints in the following diagram exist:
		\[\begin{tikzcd}
			E &&& {B \downarrow \pi} \\
			\\
			&& {\xi \downarrow A} &&& {\xi \downarrow A \times_E B \downarrow \pi} \\
			\\
			&&& {A\times B}
			\arrow[""{name=0, anchor=center, inner sep=0}, "{\iota^\pi}"{description}, from=1-1, to=1-4]
			\arrow[""{name=1, anchor=center, inner sep=0}, "{\iota_\xi}"{description}, from=1-1, to=3-3]
			\arrow[""{name=3, anchor=center, inner sep=0}, "{\langle \iota_\xi \circ \partial_1',\id \rangle}"{description, pos=0.25}, from=1-4, to=3-6]
			\arrow["{\langle \partial_1, \pi \, \partial_0 \rangle}"{description}, two heads, from=3-3, to=5-4]
			\arrow["{\langle \partial_1, \partial_0\rangle}"{description}, two heads, curve={height=-15pt}, from=3-6, to=5-4]
			\arrow["{\langle \xi\,\partial_1, \partial_0 \rangle}"{description, pos=0.3}, two heads, from=1-4, to=5-4]
			\arrow["\varphi"{description}, shift right=2, curve={height=30pt}, two heads, from=1-1, to=5-4]
			\arrow[""{name=4, anchor=center, inner sep=0}, "{\tau_\xi}"{description}, curve={height=-18pt}, dashed, from=3-3, to=1-1]
			\arrow[""{name=5, anchor=center, inner sep=0}, "\ell"{description}, curve={height=18pt}, dashed, from=3-6, to=1-4]
			\arrow[""{name=7, anchor=center, inner sep=0}, "{\tau^\pi}"{description}, curve={height=18pt}, dashed, from=1-4, to=1-1]
			\arrow["\dashv"{anchor=center, rotate=-91}, draw=none, from=7, to=0]
			\arrow["\dashv"{anchor=center, rotate=-133}, draw=none, from=5, to=3]
			\arrow["\dashv"{anchor=center, rotate=44}, draw=none, from=4, to=1]
			\arrow[""{name=6, anchor=center, inner sep=0}, "r"{description}, curve={height=-24pt}, dashed, from=3-6, to=3-3, crossing over]
			\arrow[""{name=2, anchor=center, inner sep=0}, "{\langle \id, \iota^\pi \circ \partial_0' \rangle}"{description, pos=0.6}, from=3-3, to=3-6, crossing over]
			\arrow["\dashv"{anchor=center, rotate=108}, draw=none, from=6, to=2]
		\end{tikzcd}\]
		where the pullback type is given by:
		\[\begin{tikzcd}
			{\comma{\xi}{A} \times_E \comma{B}{\pi}} && {\comma{\xi}{A}} \\
			{\comma{B}{\pi}} && E
			\arrow[from=1-1, to=2-1]
			\arrow["{\partial_1}"', from=2-1, to=2-3]
			\arrow[from=1-1, to=1-3]
			\arrow["{\partial_0'}", from=1-3, to=2-3]
			\arrow["\lrcorner"{anchor=center, pos=0.125}, draw=none, from=1-1, to=2-3]
		\end{tikzcd}\]
		Moreover, the mate of the identity $2$-cell defines a fibered isomorphism
		\[ \prod_{a:A, b:B} (\tau_\xi \circ r)_{a,b} =_{Q(a,b) \to P(a,b)} (\tau^\pi \circ \ell)_{a,b},\]
		where
		\[ Q: A \times B \to \UU, \quad Q(a,b) \simeq \comma{a}{A} \times \comma{B}{b} \times P(a,b)\]
		is the straightening of the map
		\[ F \defeq \comma{\xi}{A} \times_E \comma{B}{\pi} \fibarr A \times B.\]
	\end{enumerate}
\end{theorem}

\begin{proof}

	\begin{description}
		\item[$\ref{it:char-two-sid-i} \iff \ref{it:char-two-sid-ii}$] By \Cref{cor:cocart-left-and-cart-right}, the map $P: A \to B \to \UU$ is cocartesian on the left and cartesian on the right if and only if $\varphi: E \to_B A \times B$ is both a cartesian functor between cartesian fibrations and a cocartesian fibration sliced over $B$.
		
		In the following, we assume this is satisfied for $P$.
		
		Thus, we are left to show that, under this assumption---$P$ being cocartesian on the left and cartesian on the right---the following holds:
		\begin{align*} 
		& \text{``The fibered LARI $\chi_B: \VertArr_q(A \times B) \times_{A\times B} E \to_B \VertArr_\pi(E)$} \\ 
		& \text{is a cocartesian functor.''} \\
			& \iff  \text{``In $P$, cocartesian and cartesian lifts commute.''}
		\end{align*}
		We write $F \defeq \mathrm{Vert}_q(A \times B) \times_{A \times B} E \fibarr B$, so by projection equivalence, we consider the projection
		\[ \psi : \sum_{b:B} \sum_{a:A} (\comma{a}{A} \times P(a,b)) \simeq F \fibarr B. \]		
		The induced sliced Leibniz cotensor is given by
		\[ \kappa : \VertArr_\pi(E) \to_B F,\quad \kappa_b\big(u:\Delta^1 \to A,f:\prod_{t:\Delta^1} P(u(t),b) \big) \defeq \angled{\partial_0\,u,u, \partial_0\,f}. \]
		It  has a fibered LARI $\mu: F \to_B \VertArr_\pi(E)$ as indicated in: 
		\[\begin{tikzcd}
			{\mathrm{Vert}_\pi(E)} && {\mathrm{Vert}_q(A \times B) \times_{A \times B} E} \\
			\\
			& B
			\arrow[""{name=0, anchor=center, inner sep=0}, "\kappa"{description}, curve={height=12pt}, from=1-1, to=1-3]
			\arrow[""{name=1, anchor=center, inner sep=0}, "\mu"{description}, curve={height=18pt}, dashed, from=1-3, to=1-1]
			\arrow["{\pi'}"{description}, two heads, from=1-1, to=3-2]
			\arrow["\psi"{description}, two heads, from=1-3, to=3-2]
			\arrow["\dashv"{anchor=center, rotate=-94}, draw=none, from=1, to=0]
		\end{tikzcd}\]
		
		\begin{figure}
			\centering
			\[\begin{tikzcd}
				{e''} & e & {e''} & e && {e'} & e & {e'} & e \\
				{e'''} & {e'} &&&&&& {u_!\,e'} & {u_!\,e} \\
				{a''} & a & {a''} & a && {a''} & a & {a''} & a \\
				{a'''} & {a'} & {a'''} & {a'} && {a'''} & {a'} & {a'''} & {a'} \\
				{b'} & b & {b'} & b && {b'} & b & {b'} & b
				\arrow["g", from=1-1, to=1-2]
				\arrow["{f'}"', from=1-1, to=2-1]
				\arrow["f", from=1-2, to=2-2]
				\arrow["{g'}"', from=2-1, to=2-2]
				\arrow["m", from=3-1, to=3-2]
				\arrow["{u'}", from=3-1, to=4-1]
				\arrow[""{name=0, anchor=center, inner sep=0}, "u"', from=3-2, to=4-2]
				\arrow["v", from=5-1, to=5-2]
				\arrow["g", from=1-3, to=1-4]
				\arrow[""{name=1, anchor=center, inner sep=0}, "{u'}", from=3-3, to=4-3]
				\arrow[Rightarrow, dotted, no head, from=1-3, to=3-3]
				\arrow[Rightarrow, dotted, no head, from=1-4, to=3-4]
				\arrow["u"', from=3-4, to=4-4]
				\arrow["m", from=3-3, to=3-4]
				\arrow["{m'}"', from=4-3, to=4-4]
				\arrow["v", from=5-3, to=5-4]
				\arrow["{m'}"', from=4-1, to=4-2]
				\arrow["v", from=5-8, to=5-9]
				\arrow["u", from=3-6, to=4-6]
				\arrow[from=3-6, to=3-7]
				\arrow[from=4-6, to=4-7]
				\arrow[""{name=2, anchor=center, inner sep=0}, "u"', from=3-7, to=4-7]
				\arrow["f", from=1-6, to=1-7]
				\arrow[Rightarrow, dotted, no head, from=1-6, to=3-6]
				\arrow[Rightarrow, dotted, no head, from=1-7, to=3-7]
				\arrow["{\xi_!(u,e')}"', from=1-8, to=2-8, cocart]
				\arrow["f", from=1-8, to=1-9]
				\arrow["{\xi_!(u,e)}", from=1-9, to=2-9, cocart]
				\arrow["{\mathrm{fill}}"', dashed, from=2-8, to=2-9]
				\arrow[""{name=3, anchor=center, inner sep=0}, "u", from=3-8, to=4-8]
				\arrow[from=3-8, to=3-9]
				\arrow["u"', from=3-9, to=4-9]
				\arrow[from=4-8, to=4-9]
				\arrow["v", from=5-6, to=5-7]
				\arrow["\kappa", shorten <=6pt, shorten >=6pt, maps to, from=0, to=1]
				\arrow["\mu", shorten <=6pt, shorten >=6pt, maps to, from=2, to=3]
			\end{tikzcd}\]
			\caption{Action on arrows of the fibered functors $\kappa: F\to_B \VertArr_\pi(E)  : \mu$}\label{fig:kappa-act-arrows}
		\end{figure}
		
		By our discussion of cocartesian-on-the-left fibrations, \cf~\Cref{prop:char-fib-cocart-left}, the fibered LARI $\mu:F \to \VertArr_\pi(E)$ at~$b:B$ is given by
		\[ \mu_b(a:A,u:\comma{a}{A},e:P(a,b)) \defeq \pair{u:\comma{a}{A}} {\xi_!(u,e):e \to_{\pair{u}{b}}^P u_!\,e}. \]
		By the closure properties of cartesian fibrations~\cite[Corollary~5.2.10, Proposition~5.2.14]{BW21} we obtain that the pulled back maps 
		\[\pi' \defeq \cst^*\pi^{\Delta^1}: \VertArr_\pi(E) \fibarr B\]
		and
		\[q' \defeq \cst^*q:\VertArr_q(A \times B) \fibarr B\]
		are cartesian fibrations. Moreover, by \emph{op.~cit.}, Proposition~5.3.10, so is $\psi \defeq q' \times_q \pi : F \fibarr B$. By the computations of lifts, as elobarated in \emph{op.~cit.}, Propositions~5.2.9, 5.3.9,~and 5.3.10, the cartesian lifts in $\pi': \VertArr_\pi(E) \fibarr B$ and, resp.~$\psi: F \fibarr B$ are given by as follows (\cf~\Cref{fig:cart-ness-of-mu} for illustration):
		\begin{align*}
			& (\pi')^*\left( v:b \to_B b', \pair{u:a \to_A a'}{f:e \to^P_{\pair{u}{a}} e'}\right) \\
			& = \left \langle v, \pair{\id_a}{\id_a'} : u \rightrightarrows_A u, \pair{\pi^*(v,e)}{\pi^*(v,e')}: \tyfill \rightrightarrows^P f \right \rangle \\
			& \\
			& \psi^*\left( v:b\to_B b', \pair{u:a \to_A a'}{e:P(a,b)} \right) \\
			& = \left \langle v, \langle \pair{\id_a}{\id_a'}: u \rightrightarrows_A u, \pi^*(v,e): v^* \cartarr^\pi_v e \right \rangle
		\end{align*}
		Note that, instead of using the formulas for the lifts, one can also directly verify that the given maps are indeed cartesian. Since for $\psi$ this is straightforward to see we only discuss the case of $\pi': \VertArr_\pi(E) \fibarr B$. Consider probing maps as indicated in:
		\[\begin{tikzcd}
			d && {v^*\,e} && e \\
			{d'} && {v^*\,e'} && {e'} \\
			{a''} && a && a \\
			{a'''} && {a'} && {a'} \\
			b && b && b
			\arrow["v"{description}, from=5-3, to=5-5]
			\arrow["{v'}"{description}, from=5-1, to=5-3]
			\arrow["{v'v}"{description}, curve={height=18pt}, from=5-1, to=5-5]
			\arrow[Rightarrow, no head, from=4-3, to=4-5, shorten <=9pt, shorten >=9pt]
			\arrow["{m'}"{description}, dashed, from=4-1, to=4-3]
			\arrow["m"{description, pos=0.7}, curve={height=18pt}, from=4-1, to=4-5]
			\arrow["{u'}"', from=3-1, to=4-1]
			\arrow["u"', from=3-3, to=4-3]
			\arrow[Rightarrow, no head, from=3-3, to=3-5, shorten <=9pt, shorten >=9pt]
			\arrow["u", from=3-5, to=4-5]
			\arrow["{f'}"', from=1-1, to=2-1]
			\arrow["{g'}"{description}, from=2-1, to=2-3]
			\arrow[""{name=0, anchor=center, inner sep=0}, "g"{description}, from=1-1, to=1-3]
			\arrow["{k :\jdeq \mathrm{fill}}"', dashed, from=1-3, to=2-3]
			\arrow["\ell", from=1-3, to=1-5]
			\arrow[""{name=1, anchor=center, inner sep=0}, "{\ell'}", from=2-3, to=2-5]
			\arrow["f", from=1-5, to=2-5]
			\arrow[""{name=2, anchor=center, inner sep=0}, "{h'}"{description}, curve={height=18pt}, from=2-1, to=2-5]
			\arrow[""{name=3, anchor=center, inner sep=0}, "h"{description}, curve={height=-18pt}, from=1-1, to=1-5]
			\arrow[Rightarrow, no head, from=2-3, to=1-5, shorten <=9pt, shorten >=9pt]
			\arrow["{(?)}"{description}, draw=none, from=2-1, to=1-3]
			\arrow[shorten <=9pt, shorten >=9pt, Rightarrow, no head, from=2, to=1]
			\arrow[shorten <=8pt, shorten >=8pt, Rightarrow, no head, from=0, to=3]
			\arrow["m"{description}, dashed, from=3-1, to=3-3]
			\arrow["m"{description, pos=0.7}, curve={height=-16pt}, from=3-1, to=3-5, crossing over]
		\end{tikzcd}\]
		By the property of the $\pi$-cartesian lifts the two triangles and the right hand square commute as indicated. For the square in question on the left hand side we employ a line of reasoning familiar from fibered $1$-category theory: to give a homotopy $kg = g'f'$ it suffices to show that the $\pi$-cartesian arrow $\ell'$ equalizes both composite arrows. Note that from a path $\ell'k = f\ell$ we obtain a chain of homotopies
		\[ (\ell' k)g = (f\ell)g = fh = h'f' = (\ell' g')f'\]
		as desired. Hence the whole diagram $(?)$ commutes.
		
		Now, the fibered transport functor $\mu:F \to \VertArr_\pi(E)$ is cartesian if and only if it maps $\psi$-cartesian arrows to $\pi'$-cartesian arrows. Its action on $\psi$-cartesian arrows is given by
		\begin{align*}
			& \mu_v(\psi^*(v,\pair{u}{e}))  = \mu_v\big(\pair{\id_a}{\id_{a'}}:u \rightrightarrows_A u, \pi^*(v,e)\big) \\
			= & \Big \langle \pair{\id_a}{\id_{a'}}:u \rightrightarrows_A u, \\
			&  \big \langle \pi^*(v,e), \tyfill^\xi_{\xi_!(u,v^*e)}(\xi_!(u,e) \circ \pi^*(v,e))\big \rangle :\xi_!(u,v^*e) \rightrightarrows^P \xi_!(u,e) \Big \rangle,
		\end{align*}
		for $u:a \to_A a'$, $v:b \to_B  b'$, $e:P(a,b)$.
		Conversely, $\pi'$-cartesian lifts of $\mu$-images are of the form
		\begin{align*}
			& (\pi')^*(v,\mu_b(u,e))  = (\pi')^*(v, \pair{u}{\xi_!(u,e)}) \\
			= &	\Big \langle \pair{\id_a}{\id_{a'}} : u \rightrightarrows_A u,\\
			&  \pair{\pi^*(v,e)}{\pi^*(v,u_!\,e)}: \tyfill^\pi_{\pi^*(v,u_!\,e)}(\xi_!(u,e) \circ \pi^*(v,e)) \rightrightarrows^P \xi_!(u,e)\Big \rangle.
		\end{align*}
		But having an identification between those squares is exactly equivalent to the commutation condition, by~\Cref{prop:comm-lifts}.

		\begin{figure}
			\[\begin{tikzcd}
				{v^*\,e} & e & {v^*\,e} & e &&& e & {v^*\,e} & e \\
				&& {u_!\,v^*\,e} & {u_!\,e} &&& {u_!\,e} & {v^*\,u_!\,e} & {u_!\,e} \\
				a & a & a & a &&& a & a & a \\
				{a'} & {a'} & {a'} & {a'} &&& {a'} & {a'} & {a'} \\
				b & {b'} & b & {b'} && b & {b'} & b & {b'}
				\arrow["u", from=3-1, to=4-1]
				\arrow[Rightarrow, no head, from=4-1, to=4-2]
				\arrow[Rightarrow, no head, from=3-1, to=3-2]
				\arrow[""{name=0, anchor=center, inner sep=0}, "u"', from=3-2, to=4-2]
				\arrow["v", from=5-1, to=5-2]
				\arrow["{\pi^*(v,e)}", from=1-1, to=1-2, cart]
				\arrow["{\pi^*(v,e)}", from=1-3, to=1-4, cart]
				\arrow["{\pi_!(u,v^*\,e)}"' swap, from=1-3, to=2-3, cocart]
				\arrow["{\xi_!(u,e)}", from=1-4, to=2-4, cocart]
				\arrow["{\mathrm{fill}}", dashed, from=2-3, to=2-4]
				\arrow["v", from=5-3, to=5-4]
				\arrow[""{name=1, anchor=center, inner sep=0}, "u", from=3-3, to=4-3]
				\arrow[Rightarrow, no head, from=4-3, to=4-4]
				\arrow[Rightarrow, no head, from=3-3, to=3-4]
				\arrow["u"', from=3-4, to=4-4]
				\arrow[Rightarrow, dotted, no head, from=1-1, to=3-1]
				\arrow[Rightarrow, dotted, no head, from=1-2, to=3-2]
				\arrow["v", from=5-6, to=5-7]
				\arrow[""{name=2, anchor=center, inner sep=0}, "u", from=3-8, to=4-8]
				\arrow[Rightarrow, no head, from=4-8, to=4-9]
				\arrow[Rightarrow, no head, from=3-8, to=3-9]
				\arrow["u"', from=3-9, to=4-9]
				\arrow["{\mathrm{fill}}"', dashed, from=1-8, to=2-8]
				\arrow["{\pi^*(v,u_!\,e)}", from=2-8, to=2-9, cart]
				\arrow["{\xi_!(u,e)}" swap, from=1-9, to=2-9, cocart]
				\arrow["{\pi^*(v,e)}", from=1-8, to=1-9, cart]
				\arrow["v", from=5-8, to=5-9]
				\arrow[""{name=3, anchor=center, inner sep=0}, "u"', from=3-7, to=4-7]
				\arrow["{\xi_!(u,e)}"', from=1-7, to=2-7, cocart]
				\arrow["{\mu_v}", shorten <=7pt, shorten >=7pt, maps to, from=0, to=1]
				\arrow["{(\pi')^*(v,-)}", shorten <=7pt, shorten >=7pt, maps to, from=3, to=2]
			\end{tikzcd}\]
			\caption{Cartesianness of fibered lifting functor}\label{fig:cart-ness-of-mu}
		\end{figure}
		\item[$\ref{it:char-two-sid-ii} \iff \ref{it:char-two-sid-iii'}$] This follows from the characterization~\Cref{prop:char-cocart-fib-in-cart-fib} of cocartesian families in cartesian families, namely the equivalence of the conditions from~\Cref{it:gen-lift-comm-i} and~\Cref{it:gen-lift-comm-ii}.
		\item[$\ref{it:char-two-sid-i} \iff \ref{it:char-two-sid-iii}$] This is dual to the previous case.
		\item[$\ref{it:char-two-sid-ii} \iff \ref{it:char-two-sid-iv}$] We adapt the proof of~\cite[Theorem~7.1.4]{RV21}.
		First, observe that we have
		\[ \comma{\xi}{A} \times_E \comma{B}{\pi} \simeq \sum_{\substack{a:A \\ b:B}} \comma{a}{A} \times \comma{B}{b} \times P(a,b).\]
		Now, the assumption of $\varphi$ being cocartesian on the left and cartesian on the right is equivalent to the existence of the fibered adjoints $\tau_\xi$ and $\tau^\pi$. We will only write down the steps starting with $\tau_\xi$, since the case of $\tau^\pi$ is dual. By pulling back the fibered LARI adjunction $\tau_\xi \dashv_{A \times B} \iota_\xi$, we obtain
		\[\begin{tikzcd}
			{B \downarrow \pi} & {} &&& E \\
			& {\xi \downarrow A \times_E B \downarrow \pi} &&&& {\xi \downarrow A} \\
			{A \times B^{\Delta^1}} &&&& {A \times B}
			\arrow[from=1-1, to=1-5]
			\arrow["{\lambda u,v,e.\langle u1,v\rangle}"{description}, from=2-2, to=3-1]
			\arrow[""{name=0, anchor=center, inner sep=0}, "{\iota_\xi}"{description}, from=1-5, to=2-6]
			\arrow["\varphi"{description, pos=0.3}, from=1-5, to=3-5]
			\arrow["{\langle \partial_1, \pi \partial_0\rangle}"{description}, from=2-6, to=3-5]
			\arrow["{\id_A \times \partial_1}"{description}, from=3-1, to=3-5]
			\arrow["{\lambda a',v,e.\langle a',v,e \rangle}"{description}, from=1-1, to=3-1]
			\arrow[""{name=1, anchor=center, inner sep=0}, "{\tau_\xi}"{description}, curve={height=18pt}, dotted, from=2-6, to=1-5]
			\arrow["\lrcorner"{anchor=center, pos=0.125}, draw=none, from=2-2, to=3-5]
			\arrow[""{name=2, anchor=center, inner sep=0}, "{\langle \tau_\xi \circ \partial_1, \id \rangle}"{description}, from=1-1, to=2-2]
			\arrow[""{name=3, anchor=center, inner sep=0}, "\ell"{description}, curve={height=24pt}, dotted, from=2-2, to=1-1]
			\arrow["\lrcorner"{anchor=center, pos=0.125}, shift right=3, draw=none, from=1-2, to=3-5]
			\arrow["\dashv"{anchor=center, rotate=-136}, draw=none, from=1, to=0]
			\arrow["\dashv"{anchor=center, rotate=-127}, draw=none, from=3, to=2]
			\arrow[from=2-2, to=2-6, crossing over]
		\end{tikzcd}\]
		with projection equivalences
		\begin{align*}
			\xi \downarrow A & \simeq \sum_{\substack{a':A \\ b:B}} \sum_{u:A \downarrow a'} P(u0,b), &  E & \simeq \sum_{\substack{a':A \\ b:B}} P(a',b), \\
			A \times B^{\Delta^1} & \simeq \sum_{\substack{a':A \\ b:B}} B\downarrow b, &  B \downarrow \pi & \simeq \sum_{\substack{a':A \\ b:B}} B \downarrow b \times P(a',b). 
		\end{align*}
		Postcomposition with $\id_A : A \times B^{\Delta^1} \to A \times B$ preserves the fibered adjunction, yielding as desired:
		\[\begin{tikzcd}
			{B \downarrow \pi} &&&& {\xi \downarrow A \times_E B \downarrow \pi} \\
			\\
			&& {A \times B^{\Delta^1}} \\
			\\
			&& {A \times B}
			\arrow[""{name=0, anchor=center, inner sep=0}, "{\langle \tau^\xi \circ \partial_1 ,\id \rangle}"{description}, curve={height=12pt}, from=1-1, to=1-5]
			\arrow[from=1-1, to=3-3]
			\arrow[from=1-5, to=3-3]
			\arrow["{\mathrm{id}_A \times \partial_0}"{description}, two heads, from=3-3, to=5-3]
			\arrow["{\langle \xi \partial_1, \partial_0\rangle}"{description}, two heads, from=1-1, to=5-3]
			\arrow["{\langle \partial_1, \partial_0\rangle}"{description}, two heads, from=1-5, to=5-3]
			\arrow[""{name=1, anchor=center, inner sep=0}, "\ell"{description}, curve={height=12pt}, dashed, from=1-5, to=1-1]
			\arrow["\dashv"{anchor=center, rotate=-90}, draw=none, from=1, to=0]
		\end{tikzcd}\]
		This, together with the dual case, yields the claimed adjoints in the fibered square of~\Cref{it:char-two-sid-iv}.
		
		In sum, this is equivalent to $\pi:E \fibarr B$ being a cartesian fibration, the fibered functor
		\[\begin{tikzcd}
			{\comma{\xi}{A}} && E \\
			& B
			\arrow["{\tau_\xi}", from=1-1, to=1-3]
			\arrow["{\pi\partial_0}"', two heads, from=1-1, to=2-2]
			\arrow["\pi", two heads, from=1-3, to=2-2]
		\end{tikzcd}\]
		being a cartesian functor and a cocartesian fibration sliced over $B$. The invertibility of the mate is then equivalent to the functor
		\[\begin{tikzcd}
			E && {A \times B} \\
			& B
			\arrow["\varphi", from=1-1, to=1-3]
			\arrow["\pi"', two heads, from=1-1, to=2-2]
			\arrow["q", two heads, from=1-3, to=2-2]
		\end{tikzcd}\]
		being cartesian.
		
		By~\Cref{thm:char-cocart-fun}, this is equivalent to the mate of the identity of
		\[\begin{tikzcd}
			{\comma{\xi}{A}} && E \\
			{\comma{\xi}{A} \times_E \comma{B}{\pi}} && {\comma{B}{\pi}}
			\arrow[from=1-1, to=1-3]
			\arrow[from=1-1, to=2-1]
			\arrow[from=2-1, to=2-3]
			\arrow[from=1-3, to=2-3]
		\end{tikzcd}\]
		being invertible, as
		\[ \comma{B}{\pi\partial_0} \simeq \comma{\xi}{A} \times_E \comma{B}{\pi}.\]
	\end{description}
	
\end{proof}

\section{Two-sided cartesian functors and closure properties}\label{sec:2s-cart-fun}

\begin{definition}[Two-sided cartesian functors, \protect{\cite[Prop.~7.1.7]{RV21}}]
	Let $P,Q:A \to B \to \UU$ be two-sided cartesian families. A fibered map $h: P \to_{A \times B} Q$ is called \emph{two-sided cartesian functor} (or simply \emph{cartesian}) if it constitutes a cocartesian functor $h: P_B \to_A Q_B$ and a cartesian functor $h:P_A \to_B Q_A$.
\end{definition}

An immediate reformulation is that, for all $u:a \to_A a'$ and $v:b' \to_B v$, $e,d:P(a,b)$, we have identities\footnote{In particular, again the types of each of these identities is a proposition, hence so is their product.}
\[ h_{u,b}(P_!(u,b,e)) = Q_!(u,b,h_{a,b}(e)), \quad h_{a,v}(P^*(a,v,e)) = Q^*(a,v,h_{a,b}(e)). \]

We will state versions of the closure properties \wrt~to different bases as well as the sliced or relative versions where the base stays fixed throughout. All the following subsections are oriented along~\cite[Section~7.2]{RV21}.

\subsection{Composition and whiskering}

\begin{proposition}[Composition stability of two-sided cartesian functors]\label{prop:2s-fun-comp}
	Let $A,B,C,D,S$, and $T$ be Rezk types. Assume given two-sided cartesian families $P: A \to B \to \UU$, $Q: C \to D \to \UU$, and $R: S \to T \to \UU$, as well as two-sided cartesian functors $h:P \to_{A \times B,C \times D} Q$, $k:Q \to_{C \times D,S \times T} R$. Then the composite fibered functor $k \circ h: P \to_{A \times B,S \times T} R$ is a two-sided cartesian functor as well.
\end{proposition}

\begin{proof}
	This follows since cartesian and cocartesian functors are both closed under composition, \cf~\cite[Proposition~5.3.6, Item 1]{BW21}.
\end{proof}

\begin{corollary}[Composition stability of two-sided cartesian functors in a slice]\label{prop:2s-fun-comp-sl}
	Let $P,Q,R:A \to B \to \UU$ be two-sided cartesian families, and $h:P \to_{A \times B} Q$, $k:Q \to_{A \times B} R$ cartesian functors. Then the composite fibered functor $k \circ h: P \to_{A \times B} R$ is a two-sided cartesian functor as well.
\end{corollary}

\begin{proposition}[Whiskering with co-/cart.~fibrations~\protect{\cite[Lem.~7.2.5]{RV21}}]\label{prop:2s-cart-comp-prod-fib}
	Let $A,B,C,D$ be Rezk types. Assume $\varphi \defeq \pair{\xi}{\pi}:E \fibarr A \times B$ is a two-sided cartesian fibration. If $k:A \fibarr C$ is a cocartesian fibration and $m:B \fibarr D$ is a cartesian fibration, then the composite
	\[\begin{tikzcd}
		E && {A \times B} && {C \times D}
		\arrow["{\langle \xi,\pi \rangle}", two heads, from=1-1, to=1-3]
		\arrow["{k \times m}", two heads, from=1-3, to=1-5]
	\end{tikzcd}\]	
	is a two-sided cartesian fibration as well.
\end{proposition}

\begin{proof}
	We argue as in~\cite[Lemma~7.2.5]{RV21}. By the characterization of two-sided cartesian fibrations via cocartesian fibrations in cartesian fibrations~ \Cref{thm:char-two-sid}, \Cref{it:char-two-sid-ii}, we reason as follows. Since $k:A \fibarr C$ and $\id_B: B \fibarr B$ both are two-sided cartesian fibrations as well also their cartesian product is, by~\Cref{prop:2s-cart-closed-pi}. Hence, the following fibered maps are cocartesian fibrations in cartesian fibrations
	\[\begin{tikzcd}
		E && {A \times B} & {A \times B} && {C \times B} \\
		& B &&& B
		\arrow["\varphi", from=1-1, to=1-3]
		\arrow["\pi"', two heads, from=1-1, to=2-2]
		\arrow["q", two heads, from=1-3, to=2-2]
		\arrow["{k \times \id_B}", from=1-4, to=1-6]
		\arrow["q"', two heads, from=1-4, to=2-5]
		\arrow[two heads, from=1-6, to=2-5]
	\end{tikzcd}\]
	and so is their horizontal composite $(k \times \id_B) \circ \varphi: E \fibarr_B C \times B$.
	
	One can argue similarly for the case $\id_C: C \fibarr C$ and $m:B \fibarr D$, which establishes the claim.
\end{proof}

\begin{figure}
	\[\begin{tikzcd}
		P & e & {u_*\,e} & {e'} && d & d & {d'} & Q \\
		&& B & b & b & {b'} \\
		A & a & {a'} & {a''} && c & c & {c'} & C
		\arrow["u", from=3-2, to=3-3]
		\arrow["{\forall \,u'}", from=3-3, to=3-4]
		\arrow["g"', dashed, from=1-3, to=1-4]
		\arrow["{\id_d}"', Rightarrow, no head, from=1-6, to=1-7]
		\arrow["r"', dashed, from=1-7, to=1-8]
		\arrow["{\id_b}", Rightarrow, no head, from=2-4, to=2-5]
		\arrow["v", dashed, from=2-5, to=2-6]
		\arrow["{\id_c}", Rightarrow, no head, from=3-6, to=3-7]
		\arrow["{\forall \,w}", from=3-7, to=3-8]
		\arrow["{u'u}"{description}, curve={height=18pt}, from=3-2, to=3-4]
		\arrow["{\forall\,v}"{description}, curve={height=18pt}, from=2-4, to=2-6]
		\arrow["{\forall \,r}"{description}, curve={height=-18pt}, from=1-6, to=1-8]
		\arrow["{\forall \,h}"{description}, curve={height=-18pt}, from=1-2, to=1-4]
		\arrow["f"', from=1-2, to=1-3, cocart]
		\arrow["w"{description}, curve={height=18pt}, from=3-6, to=3-8]
	\end{tikzcd}\]
	\caption{Cocartesian lift and universal property in the span composite}
	\label{fig:cocart-lift-spcomp}
\end{figure}

\begin{proposition}[Span composition of two-sided cartesian fibrations, \cf~\protect{\cite[Proposition~7.2.6]{RV21}}]
	Let $P:A\to B \to \UU$, $Q:B \to C \to \UU$ be two-sided cartesian families over Rezk types $A$, $B$, $C$. Then the family defined by \emph{span composition}
	\[ Q \spancomp P \defeq \lambda a,c.\sum_{b:B} P\,a\,b \times Q\,b\,c: A \to C \to \UU \]
	is also two-sided cartesian.
	
	In particular, the cocartesian and cartesian lifts, resp., are given as follows: For $u:a \to_A a'$, $w:c'\to_C c$, $b:B$, $e:P(a,b)$, and $d:Q(b,c)$ we have
	\begin{align*}
		(Q \spancomp P)_!(u:a \to_A a', \id_c,\angled{b,e,d}) & \defeq \angled{u,\id_c,\id_b, P_!(u,b,e), \id_d}, \\
		(Q \spancomp P)^*(\id_a, w,\angled{b,e,d}) & \defeq \angled{\id_a,w,\id_b, \id_e, Q^*(b,w,d)}.
	\end{align*}
\end{proposition}

\begin{proof}
	We can argue on the level of fibrations just as in \cite[Proposition~7.2.6]{RV21}. Let $\varphi \defeq \pair{\xi}{\pi} \defeq \Un_{A,B}(P): E \fibarr A \times B$, $\psi \defeq \pair{\kappa}{\mu} \defeq \Un_{B,C}(Q): E \fibarr B \times C$. The unstraightening of $Q \spancomp P$ corresponds to the composite
	\[\begin{tikzcd}
		{E \times _B F} && {E \times F} && {A \times C}
		\arrow["{\langle q,p\rangle}", from=1-1, to=1-3]
		\arrow["{\langle \xi,\mu\rangle}", from=1-3, to=1-5]
		\arrow["{\psi \boxdot \varphi}"{description}, curve={height=-24pt}, from=1-1, to=1-5]
	\end{tikzcd}\]
	where $q:E \times_B F \to E$ and $p:E \times_B F \to F$ are the projections from the pullback object. Now, the map $\psi \spancomp \varphi:A \to C \to \UU$ is constructed by first taking the pullback
	\[\begin{tikzcd}
		{E \times_B F} && E \\
		{A \times F} && {A \times B}
		\arrow["{\id_A \times \kappa}"', from=2-1, to=2-3]
		\arrow[from=1-1, to=1-3]
		\arrow["\varphi", two heads, from=1-3, to=2-3]
		\arrow["{\langle \xi \circ q,p \rangle}"', two heads, from=1-1, to=2-1]
		\arrow["\lrcorner"{anchor=center, pos=0.125}, draw=none, from=1-1, to=2-3]
	\end{tikzcd}\]
	and then postcomposing the map on $\pair{\xi q}{p}:{E \times_B F} \fibarr A \times F$ with $\id_A \times \mu: A \times F \to A \times C$. Pullback along products of maps preserves two-sided cartesian fibrations by~\Cref{prop:2s-cart-closed-pb}, and so does postcomposition with the cartesian product of a cocartesian and a cartesian fibration by~\Cref{prop:2s-cart-comp-prod-fib}. Hence, the resulting map $\psi \spancomp \varphi: E \times_B F \fibarr A \times C$ is two-sided cartesian as well.
	
	The proclaimed description of the co-/cartesian lifts comes out of this construction, using the descriptions of the lifts from the constructions in~\cite[Subsections~3.2.4 and~5.3.3]{BW21}. Alternatively, one can verify the universal property directly, \cf~\Cref{fig:cocart-lift-spcomp}. \Eg, for the cocartesian case, given any $u':a'\to_A a''$, $w:c \to_C c'$, an arrow lying over the (component-wise) composite with domain $\angled{b,e,d}$ consists of some arrow $v:b \to b'$ and dependent arrows $f:e \to^P_{\pair{u'u}{v}} e'$, $r:d \to^Q_{\pair{v}{w}} d'$. By initiality, $v:\id_b \to v$ is the unique filler in the comma object $\comma{b}{B}$, and so is $r:\id_d \to r$, lying over $v:\id_b \to v$ and $w:\id_c \to c'$. By cocartesianness, we also find~\wrt~the data given the unique filler $g \defeq \cocartFill_{P_*(u,b,e)}(f)$ with $g \circ f = h$, for $f \defeq P_*(u,b,e):e \cocartarr^P_{\pair{u}{b}}u_*\,e$ in $P$ as desired. 
\end{proof}

\subsection{Pullback and reindexing}

\begin{proposition}[Pullback stability of two-sided cartesian families, \protect{\cf~\cite[Proposition~7.2.4]{RV21}}]\label{prop:2s-cart-closed-pb}
	Let $P:A \to B \to \UU$ be a two-sided cartesian family over Rezk types $A$ and $B$. Then for any pair of maps $k:C \to A$, $m:D \to B$, the pullback family
	\[ (k \times m)^*P: C \to D \to \UU \]
	is two-sided as well.
	Diagrammatically, if the two-sided fibration $\varphi: E \fibarr A \times B$ denotes the unstraightening of $P$, this means that the map $\psi$ in the following diagram is a two-sided fibration:
	\[\begin{tikzcd}
		{(k \times m)^*E} && E \\
		{C \times D} && {A \times B}
		\arrow["{(k \times m)^*\varphi}"', two heads, from=1-1, to=2-1]
		\arrow["{k \times m}"', from=2-1, to=2-3]
		\arrow[from=1-1, to=1-3]
		\arrow["\varphi", two heads, from=1-3, to=2-3]
		\arrow["\lrcorner"{anchor=center, pos=0.125}, draw=none, from=1-1, to=2-3]
	\end{tikzcd}\]
	Furthermore, we claim that this square is a two-sided cartesian functor.
\end{proposition}

Recalling the notation from~\Cref{ssec:two-var}, we write
\begin{align*}
	& P_m:A \to \UU, & P_m(a) & \defeq \sum_{b':B'} P_{m\,b'}(a) \jdeq \sum_{b':B'} P(a,m\,b'), \\
	& P^k:B \to \UU, & P^k(b) & \defeq \sum_{a':A'} P^{k\,a'}(b) \jdeq \sum_{a':A'} P(k\,a',b), \\
\end{align*}
\[\begin{tikzcd}
	{m^*\widetilde{P}} && {\widetilde{P}} && {k^*\widetilde{P}} && {\widetilde{P}} \\
	{A \times B'} && {A \times B} && {A' \times B} && {A \times B} \\
	A &&&& B
	\arrow[two heads, from=1-5, to=2-5]
	\arrow["{k \times \id_B}"', from=2-5, to=2-7]
	\arrow[two heads, from=1-7, to=2-7]
	\arrow[two heads, from=2-5, to=3-5]
	\arrow[from=1-5, to=1-7]
	\arrow["\lrcorner"{anchor=center, pos=0.125}, draw=none, from=1-5, to=2-7]
	\arrow["{\Un_B(P^k)}"'{pos=0.2}, curve={height=30pt}, from=1-5, to=3-5]
	\arrow[two heads, from=1-3, to=2-3]
	\arrow[from=1-1, to=1-3]
	\arrow[two heads, from=1-1, to=2-1]
	\arrow[two heads, from=2-1, to=3-1]
	\arrow["{\Un_A(P_m)}"'{pos=0.2}, curve={height=30pt}, two heads, from=1-1, to=3-1]
	\arrow["{\id_A \times m}"', from=2-1, to=2-3]
	\arrow["\lrcorner"{anchor=center, pos=0.125}, draw=none, from=1-1, to=2-3]
\end{tikzcd}\]
In particular, for the case of $k= \id_A: A \to A$ and $m=\id_B : B \to B$ we have $P^A = P^{\id_A}$ and $P_B = P_{\id_B}$, \cf~\Cref{ssec:two-var}.

\begin{proof}
	This follows by employing the characterization~\Cref{thm:char-two-sid}, \Cref{,it:char-two-sid-iii}, and then the closure property~\Cref{prop:clos-sl-cocart-fib-pb}. In particular, letting either of the maps $k,m$ be an identity, we can conclude that $P^k$ is cocartesian, and $P_m$ is cartesian.
	
	That the square is a two-sided cartesian functor follows by separately projecting to the factors in the base, and then using~\cite[Proposition~5.3.9]{BW21} or its dual. Namely, \eg~since $P_m:A \to \UU$ is cocartesian, so is its pullback along $k:A' \to A$ which arises as
	\[\begin{tikzcd}
		{(k \times m)^*\widetilde{P} \simeq k^*\widetilde{P_m}} && {\widetilde{P_m}} \\
		{A'} && A
		\arrow[two heads, from=1-3, to=2-3]
		\arrow["k"', from=2-1, to=2-3]
		\arrow[two heads, from=1-1, to=2-1]
		\arrow[from=1-1, to=1-3]
		\arrow["\lrcorner"{anchor=center, pos=0.125}, draw=none, from=1-1, to=2-3]
	\end{tikzcd}\]
	and the pullback square is known to be a cocartesian functor.
\end{proof}

\begin{proposition}[Pullback stability of two-sided cartesian functors]\label{prop:2s-cart-fun-pb}
	In the following, let all types be Rezk. Consider two-sided cartesian families $P,Q:A \to B\to \UU$ with unstraightenings $E \fibarr \to A \times B$ of $P$, and $F \fibarr A \times B$ of $Q$, resp. Let $\kappa:P\to_{A \times B} Q$ be a two-sided cartesian functor. Given maps $k:A' \to A$, $m:B' \to B$, then the functor $\kappa': P' \to_{A' \times B'} Q'$ induced by pullback along $k \times m$ is two-sided cartesian as well:
	\[\begin{tikzcd}
		{E'} &&& E \\
		& {Q'} &&& Q \\
		{A' \times B'} &&& {A \times B}
		\arrow["\kappa'", dashed, from=1-1, to=2-2]
		\arrow[from=1-1, to=1-4]
		\arrow["\kappa", from=1-4, to=2-5]
		\arrow[two heads, from=1-4, to=3-4]
		\arrow[two heads, from=2-5, to=3-4]
		\arrow[two heads, from=1-1, to=3-1]
		\arrow[two heads, from=2-2, to=3-1]
		\arrow["{k \times m}"{description}, from=3-1, to=3-4]
		\arrow["\lrcorner"{anchor=center, pos=0.125}, draw=none, from=2-2, to=3-4]
		\arrow["\lrcorner"{anchor=center, pos=0.125}, shift right=4, draw=none, from=1-1, to=3-4]
		\arrow[from=2-2, to=2-5, crossing over]
	\end{tikzcd}\]
\end{proposition}

\begin{proof}
	Let $u':a_0' \to_{A'} a_1'$, $b':B'$, $e:P'(a_0',b') \simeq P(k\,a_0',m\,b')$. Straightforward calculation gives
	\begin{align*}
		& \kappa'_{u',b'}(P_!'(u',b',e)) = \kappa_{k\,u',m\,'}(P_!(k\,u',m\,b',e)) \\
		& = Q_!(k\,u',m\,b',\kappa_{k\,u',m\,b'}(e)) = Q_!'(u',b',\kappa'_{u',b'}(e)), 
	\end{align*}
	where the second identity is given by $\kappa$ being cocartesian. The dual case for cartesian lifts works similarly (\cf~also~\cite[Proposition~5.3.18]{BW21}).
\end{proof}

\subsection{Dependent and sliced product}

\begin{proposition}[Product stability of two-sided cartesian families]\label{prop:2s-cart-closed-pi}
	Let $A,B:I \to \UU$ be families of Rezk types for a small type $I$. Consider a two-sided family $P:\prod_{i:I} A_i \to B_i \to \UU$. Then the induced product family
	\[ \prod_{i:I} P_i : \prod_{i:I} A_i \to \prod_{i:I} B_i \to \UU\]
	is two-sided cartesian as well.
	
	Moreover, denoting the unstraightenings of the $P_i$ by $\varphi_i: E_i \fibarr A_i \times B_i$, the squares
	\[\begin{tikzcd}
		{\prod_{i:I} E_i} && {E_k} \\
		{\prod_{i:I} A_i \times \prod_{i:I} B_i} && {A_k \times B_k}
		\arrow[two heads, from=1-1, to=2-1]
		\arrow[from=2-1, to=2-3]
		\arrow[from=1-1, to=1-3]
		\arrow[two heads, from=1-3, to=2-3]
	\end{tikzcd}\]
	are two-sided cartesian functors. Furthermore, these product cones are terminal \wrt~two-sided cartesian functors.\footnote{Here, and in the following we will not formally spell out the universal properties, but they are analogous to the respective propositions in~\cite[Subsection~5.3.3]{BW21}. The addition/generalization is that the base types are binary products, and the fibrations and functors are \emph{two-sided} cartesian.}
\end{proposition}

Fibrationally, the proposition says that given a family of two-sided fibrations $\varphi_i: E_i \fibarr A_i \times B_i$ for $i:I$, the product fibration $\prod_{i:I} \varphi_i : \prod_{i:I} E_i \fibarr \prod_{i:I} A_i \times \prod_{i:I} B_i$ is also two-sided cartesian.

\begin{proof}
	This is a consequence of the characterization~\Cref{thm:char-two-sid}, \Cref{it:char-two-sid-iii}, in combination with the closure property~\Cref{prop:clos-sl-cocart-fib-prod}. Two-sided cartesianness of the projection squares follows upon postcomposition with the respective projection, and then employing either~\cite[Proposition~5.3.7]{BW21} or its dual. Similarly, one argues for the universal property for two-sided cartesian functors, using~\cite[Proposition~5.3.8]{BW21} or its dual, resp.
\end{proof}

\begin{corollary}[Sliced product stability of two-sided cartesian families]\label{prop:2s-cart-closed-pi-loc}
	Let $A,B$ be small Rezk types. Consider a two-sided family $P:\prod_{i:I} A \to B \to \UU$. Then the induced fiberweise product family
	\[ \times_{i:I}^{A \times B} P_i : A \to B \to \UU\]
	is two-sided cartesian as well.
	
	Moreover, for every $k:I$ there is an induced canonical commutative triangle
	\[\begin{tikzcd}
		{\times_{i:I}^{A \times B} E_i} && {E_k} \\
		& {A \times B}
		\arrow[from=1-1, to=1-3]
		\arrow["{\times_{i:I}^{A \times B} \varphi_i}"', two heads, from=1-1, to=2-2]
		\arrow["{\varphi_k}", two heads, from=1-3, to=2-2]
	\end{tikzcd}\]
	which is a two-sided functor. The two-sided fibration $\prod_{i:I} E_i \fibarr A \times B$ is the terminal cone over the $\varphi_k: E_k \fibarr A \times B$~\wrt~(triangle-shaped) cones into the $\varphi_k$ whose horizontal map is two-sided cartesian.
\end{corollary}

\begin{proof}
	Recall that we have equivalences
	\begin{align*}
		\prod_i E_i & \simeq \sum_{\substack{\alpha:I \to A \\ \beta:I \to B}} \prod_{i:I} P_i(\alpha_i,\beta_i) \fibarr A^I \times B^I, \\
		\times_{i:I}^{A \times B} E_i & \simeq \sum_{\substack{a:A \\ b:B}} \prod_{i:I} P_i(a,b) \fibarr A \times B.
	\end{align*}
	
	Denote by $E_i \fibarr A \times B$ the unstraightening of the family $P_i$. By~\Cref{prop:2s-cart-closed-pb}, the induced map $\prod_{i:I} E_i \fibarr (A \times B)^I$ is two-sided cartesian:
	\[\begin{tikzcd}
		{\times_{i:I}^{A \times B} E_i} && {\prod_{i:I} E_i} \\
		{A \times B} && {(A \times B)^I}
		\arrow[two heads, from=1-1, to=2-1]
		\arrow["{\mathrm{cst}}"', from=2-1, to=2-3]
		\arrow[from=1-1, to=1-3]
		\arrow[two heads, from=1-3, to=2-3]
		\arrow["\lrcorner"{anchor=center, pos=0.125}, draw=none, from=1-1, to=2-3]
	\end{tikzcd}\]
	Invoking pullback-stability, and then considering the straightening of this map to recover a type family establishes the claim.
	
	Now, by the above description via projection equivalence, we have evaluation maps yielding the desired cones $\ev_k: {\times_{i:I}^{A \times B} E_i} \to_{A \times B} E_k$ for $k:I$. But by the universal property of the standard dependent product~\Cref{prop:2s-cart-closed-pi} (\cf~\cite[Proposition~5.3.8]{BW21}), these factor as follows
	\[\begin{tikzcd}
		{\times_{i:I}^{A \times B} E_i} && {\prod_{i:I} E_i} && {E_k} \\
		{A \times B} && {A^I \times B^I} && {A \times B}
		\arrow["{\mathrm{ev}_k}"{description}, from=1-3, to=1-5]
		\arrow[from=1-3, to=2-3]
		\arrow[from=2-3, to=2-5]
		\arrow[from=1-5, to=2-5]
		\arrow[from=1-1, to=2-1]
		\arrow["{\mathrm{cst}}"', dashed, from=2-1, to=2-3]
		\arrow[dashed, from=1-1, to=1-3]
		\arrow["{\mathrm{ev}_k}"{description}, curve={height=-18pt}, from=1-1, to=1-5]
		\arrow["\lrcorner"{anchor=center, pos=0.125}, draw=none, from=1-1, to=2-3]
	\end{tikzcd}\]
	where the upper horizontal induced functor is two-sided cartesian, as are the evaluations from the standard dependent product. Hence, so is their composite, as desired, by~\Cref{prop:2s-fun-comp}.
\end{proof}

\begin{proposition}[Pullback cones are two-sided cartesian functors]\label{prop:2s-cart-fun-pb-cones}
	Consider two-sided families over Rezk types
	\[ P:A \to B \to \UU, \quad  P':A' \to B' \to \UU, \quad  P:A'' \to B'' \to \UU. \]
	Furthermore, assume there are maps
	\[ \alpha:A' \to A, \quad \alpha'':A'' \to A, \quad \beta':B' \to B, \quad \beta'':B'' \to B \]
	and two-sided cartesian functors
	\[ \kappa':P' \to_{A'\times B', A \times B} P, \quad  \kappa'':P'' \to_{A''\times B'', A \times B} P.\]
	Denote by
	\[ \varphi:E \fibarr A \times B, \quad \varphi': E' \fibarr A' \times B', \quad \varphi'': E'' \fibarr A'' \times B'' \]
	the unstraightenings of $P$, $P'$, and $P''$, resp. Consider the induced pullback:
	\[{\small \begin{tikzcd}
		{E' \times_E E''} &&& {E''} \\
		& {E'} &&& E \\
		{(A' \times B') \times_{A \times B} (A'' \times B'')} &&& {A'' \times B''} \\
		& {A' \times B'} &&& {A \times B}
		\arrow["{\varphi'\times_\varphi \varphi''}"{description}, dashed, two heads, from=1-1, to=3-1]
		\arrow[from=3-1, to=3-4]
		\arrow[from=1-1, to=1-4]
		\arrow["{\varphi''}"{description, pos=0.2}, two heads, from=1-4, to=3-4]
		\arrow[from=1-1, to=2-2]
		\arrow["\varphi"{description, pos=0.3}, two heads, from=2-5, to=4-5]
		\arrow[from=3-1, to=4-2]
		\arrow["{\alpha '' \times \beta''}"{description}, from=3-4, to=4-5]
		\arrow["{\kappa''}"{description}, from=1-4, to=2-5]
		\arrow["{\alpha' \times \beta'}"{description}, from=4-2, to=4-5]
		\arrow["\lrcorner"{anchor=center, pos=0.125, rotate=45}, draw=none, from=3-1, to=4-5]
		\arrow["\lrcorner"{anchor=center, pos=0.125, rotate=45}, draw=none, from=1-1, to=2-5]
		\arrow["{\varphi'}"{description, pos=0.3}, two heads, from=2-2, to=4-2, crossing over]
		\arrow["{\kappa'}"{description}, from=2-2, to=2-5, crossing over]
	\end{tikzcd}}\]
	Then the mediating map
	\[ \varphi''' \defeq \varphi'\times_\varphi \varphi'': E'''\defeq {E' \times_E E''} \fibarr A''' \times B'''\]
	where
	\[ A''' \defeq A' \times_A A'', \quad B''' \defeq B' \times_B B'' \]
	is a two-sided cartesian fibration.\footnote{Note in particular that we have an equivalence $A''' \times B''' \equiv (A' \times B') \times_{A \times B} (A'' \times B'')$}
	
	Moreover, each of the projection squares from $\varphi'''$ is a two-sided cartesian functor, and $\varphi''':E''' \to A''' \times B'''$ satisfies the expected terminal universal property for cones which are two-sided cartesian functors (analogous to~\cite[Propositions~5.3.10,11]{BW21}).
\end{proposition}

\begin{proof}
	We use projection equivalence so that we can take the fibers of $\varphi'''$ to be\footnote{Where $a:A$, $a':A'$ lies strictly over $a$ via $k'$~\etc.}
	\[ P'''(a,a',a'',b,b',b'') \jdeq P'(a',b') \times_{P(a,b)} P''(a'',b'').\]
	We claim that the cocartesian lifts in $P'''$ are then given by
	\[ P_!'''(u,u',u'',b,b',b'',\angled{e,e',e''}) \jdeq \angled{P_!(u,b,e), P_!'(u',b',e'), P_!''(u'',b'',e'')},\]
	which can be checked to be cocartesian since the conditions are validated fiberwise. In particular, the cocartesian lifts in $P'$ and $P''$ indeed lie over the ones in $P$ by two-sided cartesian-ness of $\kappa'$ and $\kappa''$. So far, this is analogous to~\cite[Proposition~5.3.10]{BW21}, but we have the additional triple of points $\angled{b,b',b''}$ as data.
	
	The argument for the cartesian lifts works dually.
	Now, the compatibility condition from~\Cref{prop:comm-lifts} have to be checked. But by the projection equivalence above, the ensuing proposition is just witnessing that the condition is satisfied component-wise for triples $\angled{\sigma,\sigma',\sigma''}$ where $\sigma$ is a square of the form
	
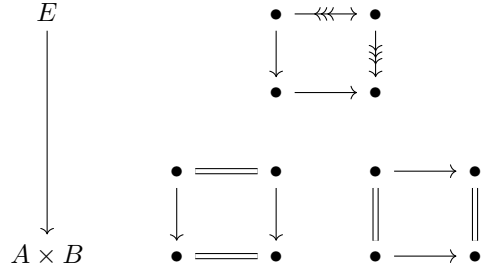
\[\begin{tikzcd}
		E && \bullet & \bullet \\
		&& \bullet & \bullet \\
		& \bullet & \bullet & \bullet & \bullet \\
		{A \times B} & \bullet & \bullet & \bullet & \bullet
		\arrow[cocart, from=1-4, to=2-4]
		\arrow[from=2-3, to=2-4]
		\arrow[from=1-3, to=2-3]
		\arrow[cart, from=1-3, to=1-4]
		\arrow[Rightarrow, no head, from=3-2, to=3-3]
		\arrow[from=3-3, to=4-3]
		\arrow[from=3-2, to=4-2]
		\arrow[Rightarrow, no head, from=4-2, to=4-3]
		\arrow[Rightarrow, no head, from=3-4, to=4-4]
		\arrow[from=4-4, to=4-5]
		\arrow[from=3-4, to=3-5]
		\arrow[Rightarrow, no head, from=3-5, to=4-5]
		\arrow[from=1-1, to=4-1]
	\end{tikzcd}\]
	and $\sigma',\sigma''$ are of the same shape, lying above. Since the compatbility condition is satisfied for each of those, we are done. This shows that $\varphi'''$ is a two-sided cartesian fibration, as desired.
	
	From the discussion of the lifts, it is also clear that both the projection squares are two-sided cartesian functors, since we just project to the respective coordinates. Furthermore, the universal property is established, again, by postcomposing separately with the projections to either $A'''$ or $B'''$, then applying either~\cite[Proposition~5.3.11]{BW21} for the one-sided cocartesian case, or its dual for the cartesian case.
\end{proof}

\begin{corollary}[Pullback cones in a slice are two-sided cartesian functors]\label{prop:2s-cart-fun-pb-cones-sliced}
	Consider two-sided families over Rezk types $P,P',P'':A \to B \to \UU$ with unstraightenings $\varphi:E \fibarr A \times B$, $\varphi':E' \fibarr A' \times B'$, and $\varphi'':E'' \fibarr A'' \times B''$. Given two-sided cartesian functors $\kappa':P' \to_{A \times B} P$ and $\kappa'':P'' \to P$,
	consider the induced pullback over $A \times B$:
	\[\begin{tikzcd}
		{E' \times_E E''} &&& {E''} \\
		& {E'} &&& E \\
		&& {A \times B}
		\arrow[from=1-1, to=1-4]
		\arrow[from=1-1, to=2-2]
		\arrow["{\kappa''}"{description}, from=1-4, to=2-5]
		\arrow["\lrcorner"{anchor=center, pos=0.125, rotate=45}, draw=none, from=1-1, to=2-5]
		\arrow[two heads, from=2-5, to=3-3]
		\arrow[two heads, from=1-4, to=3-3]
		\arrow[two heads, from=2-2, to=3-3]
		\arrow[curve={height=30pt}, from=1-1, to=3-3]
		\arrow["{\kappa'}"{description, pos=0.3}, from=2-2, to=2-5, crossing over]
	\end{tikzcd}\]
	Then the mediating map
	\[ \varphi''' \defeq \varphi'\times_\varphi \varphi'': E'''\defeq {E' \times_E E''} \fibarr A \times B\]
	is a two-sided cartesian fibration.
	
	Moreover, each of the projection squares from $\varphi'''$ is a two-sided cartesian functor, and $\varphi''':E''' \to A \times B$ satisfies the expected terminal universal property for cones which are two-sided cartesian functors over $A \times B$.
\end{corollary}

\subsection{Sequential limit}

\begin{proposition}[Sequential limit cones are cocartesian functors]\label{prop:cocart-fun-seqlim}
	Consider an inverse diagram of two-sided cartesian fibrations as below where all of the connecting squares are two-sided cartesian functors:
	\[{\small \begin{tikzcd}
		& \ldots &&&& {E_\infty} \\
		\cdots && {E_2} && {E_1} && {E_0} \\
		& \ldots &&&& {A_\infty \times B_\infty} \\
		\cdots && {A_2 \times B_2} && {A_1 \times B_1} && {A_0 \times B_0}
		\arrow["{\pi_\infty}"{description, pos=0.3}, dashed, from=1-6, to=3-6]
		\arrow["{f_0}"{description, pos=0.7}, from=4-5, to=4-7]
		\arrow["{\kappa_0}"{description}, dashed, from=1-6, to=2-7]
		\arrow["{\langle \alpha_0,\beta_0 \rangle}"{description}, dashed, from=3-6, to=4-7]
		\arrow["{f_1}"{description, pos=0.7}, from=4-3, to=4-5]
		\arrow["{g_1}"{description, pos=0.7}, from=2-3, to=2-5]
		\arrow["{g_2}"{description, pos=0.7}, from=2-1, to=2-3]
		\arrow["{f_2}"{description, pos=0.7}, from=4-1, to=4-3]
		\arrow["{\kappa_1}"{description}, dashed, from=1-6, to=2-5]
		\arrow["{\kappa_2}"{description}, dashed, from=1-6, to=2-3]
		\arrow["{\langle \alpha_2,\beta_2 \rangle}"{description, pos=0.7}, curve={height=6pt}, dashed, from=3-6, to=4-3]
		\arrow["{\langle \alpha_1,\beta_1 \rangle}"{description}, dashed, from=3-6, to=4-5]
		\arrow["{\kappa_3}"{description}, curve={height=12pt}, dashed, from=1-6, to=2-1]
		\arrow["{\langle \alpha_3,\beta_3 \rangle}"{description}, curve={height=12pt}, dashed, from=3-6, to=4-1]
		\arrow["{g_0}"{description, pos=0.7}, from=2-5, to=2-7, crossing over]
		\arrow["{\varphi_2}"{description, pos=0.3}, two heads, from=2-3, to=4-3, crossing over]
		\arrow["{\varphi_0}"{description, pos=0.3}, two heads, from=2-7, to=4-7, crossing over]
		\arrow["{\varphi_1}"{description, pos=0.3}, two heads, from=2-5, to=4-5, crossing over]
	\end{tikzcd}}\]
	Then the induced map $\pi_\infty: E_\infty \to B_\infty$ between the limit types is a two-sided cartesian fibration, and the projection squares constitute two-sided cartesian functors.
	
	Furthermore, $\pi_\infty: E_\infty \to B_\infty$ together with the projection squares satisfies the universal property of a sequential limit~\wrt~to cones of two-sided cartesian functors.
\end{proposition}

\begin{proof}
	Using the closure properties from this chapter, we can argue along the lines of~\cite[Proposition~5.3.12]{BW21}.\footnote{We thank Ulrik Buchholtz for initially suggesting this proof in~\cite{BW21} because it circumvents dealing with unwieldy coherence data that would occur in different presentations of the sequential limit.} This means, the limit fibration, again is constructed via the pullback
	\[{\small \begin{tikzcd}
		{E_\infty} & {} & {\prod_{n:\mathbb N} E_{2n}} & {} \\
		& {\prod_{n:\mathbb N} E_{2n+1}} & {} & {\prod_{n:\mathbb N} E_{2n}} & {} \\
		{A_\infty \times B_\infty} && {\mathllap{\prod_{n:\mathbb N} A_{2n}} \times \prod_{n:\mathbb N} B_{2n}} \\
		& {\prod_{n:\mathbb N} A_{2n+1} \times \prod_{n:\mathbb N} B_{2n+1}} && {\mathllap{\prod_{n:\mathbb N} A_n} \times \prod_{n:\mathbb N} B_n} & {}
		\arrow[dashed, two heads, from=1-1, to=3-1]
		\arrow[two heads, from=1-3, to=3-3]
		\arrow[from=3-1, to=4-2]
		\arrow[shift right=2, dashed, from=1-1, to=2-2]
		\arrow[shorten >=41pt, from=4-2, to=4-4]
		\arrow[shorten >=45pt, from=3-1, to=3-3]
		\arrow[from=3-3, to=4-4]
		\arrow[two heads, from=2-4, to=4-4]
		\arrow[from=1-1, to=1-3]
		\arrow[from=1-3, to=2-4]
		\arrow["\lrcorner"{anchor=center, pos=0.125, rotate=45},draw=none, from=1-1, to=4-4, shift left=3]
		\arrow["\lrcorner"{anchor=center, pos=0.125, rotate=45}, draw=none, from=3-1, to=4-4]
		\arrow[from=2-2, to=2-4, crossing over]
		\arrow[two heads, from=2-2, to=4-2, crossing over]
	\end{tikzcd}}\]
	and is two-sided cartesian due to~\Cref{prop:2s-cart-fun-pb-cones}. From~\cite[Proposition~5.3.12]{BW21} and its dual we get that the projection squares are two-sided cartesian functors. The universal property is established using~\cite[Proposition~5.3.13]{BW21} and its dual.
\end{proof}

\begin{corollary}[Sequential limit cones in a slice are cocartesian functors]\label{prop:cocart-fun-seqlim-sl}
	Consider an inverse diagram of two-sided cartesian fibrations as below where all of the connecting squares are two-sided cartesian functors:
	\[\begin{tikzcd}
		&& \cdots && {E_\infty} \\
		\cdots & {E_2} && {E_1} && {E_0} \\
		& \cdots \\
		& {} && {A \times B}
		\arrow["{\kappa_0}"{description}, dashed, from=1-5, to=2-6]
		\arrow["{g_1}"{description, pos=0.7}, from=2-2, to=2-4]
		\arrow["{\kappa_1}"{description}, dashed, from=1-5, to=2-4]
		\arrow["{\kappa_2}"{description}, dashed, from=1-5, to=2-2]
		\arrow["{\varphi_1}"{description}, from=2-4, to=4-4]
		\arrow["{\varphi_2}"{description}, from=2-2, to=4-4]
		\arrow["{\varphi_0}"{description}, from=2-6, to=4-4]
		\arrow[dashed, from=1-5, to=4-4, "\varphi_\infty"]
		\arrow[from=2-1, to=2-2]
		\arrow["{g_0}"{description, pos=0.7}, from=2-4, to=2-6, crossing over]
	\end{tikzcd}\]
	Then the induced map $\pi_\infty: E_\infty \to A \times B$ between the limit types is a two-sided cartesian fibration, and the projection squares constitute two-sided cartesian functors.
	
	Furthermore, $\pi_\infty: E_\infty \fibarr A \times B$ together with the projection squares satisfies the universal property of a sequential limit, relativized to the basis $A \times B$.
\end{corollary}

\begin{proof}
	The sequential limit of a diagram of identity maps is the object itself, \eg~
	\[ \seqlim_n \pair{B}{\id_B} \simeq \sum_{\sigma:\N \to B} \prod_{n,k:\N} \sigma(n) = \sigma(n+k) \simeq B.\]
	Thus, the claim follows from~\Cref{prop:cocart-fun-seqlim}.
\end{proof}

\subsection{Cotensors}

\begin{proposition}[Cocart.~fibrations are cotensored over maps/shape inclusions]\label{prop:2scart-fun-cotensor-maps}
	Let $P: A \to B \to \UU$ be a two-sided cartesian family with associated projection $\varphi \defeq \pair{\xi}{\pi}:E \fibarr A \times B$. For any type map or shape inclusion $j:Y \to X$, the maps $\varphi^X$ and $\varphi^Y$ are two-sided cartesian fibrations, and moreover the square
	\[\begin{tikzcd}
		{E^X} && {E^Y} \\
		{A^X \times B^X} && {A^Y \times B^Y}
		\arrow["{\varphi^X}"', two heads, from=1-1, to=2-1]
		\arrow[from=2-1, to=2-3]
		\arrow[from=1-1, to=1-3]
		\arrow["{\varphi^Y}", two heads, from=1-3, to=2-3]
	\end{tikzcd}\]
	is a two-sided cartesian functor.
\end{proposition}

\begin{proof}
	By closedness under products, the maps $\varphi^X$, $\varphi^Y$ are two-sided cartesian fibrations. From~\cite[Proposition~5.3.15]{BW21}, we know that the square formed by the composites
	\[\begin{tikzcd}
		{E^X} && {E^Y} \\
		{A^X \times B^X} && {A^Y \times B^Y} \\
		{A^X} && {A^Y}
		\arrow["{\varphi^X}"', two heads, from=1-1, to=2-1]
		\arrow[from=2-1, to=2-3]
		\arrow[from=1-1, to=1-3]
		\arrow["{\varphi^Y}", two heads, from=1-3, to=2-3]
		\arrow[two heads, from=2-1, to=3-1]
		\arrow[two heads, from=2-3, to=3-3]
		\arrow["{\xi^X}"{description}, curve={height=40pt}, two heads, from=1-1, to=3-1]
		\arrow["{\xi^Y}"{description}, curve={height=-40pt}, two heads, from=1-3, to=3-3]
		\arrow[from=3-1, to=3-3]
	\end{tikzcd}\]
	is a cocartesian functor (since by precondition $\xi:E \fibarr A$ is a cocartesian fibration). The cartesian case over $B$ works the same.
\end{proof}

\begin{corollary}[Cocartesian fibrations in a slice are cotensored over maps/shape inclusions]\label{prop:2scart-fun-cotensor-maps-sl}
	Let $P: A \to B \to \UU$ be a two-sided cartesian family with associated projection $\varphi \defeq \pair{\xi}{\pi}:E \fibarr A \times B$. For any type map or shape inclusion $j:Y \to X$, the maps $X \iexp \varphi$ and $Y \iexp \varphi$ are two-sided cartesian fibrations, and moreover the triangle
	\[\begin{tikzcd}
		{X \iexp E} && {Y \iexp E} \\
		& {A \times B}
		\arrow["j \iexp \varphi", from=1-1, to=1-3]
		\arrow["{X \iexp \varphi}"', two heads, from=1-1, to=2-2]
		\arrow["{Y \iexp \varphi}", two heads, from=1-3, to=2-2]
	\end{tikzcd}\]
	is a two-sided cartesian functor.
\end{corollary}

\begin{proposition}[Cocart.~functors are closed under Leibniz cotensors]\label{prop:2scart-fun-leibniz}
	Let $j: Y \to X$ be a type map or shape inclusion. Then, given two-sided cartesian fibrations $\psi: F \fibarr A \times B$, $\varphi: E \fibarr C \times D$, and a cocartesian functor
	\[\begin{tikzcd}
		F && E \\
		{A \times B} && {C \times D}
		\arrow["\psi"', two heads, from=1-1, to=2-1]
		\arrow["{\pair{k}{m}}"', from=2-1, to=2-3]
		\arrow["\mu", from=1-1, to=1-3]
		\arrow["\varphi", two heads, from=1-3, to=2-3]
	\end{tikzcd}\]
	the square induced between the Leibniz cotensors
	\[\begin{tikzcd}
		{F^X} && {F^Y\times_{E^Y} E^X} \\
		{(A \times B)^X} && {(A \times B)^X \times_{(A \times B)^Y} (C \times D)^X}
		\arrow[two heads, from=1-3, to=2-3]
		\arrow["{j \cotens \mu}", from=1-1, to=1-3]
		\arrow["{j \cotens \langle k,m \rangle }"', from=2-1, to=2-3]
		\arrow["{\psi^X}"', two heads, from=1-1, to=2-1]
	\end{tikzcd}\]
	is a cocartesian functor.
\end{proposition}

\begin{proof}
	This works, again, analogously to~\cite[Proposition~5.3.16]{BW21}, using~\Cref{prop:2scart-fun-cotensor-maps}, and then~\Cref{prop:2s-cart-fun-pb-cones}.
\end{proof}

\begin{corollary}[Cocartesian functors in a slice are closed under Leibniz cotensors]\label{prop:2scart-fun-leibniz-sl}
	Let $j: Y \to X$ be a type map or shape inclusion. Then, given two-sided cartesian fibrations $\psi: F \fibarr A \times B$, $\varphi: E \fibarr A \times B$, and a two-sided cartesian functor
	\[\begin{tikzcd}
		F && E \\
		& {A \times B}
		\arrow["\kappa", from=1-1, to=1-3]
		\arrow["\psi"', two heads, from=1-1, to=2-2]
		\arrow["\varphi", two heads, from=1-3, to=2-2]
	\end{tikzcd}\]
	the square induced between the Leibniz cotensors
	\[\begin{tikzcd}
		{X \iexp F} &&&& {Y \iexp F \times_{Y \iexp E} X \iexp F} \\
		&& {A \times B}
		\arrow["{j \cotens_{A \times B} \kappa}", from=1-1, to=1-5]
		\arrow[two heads, from=1-1, to=2-3]
		\arrow[two heads, from=1-5, to=2-3]
	\end{tikzcd}\]
	is a two-sided functor.
\end{corollary}

In sum, we obtain a synthetic analogue of the cosmological closure properties of two-sided cartesian fibrations, \wrt~varying as well as a fixed base (\cf~\cite[Theorem~7.1.4, Proposition~7.1.7, and Definition~7.2.1]{RV21}):
\begin{theorem}[(Sliced) cosmological closure properties of two-sided cartesian families]\label{thm:2scart-cosm-closure}
	Over Rezk bases, it holds that:
	
	Two-sided cartesian families are closed under composition, dependent products, pullback along arbitrary maps, and cotensoring with maps/shape inclusions. Families corresponding to equivalences or terminal projections are always cocartesian.
	
	Between two-sided cartesian families over Rezk bases, it holds that:
	Two-sided cartesian functors are closed under (both horizontal and vertical) composition, dependent products, pullback, sequential limits,\footnote{all three objectwise limit notions satisfying the expected universal properties \wrt~to cocartesian functors} and Leibniz cotensors.
	
	Fibered equivalences and fibered functors into the identity of $\unit$ are always cocartesian.
	
	Furthermore, all of this is analogously true~\wrt~two-sided cartesian families over the same base and applying sliced versions of the constructions,
\end{theorem}

\section{Two-sided Yoneda Lemma}\label{sec:2s-yon}

In this section we establish a two-sided Yoneda Lemma in fibrational form. It constitutes a type-theoretic version of~\cite[Theorem~7.3.2]{RV21}, after the previous versions for the discrete case~\cite[Theorems~9.1 and Theorem~9.5]{RS17} and the one-sided cocartesian case~\cite[Theorems~7.2.2 and~7.2.3]{BW21}. As explained in the aforementioned sources, in the type-theoretic context, it can be understood as hom-type induction principle for two-sided cartesian type families, analogous to path induction for identity types, \cf~\eg~\cite[Section~1.12.1]{hottbook}, \cite[Section~5.1]{RijIntro}, ~\cite[September~15:~Identity Types]{RieHoTT}.

\subsection{Two-sided cartesian sections}

\begin{definition}[Two-sided cartesian sections]
	Let $P:A \times B \to \UU$ be a two-sided family with associated cocartesian fibration $\xi: E \fibarr A$ and cartesian fibration $\pi: E \fibarr B$, resp.
	
	A section $\sigma: \prod_{\substack{a:A \\ b:B}} P(a,b)$ is \emph{two-sided cartesian} if it maps pairs $\pair{u}{\id_b}$ to $\xi$-cocartesian sections and $\pair{\id_a}{v}$ to $\pi$-cartesian sections, \ie:~for all arrows $u:\Delta^1 \to A$, $v:\Delta^1 \to B$ and elements $a:A$, $b:B$ the dependent arrow $\sigma(u,\id_b) : \prod_{t:\Delta^1} P(u(t),b)$ is $\xi$-cocartesian while $\sigma(\id_a,v): \prod_{t:\Delta^1} P(a,v(t))$ is $\pi$-cartesian.
	
	Note that this yields a proposition, and the (sub-)type of such sections is denoted by
	\[ \prod_{\substack{a:A\\b:B}}^\tscart P(a,b) \cofibarr  \prod_{\substack{a:A\\b:B}} P(a,b) . \]
\end{definition}

Of central importance will be the following map. We fix a two-sided family $P:A \times B \to \UU$, and assume $a:A$ to be initial and $b:B$ to be terminal. We then define
\[ \yon: P(a,b) \to \prod_{A \times B} P, \quad \yon \defeq \lambda d,x,y. (\emptyset_x)_!((!_y)^* d). \]
Note that by two-sidedness of $P$ we have a path
\[ \yon(d)(x,y) = (\emptyset_x)_!((!_y)^* b) = (!_y)^*((\emptyset_x)_! d). \]

In the following, we first show that $\yon$ is, in fact, valued in two-sided cartesian sections. We then conclude that it is a quasi-inverse of the evaluation map, constituting a quasi-equivalence:
\[\begin{tikzcd}
	{\prod_{A \times B}^\tscart P} && {P(a,b)}
	\arrow[""{name=0, anchor=center, inner sep=0}, "{\ev_{\langle a,b \rangle}}"', curve={height=12pt}, from=1-1, to=1-3]
	\arrow[""{name=1, anchor=center, inner sep=0}, "\yon"', curve={height=12pt}, from=1-3, to=1-1]
	\arrow["\simeq"{description}, Rightarrow, draw=none, from=1, to=0]
\end{tikzcd}\]
Finally, the Yoneda Lemmas will follow as instances from this.

\begin{proposition}
	Let $P:A \times B \to \UU$ be a two-sided family. Assume $a:A$ is initial and $b:B$ is terminal. Then for all $d:P(a,b)$, the section $\yon\,d:\prod_{A \times B} P$ is two-sided cartesian. 
\end{proposition}

\begin{proof}
	This is an extension of \cite[Proposition~7.1.3]{BW21} to the two-sided case. We write $E \defeq \totalty{P}$ and fix an element $d:P(a,b)$. Let $\xi:E \fibarr A$ denote the associated cocartesian fibration. Again, we will only establish one of the two dual parts of the statement, namely that $\yon d(u,\id_y)$ is $\xi$-cocartesian for any $u:\Delta^1 \to A$ and $y:B$
	
	From this, we define the map that yields the cartesian transport along the terminal maps in $B$, \ie,
	\[ \tau \jdeq \lambda x,y.(!_y)^*d : A \times B \to E. \]
	Next, consider the family of cocartesian lifts over the initial maps in $A$, starting at the points given by $\tau$. This is realized by the $2$-cell $\tau:\hom_{A \times B \to E}(\tau\,d, \yon \,d)$ defined by
	$\tau(x,y) \defeq \xi_!(\emptyset_x,(!_y)^*\,d): (!_y)^*d \cocartarr_{\pair{\emptyset_x}{y}} (\emptyset_x)_!(!_y)^*d$.
	The action of the $2$-cell $\chi$ on a pair $\pair{u}{\id_y}$ for $u:\Delta^1 \to A$ and $y:B$ is given by the following dependent square:
	\[\begin{tikzcd}
		E & {\tau d(x,y)} && {\tau d(x',y)} \\
		& {\yon d(x,y)} && {\yon d(x',y)} \\
		{A \times B} & a && a \\
		& x && {x'} \\
		& y && y \\
		& y && y
		\arrow["{\id_{(!_y)^*d}}", Rightarrow, no head, from=1-2, to=1-4]
		\arrow[from=1-2, to=2-2, cocart]
		\arrow["{\yon d(u,\id_y)}"', from=2-2, to=2-4]
		\arrow[from=1-4, to=2-4, cocart]
		\arrow["{\id_a}", Rightarrow, no head, from=3-2, to=3-4]
		\arrow["u"', from=4-2, to=4-4]
		\arrow[dashed, from=3-4, to=4-4, "\emptyset_x"]
		\arrow[dashed, from=3-2, to=4-2, "\emptyset_{x'}", swap]
		\arrow["{\id_y}"', Rightarrow, no head, from=5-2, to=6-2]
		\arrow["{\id_y}", Rightarrow, no head, from=5-2, to=5-4]
		\arrow["{\id_y}"', Rightarrow, no head, from=6-2, to=6-4]
		\arrow[two heads, from=1-1, to=3-1]
		\arrow["{\id_y}", Rightarrow, no head, from=5-4, to=6-4]
	\end{tikzcd}\]
	By right cancelation of cocartesian arrows, $\yon d(u,\id_y)$ is cocartesian, too.
\end{proof}

We need one more lemma before we are ready to prove the main theorem of this subsection, which in turn will specialize to the desired versions of the Yoneda Lemma. The lemma gives canonical identities in the presence of inital and terminal elements, resp., in the base types.

\begin{lemma}[Coherence of terminal transport with two-sided cartesian sections]\label{lem:coh-ttransp-2scart}
	Let $A$ and $B$ types with an initial element $a:A$ and terminal element $b:B$. Furthermore, consider a two-sided family $P:A \times B \to \UU$ with associated cocartesian fibration $\xi: E \fibarr A$ and cartesian fibration $\pi: E \fibarr B$, resp. Given a section $\sigma:\prod_{\substack{a:A\\b:B}}^\tscart P$, for any $x:A$, $y:B$  there are identifications
	\[ (!_y)^*\sigma(a,b) = \sigma(a,y) \quad \text{and} \quad (\emptyset_x)_!\sigma(a,b) = \sigma(x,b). \]
\end{lemma}

\begin{proof}
	We only treat the first named case since the second named one is completely dual.
	
	Consider on the one hand the cartesian lift of the terminal map $!_y: y \to b$ \wrt~$\sigma(a,b)$ (over the identity $\id_a: a = a$), \ie, the dependent arrow $f \defeq (!_y)^*\sigma(a,b) \cartarr \sigma(a,b)$. On the other hand, consider the action of $\sigma$ on the pair $\pair{\id_a}{!_y}$, namely $h\defeq \sigma(\id_a,!_y) : \sigma(a,y) \cartarr \sigma(a,b)$, which is a cartesian arrow since the section $\sigma$ is two-sided cartesian. Then the mediating induced arrow $g \defeq \tyfill_h(f): \sigma(a,y) \cartarr (!_y)^*\sigma(a,b)$ is cartesian by left cancelation. But since it also $\pi$-vertical, lying over $\id_y$, it is an isomorphism, \cf~\Cref{fig:coh-2ssec}.
\end{proof}
\begin{figure}
	\[\begin{tikzcd}
		& {\sigma(a,y)} \\
		E & {(!_y)^*\sigma(a,b)} && {\sigma(a,b)} \\
		& a && a \\
		{A \times B} & y && b
		\arrow["{!_y}", from=4-2, to=4-4]
		\arrow[Rightarrow, no head, from=3-2, to=3-4]
		\arrow["f"', cart, from=2-2, to=2-4]
		\arrow["g"', dashed, "\simeq", from=1-2, to=2-2]
		\arrow["h", cart, from=1-2, to=2-4]
		\arrow[two heads, from=2-1, to=4-1]
	\end{tikzcd}\]
	\caption{Coherence of terminal transport with two-sided cartesian sections}\label{fig:coh-2ssec}
\end{figure}

In analogy with \cite[Proposition~7.1.4]{BW21}, \cite[Theorem~5.7.18]{RV21}, \cite[Theorem~9.7]{RS17}, the map $\yon: P(a,b) \to \prod_{A \times B}^{\tscart} P$ mediates an equivalence between these two types:
\begin{proposition}\label{prop:fiber-init-term-2s}
	Let $A,B$ be Rezk types with an initial element $a:A$ and a terminal element $b:B$. For a two-sided family $P:A \times B \to \UU$, evaluation at $\pair{a}{b}$ given by is an equivalence
	 $\ev_{\pair{a}{b}}: \Big( \prod_{A \times B}^\tscart P\Big) \stackrel{\simeq}{\to} P(a,b)$.
\end{proposition}

\begin{figure}
	\[\begin{tikzcd}
		E && {\sigma(a,y)} && {\sigma(x,y)} \\
		&& {(!_y)^*\sigma(a,b)} && {\yon d(\sigma(a,b))(x,y)} \\
		{A \times B} && a && x \\
		&& a && x \\
		&& y && y \\
		&& y && y
		\arrow["{\sigma(\emptyset_x, \id_y)}", from=1-3, to=1-5, cocart]
		\arrow["q"', Rightarrow, no head, from=1-3, to=2-3]
		\arrow["{\xi_!(\emptyset_x,(!_y)^* \sigma(a,b))}"', from=2-3, to=2-5, cocart]
		\arrow["{\emptyset_x}", from=3-3, to=3-5]
		\arrow["{\id_a}"', Rightarrow, no head, from=3-3, to=4-3]
		\arrow["{\id_x}", Rightarrow, no head, from=3-5, to=4-5]
		\arrow["{\id_y}"', Rightarrow, no head, from=5-3, to=6-3]
		\arrow["{\id_y}", Rightarrow, no head, from=5-3, to=5-5]
		\arrow["{\id_y}"', Rightarrow, no head, from=6-3, to=6-5]
		\arrow["{\id_y}", Rightarrow, no head, from=5-5, to=6-5]
		\arrow[two heads, from=1-1, to=3-1]
		\arrow["{\emptyset_x}"', from=4-3, to=4-5]
		\arrow[dashed, "g", from=1-5, to=2-5]
	\end{tikzcd}\]
	\caption{Two-sided cartesian sections}
	\label{fig:2s-ev-dep-sq}
\end{figure}

\begin{proof}
	We show that for the map $\yon$ as defined above we have identifications $\ev_{\pair{a}{b}} \circ \yon = \id_{P(a,b)}$ and $\yon \circ \ev_{\pair{a}{b}} = \id_{\prod_{A \times B}P}$. The first case is easy: the initial map into the initial element $a$ is just the identity, whose cocartesian lift is an identity as well, and the same holds analogously for the terminal element $b$, \ie, $ \yon(d)(\sigma)(a,b) = (\emptyset_a)_!(!_b)^*(d) = (\id_a)_!(\id_b)^*(d) = d$.
	For the other round-trip, we have to give an identification
	$(\yon \circ \ev_{\pair{a}{b}})(\sigma)(x,y) = \sigma(x,y)$.
	Fix elements $x:A$ and $y:B$.
	Note that by \Cref{lem:coh-ttransp-2scart} there is a path
	$q:\sigma(a,y) =_{P(a,y)} (!_y)^*(\sigma(a,b))$.
	Since $\sigma$ is two-sided cartesian, we obtain the dependent square in~\Cref{fig:2s-ev-dep-sq}. As the filler $g$ is cocartesian by right cancelation, and vertical at the same time it is an isomorphism, hence an identity $\sigma(x,y) = \yon d(\sigma(a,b))(x,y)$.
\end{proof}

\subsection{Dependent and absolute two-sided Yoneda Lemma}

Following~\cite[Theorem~9.5]{RS17} and~\cite[Theorem~7.2.2]{BW21}, we obtain the \emph{dependent Yoneda Lemma for two-sided families} by~\Cref{prop:fiber-init-term-2s}, and again this will in turn imply the absolute version. Simultaneously, this functions as a type-theoretic version of~\cite[Theorem~7.3.2]{RV21}.

\begin{theorem}[Dependent Yoneda Lemma for two-sided families]\label{thm:dep-yon-2s}
	Let $Q: \comma{a}{A} \times \comma{B}{b} \to \UU$ be a two-sided family over Rezk types $A$ and $B$. For any $a:A$ and $b:$, the evaluation map
	\[ \ev_{\id_{\pair{a}{b}}} : \Big( \prod_{\comma{a}{A} \times \comma{B}{b}}^\tscart Q\Big) \rightarrow Q(\id_a,\id_b) \]
	is an equivalence.
\end{theorem}

\begin{proof}
	Recall that by~\cite[Lemma~8.9]{RS17}, the identity map $\id_a$ is an initial object of the comma type $\comma{a}{A}$, while analogously the identity map $\id_b$ is a terminal object of the cocomma type $\comma{B}{b}$. Thus, the claim follows as an instance of~\Cref{prop:fiber-init-term-2s}. 
\end{proof}

We also obtain an absolute analogue in the fashion of~\cite[Theorem~9.1]{RS17} and~ \cite[Theorem~7.2.3]{BW21}.

\begin{theorem}[Absolute Yoneda Lemma for two-sided families, \protect{\cite[Theorem~7.3.2]{RV21}}]\label{thm:abs-yon-2s}
	Let $P: A \times B \to \UU$ be a two-sided family over Rezk types $A$ and $B$. For any $a:A$ and $b:B$, the evaluation map
	\[ \ev_{\pair{\id_a}{\id_b}} : \Big( \prod_{\substack{u:\comma{a}{A} \\ v:\comma{B}{b}}}^\tscart P(\partial_1 \, u, \partial_0 \, v)\Big) \rightarrow P(a,b) \]
	is an equivalence.
\end{theorem}

\begin{proof}
	The claim follows by setting $Q \jdeq \pair{\partial_1}{\partial_0}^* P : \comma{a}{A} \times \comma{B}{b} \to \UU$ in~\Cref{thm:dep-yon-2s}. In particular, $Q$ is a two-sided fibration again by pullback stability.	
\end{proof}

\section{Discrete two-sided families}\label{sec:2s-disc}

\subsection{Definition and characterization}

\begin{definition}[Two-sided discrete families, \protect{\cite[Definition~8.28]{RS17}}]\label{def:2s-disc-cart}
	Let $P:A \to B \to \UU$ be a two-variable family over Rezk types $A$ and $B$. Then $P$ is a \emph{two-sided discrete family} if for all $a:A$, $b:B$ the family $P_b:A \to \UU$ is covariant and $P^a:B \to \UU$ is contravariant.
\end{definition}

\begin{proposition}[Two-sided discrete families as discrete objects, \cf~\protect{\cite[Prop.~7.2.4]{RV21}}]\label{prop:2s-disc}
	Given $P:A \to B \to \UU$ over Rezk types, the following are equivalent:
	\begin{enumerate}
		\item\label{it:2s-disc-i} The family $P$ is two-sided discrete.
		\item\label{it:2s-disc-ii} The family $P$ is both cocartesian on the left and cartesian on the right, and every bifiber $P(a,b)$ is discrete, for any $a:A$, $b:B$.
	\end{enumerate}
\end{proposition}

\begin{proof}
	\begin{description}
		\item[$\ref{it:2s-disc-ii} \implies \ref{it:2s-disc-i}$] 
		By \Cref{prop:char-fib-cocart-left}, \Cref{it:coc-left-iii}, there is a fibered adjunction which pulls back as follows by~\cite[Proposition~B.2.3]{BW21}:\footnote{In contrast to the current version of~\cite[Proposition~B.2.3]{BW21} one only needs the fibrations involved to be isoinner, and not cocartesian, which in any case becomes clear from the given proof.}
		\[\begin{tikzcd}
			{E_b} &&& E \\
			& {\xi_b\downarrow A} &&& \comma{\xi}{A} \\
			{A \times \mathbf{1}} &&& {A \times B}
			\arrow[from=1-1, to=1-4, shorten <=22pt]
			\arrow[""{name=0, anchor=center, inner sep=0}, dashed, from=1-1, to=2-2]
			\arrow[two heads, from=1-1, to=3-1]
			\arrow[two heads, from=2-2, to=3-1]
			\arrow["{\id_A \times b}"{description}, from=3-1, to=3-4]
			\arrow[two heads, from=1-4, to=3-4]
			\arrow[two heads, from=2-5, to=3-4]
			\arrow[""{name=1, anchor=center, inner sep=0}, "\iota"{description}, from=1-4, to=2-5]
			\arrow[""{name=2, anchor=center, inner sep=0}, "\tau"{description}, curve={height=18pt}, dashed, from=2-5, to=1-4]
			\arrow[""{name=3, anchor=center, inner sep=0}, curve={height=12pt}, dashed, from=2-2, to=1-1, curve={height=18pt}]
			\arrow["\lrcorner"{anchor=center, pos=0.125}, draw=none, from=2-2, to=3-4]
			\arrow["\lrcorner"{anchor=center, pos=0.125}, shift right=5, draw=none, from=1-1, to=3-4]
			\arrow["\dashv"{anchor=center, rotate=-137}, draw=none, from=3, to=0]
			\arrow["\dashv"{anchor=center, rotate=-133}, draw=none, from=2, to=1]
			\arrow[from=2-2, to=2-5, crossing over]
		\end{tikzcd}\]
		This means exactly that $P^a$ is covariant. The analogous reasoning establishes the claim for $P_b$ being covariant since $P$ is cartesian on the right, for any $b:B$.
		Now, since any $P(a,b)$ is discrete, and the fibers of $P_b$ are given by $P(a,b)$ for any $a:A$, we obtain that $P^a:B \to \UU$ is a cocartesian family with discrete fibers, which is equivalent to $P^a$ being covariant by~\cite[Corollary~6.1.4]{BW21}.
		\item[$\ref{it:2s-disc-i} \implies \ref{it:2s-disc-ii}$] The fibrations $P_a$ and $P^b$, resp., being contra- and covariant, resp., imply that all bifibers $P(a,b)$ are discrete, for all $a:A$, $b:B$.
		
		Furthermore, for any $b:B$, the family $P^b: A \to \UU$ being cocartesian means that $P^B :A \to \UU$ is cocartesian, and in addition all $P^B$-cocartesian lifts are $P_A$-vertical, \ie, lie over an identity in $B$.\footnote{This could have also been used as a more direct argument to prove ``$\ref{it:2s-disc-ii} \implies \ref{it:2s-disc-i}$'' as well.}
	\end{description}
\end{proof}

\begin{corollary}[Co-/cart.~arrows and two-sided cart.~functors, \cf~\protect{\cite[Lemma~7.4.3]{RV21}}]
	In a two-sided discrete family $P:A \to B \to \UU$ an arrow is $P_B$-cocartesian if and only if it is $P^A$-vertical. Similarly, an arrow is $P^A$-cartesian if and only if it is $P_B$-vertical.
	In particular, any two-sided discrete cartesian family is two-sided cartesian, and any fibered functor between two-sided discrete families is two-sided cartesian.
\end{corollary}

\begin{proof}
	The first statement follows from the inspection in the proof of~\Cref{prop:2s-disc}, together with~\cite[Proposition~6.1.5]{BW21}.
	
	This also establishes that any two-sided discrete cartesian family is, in fact, two-sided cartesian: the commutation condition~\Cref{prop:comm-lifts} is readily verified, because for dependent arrows being vertical (in the respective appropriate sense) is already sufficient for being co-/cartesian, resp.
	
	By naturality, this implies that any fibered functor between two-sided discrete families is two-sided cartesian.
\end{proof}

\section{Conclusion}

Working in the type theory introduced by Riehl--Shulman~\cite{RS17} we have developed and investigated a synthetic notion of two-sided cartesian family of $\inftyone$-categories. This extends previous work~\cite{BW21} about the one-sided case. The present work also constitutes a translation of the results from Riehl--Verity's $\infty$-cosmos theory~\cite[Chapter~7]{RV21} into type theory.

Since we have been able to obtain a lot of the central results, namely characterization theorems, the Yoneda Lemma and several closure properties, we conclude that simplicial homotopy type theory allows to do a fair amount of technically challenging fibered $\inftyone$-category theory in a synthetic and more native way. However, in the future it would be desirable, given an appropriate internal construction of the universe types $\mathrm{Cat}$ or $\mathrm{Space}$, to capture the theory in a more categorical spirit. \Eg, we cannot describe in our type theory a sliced cocartesian fibration $\varphi: \pi \to_B \xi$ literally as a fibration internal to a Rezk type $\mathrm{Cat}/B$, since the latter has not been defined in this theory in the first place. Thus, our presentation necessarily becomes quite explicit at times somewhat blurring the bigger categorical picture. 

\appendix

\section{Fibered equivalences}\label{sec:fib-equiv}

We state some expected and useful closure properties of fibered equivalences.

\begin{lemma}[Right properness]\label{lem:rprop}
	Pullbacks of weak equivalences are weak equivalences again, \ie, given a pullback diagram
	\[\begin{tikzcd}
		{C \times_A B} && B \\
		C && A
		\arrow["\simeq" swap, from=1-3, to=2-3, "k"]
		\arrow["j", from=2-1, to=2-3]
		\arrow["\simeq", from=1-1, to=2-1, "k'" swap]
		\arrow[from=1-1, to=1-3]
		\arrow["\lrcorner"{anchor=center, pos=0.125}, draw=none, from=1-1, to=2-3]
	\end{tikzcd}\]
	then, as indicated if the right vertical map is a weak equivalence, then so is the left hand one.	
\end{lemma}

\begin{proof}
	Denote by $P: A \to \UU$ the straightening of $k:B \to A$, so that $B \simeq \sum_{a:A} P\,a$, and $D \simeq \sum_{c:C} P(j\,c)$. The map $k:B \to A$ (identified with its projection equivalence) is a weak equivalence if and only if $\prod_{a:A} \isContr(P\,a)$. This implies $\prod_{c:C} P(j\,c)$ which is equivalent to $k' = j^*k$ being a weak equivalence, as desired.
\end{proof}

\begin{proposition}[Homotopy invariance of homotopy pullbacks]\label{prop:htopy-inv-of-pb}
	Given a map between cospans of types
	\[\begin{tikzcd}
		C && A && B \\
		{C'} && {A'} && {B'}
		\arrow["{\simeq~ r}"', from=1-1, to=2-1]
		\arrow["{g'}"', from=2-1, to=2-3]
		\arrow["g", from=1-1, to=1-3]
		\arrow["{\simeq~p}"', from=1-3, to=2-3]
		\arrow["f"', from=1-5, to=1-3]
		\arrow["{f'}", from=2-5, to=2-3]
		\arrow["{\simeq~q}", from=1-5, to=2-5]
	\end{tikzcd}\]
	where the vertical arrows are weak euivalences the induced map
	\[ B \times_A C \to B' \times_{A'} C'\]
	is an equivalence as well.
\end{proposition}

\begin{proof}
	By right properness and $2$-out-of-$3$ the mediating map $C \to P \defeq C' \times_{A'} A$ is an equivalence as can be seen from the diagram:
	\[\begin{tikzcd}
		C &&& A \\
		& P \\
		{C'} &&& {A'}
		\arrow["\simeq"{description}, from=1-1, to=3-1]
		\arrow[from=3-1, to=3-4]
		\arrow[from=1-1, to=1-4]
		\arrow["\simeq", from=1-4, to=3-4]
		\arrow["\simeq"{description}, dashed, from=1-1, to=2-2]
		\arrow[from=2-2, to=1-4]
		\arrow["\simeq"{description}, from=2-2, to=3-1]
		\arrow["\lrcorner"{anchor=center, pos=0.125}, draw=none, from=2-2, to=3-4]
	\end{tikzcd}\]
	This gives the following cube
	\[\begin{tikzcd}
		{C \times_A B} && B \\
		& C && A \\
		{C' \times_{A'} B'} && {B'} \\
		& {C'} && {A'}
		\arrow[from=1-1, to=3-1]
		\arrow[from=3-3, to=4-4]
		\arrow[from=3-1, to=4-2]
		\arrow[from=1-1, to=2-2]
		\arrow[from=1-3, to=2-4]
		\arrow["\lrcorner"{anchor=center, pos=0.125, rotate=45}, draw=none, from=1-1, to=2-4]
		\arrow["\lrcorner"{anchor=center, pos=0.125}, draw=none, from=2-2, to=4-4]
		\arrow["\lrcorner"{anchor=center, pos=0.125, rotate=45}, draw=none, from=3-1, to=4-4]
		\arrow[from=4-2, to=4-4]
		\arrow[from=3-1, to=3-3]
		\arrow["\simeq"{description, pos=0.3}, from=1-3, to=3-3]
		\arrow["\simeq"{description, pos=0.3}, from=2-4, to=4-4]
		\arrow[from=1-1, to=1-3]
		\arrow["\simeq"{description, pos=0.3}, from=2-2, to=4-2, crossing over]
		\arrow[from=2-2, to=2-4, crossing over]
	\end{tikzcd}\]
	and by the Pullback Lemma we know that the back face is a pullback, too. Then again, by right properness
	\[B \times_A C \to B' \times_{A'} C' \]
	is an equivalence.
\end{proof}

\begin{proposition}
	Given a fibered equivalence as below
	\[\begin{tikzcd}
		F && E \\
		& G \\
		& B
		\arrow[two heads, from=1-1, to=2-2]
		\arrow[two heads, from=1-3, to=2-2]
		\arrow[two heads, from=2-2, to=3-2]
		\arrow[two heads, from=1-1, to=3-2]
		\arrow[two heads, from=1-3, to=3-2]
		\arrow["\simeq" swap, "\varphi", from=1-1, to=1-3]
	\end{tikzcd}\]
	and a map $k:A \to B$ the fibered equivalence $\varphi$ pulls back as shown below:
	\[\begin{tikzcd}
		{F'} &&&&& F \\
		&&& {E'} &&&&& E \\
		& {G'} &&&&& {G} \\
		\\
		A &&&&& B
		\arrow["\varphi", "\simeq" swap, from=1-6, to=2-9]
		\arrow[two heads, from=2-9, to=5-6]
		\arrow[two heads, from=1-6, to=3-7]
		\arrow[two heads, from=2-9, to=3-7]
		\arrow[from=1-1, to=1-6]
		\arrow["\varphi'", "\simeq" swap, dashed, from=1-1, to=2-4, shorten >=8pt]
		\arrow[two heads, from=2-4, to=3-2]
		\arrow[two heads, from=1-1, to=5-1]
		\arrow["k", from=5-1, to=5-6]
		\arrow[two heads, from=3-2, to=5-1]
		\arrow[two heads, from=3-7, to=5-6]
		\arrow["\lrcorner"{anchor=center, pos=0.125}, draw=none, from=3-2, to=5-6]
		\arrow["\lrcorner"{anchor=center, pos=0.125}, draw=none, from=2-4, to=5-6]
		\arrow["\lrcorner"{anchor=center, pos=0.125}, draw=none, from=1-1, to=5-6]
		\arrow[two heads, from=1-6, to=5-6]
		\arrow[two heads, from=1-1, to=3-2]
		\arrow[from=3-2, to=3-7, crossing over]
		\arrow[two heads, from=2-4, to=5-1, crossing over]
		\arrow[from=2-4, to=2-9, crossing over]
	\end{tikzcd}
	\]
\end{proposition}

\begin{proof}
	By projection equivalence, we can consider families $R:B \to \UU$, $P: F \jdeq \totalty{R} \to \UU$, $Q : E \jdeq \totalty{P} \to \UU$, with $F \jdeq \totalty{Q}$. The fibered equivalence $\varphi$ is given by a family of equivalences
	\[ \varphi: \prod_{\substack{b:B \\ x:R\,b}} \big( \sum_{e:P\,b\,x} Q \, b \, x \, e \big) \stackrel{\simeq}{\longrightarrow} P\,b\,x. \]
	The induced family
	\[ \varphi' \defeq \lambda a,x.\varphi(k\,a,x): \prod_{\substack{b:B \\ x:R\,k(a)}} \big( \sum_{e:P\,k(a)\,x} Q \, k(a) \, x \, e \big) \stackrel{\simeq}{\longrightarrow} P\,k(a)\,x\]
	also constitutes a fibered equivalence. Commutation of all the diagrams is clear since, after projection equivalence, all the vertical maps are given by projections.
\end{proof}

		\begin{lemma}[Closedness of fibered equivalences under dependent products]\label{lem:fib-eq-closed-pi}
			Let $I$ be a type. Suppose given a family $B:I\to \UU$ and indexed families $P,Q:\prod_{i:I} B_i \to \UU$ together with a fiberwise equivalence $\varphi: \prod_{i:I} \prod_{b:B} P_i\,b \stackrel{\simeq}{\longrightarrow} Q_i\,b$. Then the map
			\[ \prod_{i:I} \varphi_i : \big(\prod_{i:I} \totalty{P_i} \big) \longrightarrow_{\prod_{i:I} B_i} \big(\prod_{i:I} \totalty{Q_i} \big) \]
			induced by taking the dependent product over $I$ is a fiberwise equivalence, too.
		\end{lemma}
		
		\begin{proof}
			For $i:I$, we fibrantly replace the given fiberwise equivalence $\varphi_i$ by projections, giving rise to (strictly) commutative diagrams:
			\[\begin{tikzcd}
				{\sum_{b:B} \sum_{x:Q_i(b)} P_i(b,x) \simeq E_i} && {F_i \simeq \sum_{b:B} P_i(b)} \\
				& {B_i}
				\arrow["{\pi_i}"', two heads, from=1-1, to=2-2]
				\arrow["{\xi_i}", two heads, from=1-3, to=2-2]
				\arrow["{\varphi_i}", from=1-1, to=1-3, "\simeq" swap]
			\end{tikzcd}\]
			Now, $\varphi$ being a fiberwise equivalence is equivalent to
			\begin{align*}
				\prod_{i:I} \isEquiv(\varphi_i)	&\simeq  \prod_{i:I} \prod_{\substack{b:B_i \\ x:Q_i(b)}} \isContr(P_i(b,x)) \\
				&	\simeq   \prod_{\beta:\prod_{i:I} B_i} \prod_{\sigma:\prod_{i:I} \beta^* P_i} \prod_{i:I} \isContr\left( P_i(\beta(i), \sigma(i))\right).
			\end{align*}
			By (weak) function extensionality,\footnote{\cf~\cite[Theorem~13.1.4(ii)]{RijIntro}, or the discussion at the beginning of~\cite[Section~4.4.]{RS17}} this implies
			\begin{align*}
				&\prod_{\beta:\prod_{i:I} B_i} \prod_{\sigma:\prod_{i:I} \beta^*Q_i} \isContr \left( \prod_{i:I} \pair{\beta}{\sigma}^* P_i\right) \\ 
				\simeq & \isEquiv\left( \prod_{i:I} \varphi_i \right)
			\end{align*}
			wich yields the desired statement. Note, that the latter equivalence follows by projection equivalence of the diagram obtained by applying $\prod_{i:I}(-)$:
			\[\begin{tikzcd}
				{\sum_{\substack{\beta:\prod_{i:I} B_i \\ \sigma:\prod_{i:I} \beta^*Q_i}} \prod_{i:I} \langle\beta,\sigma\rangle^*P_i \simeq \prod_{i:I} E_i} && {\mathllap{\prod_{i:I} F_i \simeq} \sum_{\beta:\prod_{j:I} B_i} \prod_{i:I} \beta^*Q_i} \\
				& {\prod_{i:I} B_i}
				\arrow["{\prod_{i:I} \pi_i}"', from=1-1, to=2-2]
				\arrow["{\prod_{i:I} \varphi_i}"{pos=0.3}, shorten >=45pt, from=1-1, to=1-3]
				\arrow["{\prod_{i:I} \xi_i}", from=1-3, to=2-2]
			\end{tikzcd}\]
		\end{proof}
		
		\begin{lemma}[Closedness of fibered equivalences under sliced products]\label{prop:fib-eqs-sliced-pi}
			Given indexed families $P,Q:I \to B \to \UU$ and a family of fibered equivalences $\prod_{i:I} \prod_{b:B} P_i \,b  \stackrel{\simeq}{\longrightarrow} Q_i \,b$. Then the induced fibered functor
			\[ \times_{i:I}^B \varphi_i : \prod_{i:I} \prod_{b:B} \times_{i:I}^B P_i \,b \longrightarrow \times_{i:I}^B Q_i\,b \]
			between the sliced products over $B$ is also a fibered equivalence.
		\end{lemma}
		
		\begin{proof}
			As usual, denote for $i:$ by $\pi_i \defeq \Un_B(P_i) : E_i \to B$ and $\xi_i \defeq \Un_B(Q_i) : F_i \to B$ the unstraightenings of the given fibered families, giving rise to a (strict) diagram:
			\[\begin{tikzcd}
				{E_i} && {F_i} \\
				& B
				\arrow["{\pi_i}"', two heads, from=1-1, to=2-2]
				\arrow["{\xi_i}", two heads, from=1-3, to=2-2]
				\arrow["{\varphi_i}", from=1-1, to=1-3]
			\end{tikzcd}\]
			Since weak equivalences are closed under taking dependent products, the induced fibered map $\prod_{i:I} \varphi_i: \prod_{i:I} E_i \to_{I \to B} \prod_{i:I} F_i$ is also a weak equivalence, and by right properness \Cref{lem:rprop} the desired mediating map is as well:
			\[\begin{tikzcd}
				{\times^B_{i:I} E_i} && {\prod_{i:I} E_i} \\
				& {\times^B_{i:I} F_i} && {\prod_{i:I} F_i} \\
				& B && {B^I}
				\arrow["{\prod_{i:I} \varphi_i}", from=1-3, to=2-4]
				\arrow[from=1-1, to=1-3]
				\arrow["{\times_{i:I}^B \varphi_i}"{pos=0.2} description, dashed, from=1-1, to=2-2]
				\arrow[two heads, from=2-4, to=3-4]
				\arrow["\cst", from=3-2, to=3-4]
				\arrow[two heads, from=2-2, to=3-2]
				\arrow[curve={height=6pt}, two heads, from=1-1, to=3-2]
				\arrow["\lrcorner"{anchor=center, pos=0.125}, draw=none, from=2-2, to=3-4]
				\arrow["\lrcorner"{anchor=center, pos=0.125}, draw=none, from=1-1, to=3-4]
				\arrow[curve={height=6pt}, two heads, from=1-3, to=3-4]
				\arrow[from=2-2, to=2-4, crossing over]
			\end{tikzcd}\]
		\end{proof}

		\begin{proposition}
			For an indexing type $I$ and a base Rezk type $B$, families of fibered equivalences between Rezk types over $B$ are closed under taking sliced products, \ie: Given a family of isoinner fibrations over $B$ together with a fibered equivalence as below left, the induced maps on the right make up a fibered equivalence as well:
			\[\begin{tikzcd}
				{F_i} && {E_i} & {\times_{i:I}^B F_i} && {\times_{i:I}^B E_i} \\
				& {G_i} & {} & {} & {\times_{i:I}^B G_i} \\
				& B &&& B
				\arrow[""{name=0, anchor=center, inner sep=0}, "{\varphi_i}"', curve={height=6pt}, from=1-1, to=1-3]
				\arrow[""{name=1, anchor=center, inner sep=0}, "{\psi_i}"', curve={height=12pt}, from=1-3, to=1-1]
				\arrow[""{name=2, anchor=center, inner sep=0}, "{\times_{i:I}^B \varphi_i}"', curve={height=6pt}, from=1-4, to=1-6]
				\arrow[""{name=3, anchor=center, inner sep=0}, "{\times_{i:I}^B \psi_i}"', curve={height=12pt}, from=1-6, to=1-4]
				\arrow[two heads, from=1-1, to=2-2]
				\arrow[two heads, from=1-3, to=2-2]
				\arrow[curve={height=12pt}, two heads, from=1-1, to=3-2]
				\arrow[curve={height=-12pt}, two heads, from=1-3, to=3-2]
				\arrow[two heads, from=2-2, to=3-2]
				\arrow[curve={height=12pt}, two heads, from=1-4, to=3-5]
				\arrow[curve={height=-12pt}, two heads, from=1-6, to=3-5]
				\arrow[two heads, from=2-5, to=3-5]
				\arrow[two heads, from=1-6, to=2-5]
				\arrow[two heads, from=1-4, to=2-5]
				\arrow[squiggly, from=2-3, to=2-4]
				\arrow["\dashv"{anchor=center, rotate=-90}, draw=none, from=1, to=0]
				\arrow["\dashv"{anchor=center, rotate=-90}, draw=none, from=3, to=2]
			\end{tikzcd}\]
		\end{proposition}
		
		\begin{proof}
			This is a consequence of~\Cref{lem:fib-eq-closed-pi,prop:fib-eqs-sliced-pi}, considering the following diagram:
			\[\begin{tikzcd}
				{\times^B_{i:I} F_i} &&&&& {\prod_{i:I} E_i} \\
				& {} & {\times^B_{i:I} F_i} &&&&& {\prod_{i:I} F_i} \\
				& {\times^B_{i:I} G_i} &&&&& {\prod_{i:I} G_i} \\
				B &&&&& {B^I}
				\arrow[from=1-1, to=1-6]
				\arrow["\simeq"{description}, from=1-1, to=2-3]
				\arrow[two heads, from=1-1, to=4-1]
				\arrow[from=4-1, to=4-6]
				\arrow[two heads, from=1-6, to=4-6]
				\arrow[dashed, two heads, from=1-1, to=3-2]
				\arrow[two heads, from=3-2, to=4-1]
				\arrow[dashed, two heads, from=2-3, to=3-2]
				\arrow[two heads, from=1-6, to=3-7]
				\arrow["\simeq"{description}, from=1-6, to=2-8]
				\arrow[two heads, from=2-8, to=3-7]
				\arrow[curve={height=-22pt}, two heads, from=2-8, to=4-6]
				\arrow[two heads, from=3-7, to=4-6]
				\arrow["\lrcorner"{anchor=center, pos=0.125}, draw=none, from=2-3, to=3-7]
				\arrow["\lrcorner"{anchor=center, pos=0.125}, draw=none, from=3-2, to=4-6]
				\arrow["\lrcorner"{anchor=center, pos=0.125}, draw=none, from=1-1, to=2-8]
				\arrow[from=3-2, to=3-7]
				\arrow[from=2-3, to=2-8, crossing over]
				\arrow[curve={height=-22pt}, two heads, from=2-3, to=4-1, crossing over]
			\end{tikzcd}\]
		\end{proof}
		
		\section{Fibered (LARI) adjunctions}\label{sec:fib-lari-adj}
		
		Building on previous work \cite[Section~11]{RS17} and \cite[Appendix~B]{BW21} we provide a characterization of fibered LARI adjunctions along similar lines.

			\begin{theorem}[Characterizations of fibered adjunctions, cf.~{\protect\cite[Theorem~11.23]{RS17}, \cite[Theorem~B.1.4]{BW21}}]\label{thm:char-fib-adj}
				Let $B$ be a Rezk type. For $P,Q:B \to \UU$ isoinner families we write $\pi \defeq \Un_B(P) : E \defeq \totalty{P} \fibarr B$ and $\xi\defeq \Un_B(Q):F \jdeq \totalty{Q} \fibarr B$. Given a fibered functor $\varphi: E \to_B F$ such that (strictly)
				\[\begin{tikzcd}
					E && F \\
					& B
					\arrow["\varphi", from=1-1, to=1-3]
					\arrow["\pi"', two heads, from=1-1, to=2-2]
					\arrow["\xi", two heads, from=1-3, to=2-2]
				\end{tikzcd}\]
				the following are equivalent propositions:
				\begin{enumerate}
					\item\label{it:fib-ladj-vert-i} The type of \emph{fibered left adjoints} of $\varphi$, \ie, fibered functors $\psi$ which are ordinary (transposing) left adjoints of $\varphi$ whose unit, moreover, is vertical.
					\item\label{it:fib-ladj-vert-ii} The type of fibered functors $\psi:F \to_B E$ together with a vertical $2$-cell
					$\eta: \id_F \Rightarrow_B \varphi \psi$ 
					 s.t.
					\[\Phi_\eta \defeq \lambda u,d,e,k.\varphi_u \,k \circ \eta_d: \prod_{\substack{a,b:B \\ u:a \to b}} \prod_{d:Q\,a,~  e:P\,b} \big(\psi_a \,d \to^P_u e \big) \to \big( d \to^Q_u \varphi_b \, e)\]
					is a fiberwise equivalence. 
					\item\label{it:fib-ladj-sliced} The type of \emph{sliced} (or \emph{fiberwise}) \emph{left adjoints (over $B$)} to $\varphi$, \ie, fibered functors $\psi:F \to_B E$ together with a fibered equivalence
					\[ \relcomma{B}{\psi}{E} \simeq_{F \times_B E} \relcomma{B}{F}{\varphi}.\]
					\item\label{it:fib-ladj-bi-diag-i} The type of \emph{bi-diagrammatic fibered} (or \emph{fiberwise}) \emph{left adjoints}, \ie, fibered functors $\psi$ together with:
					\begin{itemize}
						\item a vertical natural transformation $\eta : \id_F \Rightarrow_B \varphi \psi$
						\item two \emph{vertical} natural transformations $\varepsilon, \varepsilon': \psi \varphi \Rightarrow_B \id_E$
						\item homotopies\footnote{by Segal-ness, the witnesses for the triangle identities are actually unique up to contractibility} $\alpha : \varphi \varepsilon \circ \eta \varphi =_{E \to F} \id_\varphi, ~\beta : \varepsilon' \psi \circ \psi \eta =_{F \to E} \id_\psi$
					\end{itemize}
					\item\label{it:fib-ladj-bi-diag-ii} The type of fibered functors $\psi$ together with:
					\begin{itemize}
						\item a vertical natural transformation $\eta : \id_F \Rightarrow_B \varphi \psi$
						\item two natural transformations $\varepsilon, \varepsilon': \psi \varphi \Rightarrow \id_E$
						\item homotopies $\alpha : \varphi \varepsilon \circ \eta \varphi =_{E \to F} \id_\varphi, ~\beta : \varepsilon' \psi \circ \psi \eta =_{F \to E} \id_\psi$
					\end{itemize}
				\end{enumerate}
			\end{theorem}
			
			\begin{proof} ~ \\
				At first, we prove that, given a \emph{fixed} and \emph{fibered} functor $\psi:F \to_B E$ the respective witnessing data are propositions.\footnote{This justifies the ensuing list of \emph{logical} equivalences.}
				\begin{description}
					\item[$\ref{it:fib-ladj-vert-i} \iff \ref{it:fib-ladj-bi-diag-ii}$] This follows from the equivalence between transposing left adjoint and bi-diagrammatic left adjoint data, \cf~\cite[Theorem~11.23]{RS17}.
					\item[$\ref{it:fib-ladj-bi-diag-i} \implies \ref{it:fib-ladj-bi-diag-ii}$] This is clear since the latter is a weakening of the former. 
					\item[$\ref{it:fib-ladj-bi-diag-ii} \implies \ref{it:fib-ladj-bi-diag-i}$] Denoting the base component of $\varepsilon$ by $v:\Delta^1 \to (B \to B) \simeq B \to (\Delta^1 \to B)$, projecting down from $\alpha$ via $\xi$ we obtain the identity $\xi \alpha : v \circ \id_B = \id_B$. Thus $\varepsilon$ is vertical, and similarly one argues for $\varepsilon'$.
					\item[$\ref{it:fib-ladj-sliced} \iff \ref{it:fib-ladj-bi-diag-i}$] Given the fibered functor $\psi$, both lists of data witness that for every $b:B$ the components $\psi_b \dashv \varphi_b: P\,b \to Q\,b$ define an adjunction between the fibers, again by~\cite[Theorem~11.23]{RS17}.
					\item[$\ref{it:fib-ladj-vert-ii} \implies \ref{it:fib-ladj-sliced}$] The latter is an instance of the former.
					\item[$\ref{it:fib-ladj-bi-diag-i} \implies \ref{it:fib-ladj-vert-ii}$] Using naturality and the triangle identities, we show that the fiberwise conditions (vertical case) can be lifted to the case of arbitrary arrows in the base.\footnote{We thank Ulrik Buchholtz for pointing out the subsequent argument.}
					Consider the transposing maps:
					\begin{align*}
						\Phi \defeq \lambda k.\varphi_u\,k \circ \eta_d & : \prod_{\substack{a,b:B \\ u:a \to b}} \prod_{\substack{d:Q\,a \\  e:P\,b}} \big(\psi_a \,d \to^P_u e \big) \to \big( d \to^Q_u \varphi_b \, e \big) \\
						\Psi \defeq \lambda m.\varepsilon_e \circ \psi_u\,m  & : \prod_{\substack{a,b:B \\ u:a \to b}} \prod_{\substack{d:Q\,a \\ e:P\,b}}  \big( d \to^Q_u \varphi_b \, e) \to \big(\psi_a \,d \to^P_u e \big) 
					\end{align*}
					The first roundtrip yields:
					\[\begin{tikzcd}
						{\big(k:\psi_a\,d} & {e\big)} & {\big(\varphi_u\,k \circ \eta_d: d} & {\varphi_b\,e \big)} & {} \\
						&& {\big( \varepsilon_e \circ \psi_a(\varphi_u\,k \circ \eta_d):\psi_a\,d} & e
						\arrow["P", from=1-1, to=1-2]
						\arrow["\Phi", maps to, from=1-2, to=1-3]
						\arrow[""{name=0, anchor=center, inner sep=0}, "Q", from=1-3, to=1-4]
						\arrow["P", from=2-3, to=2-4]
						\arrow["\Psi"{description}, shorten <=6pt, maps to, from=0, to=2-3]
					\end{tikzcd}\]
					The result yields back $k$ using a triangle identity in the triangle on the left, and naturality of $\varepsilon$ in the square on the right:
					\[\begin{tikzcd}
						{\psi_a\,d} && {(\psi \varphi \psi)_a\,d} & {} & {(\psi\varphi)_b\,e} \\
						&& {\psi_a\,d} && e
						\arrow["{\psi_a\eta_d}", from=1-1, to=1-3]
						\arrow[""{name=0, anchor=center, inner sep=0}, "{\id_{\psi_a\,d}}"', curve={height=12pt}, Rightarrow, no head, from=1-1, to=2-3]
						\arrow["{\psi_a\varepsilon_d}", from=1-3, to=2-3]
						\arrow["{(\psi\varphi)_u\,k}", from=1-3, to=1-5]
						\arrow["k"', from=2-3, to=2-5]
						\arrow["{\varepsilon_e}", from=1-5, to=2-5]
						\arrow[Rightarrow, no head, from=2-3, to=1-5]
						\arrow[shorten <=6pt, Rightarrow, no head, from=0, to=1-3]
					\end{tikzcd}\]
					In addition, we have also used naturality of $\varepsilon$ for $\psi_a \varepsilon_d \jdeq \varepsilon_{\psi_a\,d}$. An analogous argument proves the other roundtrip.
				\end{description}
				We have proven, that relative to a fixed fibered functor $\psi: F \to_B E$ the different kinds of witnesses that this is a fibered left adjoint to $\varphi$ are equivalent propositions, giving rise to a predicate $\isFibLAdj_\varphi : (F \to_B E) \to \Prop$. What about the $\Sigma$-type $\FibLAdj_\varphi \defeq \sum_{\psi:F \to_B E} \isFibLAdj_\varphi(\psi)$ as a whole? E.g.~using the data from \cref{it:fib-ladj-sliced} (after conversion via~\cite[Theorem~11.23]{RS17}), said type is equivalent to
				\begin{align*}
					\FibLAdj_\varphi(\psi) & \simeq \sum_{\psi:\prod_{b:B}  P\,b \to Q\,b} \sum_{\eta:\prod_{b:B} \prod_{d:Q\,b} \hom_{Q\,b}(d,(\varphi\,\psi)_b\,d)} \prod_{\substack{b:B \\ d:Q\,b \\e:P\,b}} \isEquiv\big( \lambda k.\varphi_b(k) \circ \eta_d\big) \\ 
					& \simeq  \prod_{\substack{b:B \\ d:Q\,b}} \sum_{\psi_b:P\,b} \sum_{\eta_d:d \to_{Q\,b} \varphi_b(\psi_b d)} \prod_{e:P\,b} \isEquiv(\lambda k.\varphi_b(k) \circ \eta_d).
				\end{align*}
				Finally, one shows that this is indeed a proposition, completely analogously to the argument given in the proof of \cite[Theorem~11.23]{RS17} for the non-dependent case.
			\end{proof}
			
			\begin{definition}[Fibered (LARI) adjunction]\label{def:fib-lari-adj}
				Let $B$ be a Rezk type and $\pi: E \fibarr B$, $\xi:F \fibarr B$ be isoinner fibrations, with $P \defeq \St_B(\pi)$ and $Q \defeq \St_B(\xi)$. Given a fibered functor $\varphi: E \to_B F$, the data of a \emph{fibered left adjoint right inverse (LARI) adjunction} is given by
				\begin{itemize}
					\item a fibered functor $\psi: F \to_B E$,
					\item and an equivalence $\Phi: \relcomma{B}{\psi}{E} \simeq_{F \times_B E} \relcomma{B}{F}{\varphi}$ s.t.~the fibered unit
					\[ \eta_\Phi \defeq \lambda b,d.\Phi_{b,d,\psi_b \, d}(\id_{\psi_b\,d}):\prod_{b:B} \prod_{d:Q\,b} d \to_{Q\,b} (\varphi \psi)_b(d) \]
					is a componentwise homotopy.
				\end{itemize} 
				Together, this defines the data of a \emph{fibered LARI adjunction}. Diagrammatically, we represent this by:
				\[\begin{tikzcd}
					E && F \\
					& B
					\arrow["\pi"', two heads, from=1-1, to=2-2]
					\arrow["\xi", two heads, from=1-3, to=2-2]
					\arrow[""{name=0, anchor=center, inner sep=0}, "\varphi"', from=1-1, to=1-3]
					\arrow[""{name=1, anchor=center, inner sep=0}, "\psi"', curve={height=12pt}, dotted, from=1-3, to=1-1]
					\arrow["\dashv"{anchor=center, rotate=-90}, draw=none, from=1, to=0]
				\end{tikzcd}\]
			\end{definition}
			
			In fact, as established in the previous works of~\cite[Section~11]{RS17} the unit of a coherent adjunction is determined uniquely up to homotopy. Hence, using the characterizations of a (coherent) LARI adjunction, the type of fibered LARI adjunctions in the above sense is equivalent to the type of LARI adjunctions which are also fibered adjunctions. This implies the validity of the familiar closure properties for this restricted class as well.

			\section{Sliced commas and products}\label{sec:sl-comma-prod}
			
			We record here explicitly some closure properties involving sliced commas and products that are often used, especially in the treatise of two-sided fibrations and related notions.
			
			\begin{proposition}[Dependent products of sliced commas]\label{prop:dep-prod-comm-sl-commas}
				For a type $I$ and $i:I$, given fibered cospans
				\[ \psi_i: F_i \to_{B_i} G_i \leftarrow_{B_i} E_i : \varphi_i \]
				of Rezk types, taking the dependent product fiberwisely commutes with forming sliced comma types:
				\[\begin{tikzcd}
					& {\psi_i \downarrow_{B_i} \varphi_i} \\
					{F_i} && {E_i} & {\prod_{i:I} (\psi_i \downarrow_{B_i} \varphi_i)} & {\left(\prod_{i:I} \psi_i \right) \downarrow_{\left(\prod_{i:I} B_i \right)} \left(\prod_{i:I}\varphi_i\right)} \\
					& {G_i} & {} & {} & {} \\
					& {B_i} &&& {\prod_{i:I} B_i}
					\arrow["{\psi_i}"{description}, two heads, from=2-1, to=3-2]
					\arrow["{\varphi_i}"{description}, two heads, from=2-3, to=3-2]
					\arrow[two heads, from=3-2, to=4-2]
					\arrow[curve={height=12pt}, two heads, from=2-1, to=4-2]
					\arrow[curve={height=-12pt}, two heads, from=2-3, to=4-2]
					\arrow[from=1-2, to=2-1]
					\arrow[from=1-2, to=2-3]
					\arrow[shorten <=25pt, shorten >=25pt, Rightarrow, from=2-1, to=2-3]
					\arrow[two heads, from=2-4, to=4-5]
					\arrow[squiggly, from=3-3, to=3-4]
					\arrow[two heads, from=2-5, to=4-5]
					\arrow[""{name=0, anchor=center, inner sep=0}, shift left=2, curve={height=-6pt}, from=2-5, to=2-4]
					\arrow[""{name=1, anchor=center, inner sep=0}, shift left=2, curve={height=-6pt}, from=2-4, to=2-5]
					\arrow["\simeq"{description}, Rightarrow, draw=none, from=1, to=0]
				\end{tikzcd}\]
			\end{proposition} 
			
			\begin{proof}
				We denote by $P_i, Q_i, R_i: B_i \to \UU$ the straightenings of the given maps $E_i \fibarr B_i$, $F_i \fibarr B_i$, and $G_i \fibarr B_i$, resp.
				Using projection equivalence, the sliced commas are computed as
				\begin{align}\relcomma{B_i}{\psi_i}{\varphi_i} \simeq \sum_{b:B_i} \sum_{\substack{e:P_i(b) \\ d:Q_i(b)}} \big(\psi_i\big)_b(d) \longrightarrow_{R_i(b)} \big(\varphi_i\big)_b(e).
				\end{align}\label{eq:fam-sliced-comma}
				From this and the type-theoretic axiom of choice, we obtain as projection equivalence for
				\[ \prod_{i:I} (\psi_i \downarrow_{B_i} \varphi_i) \fibarr \prod_{i:I} B_i\]
				the type
				\begin{align*}
					\prod_{i:I} (\psi_i \downarrow_{B_i} \varphi_i)  & \stackrel{\text{(\ref{eq:fam-sliced-comma})} }{\simeq} \prod_{i:I}  \sum_{b:B_i}  \sum_{\substack{e:P_i(b) \\ d:Q_i(b)}} \big(\psi_i\big)_b(d) \longrightarrow_{R_i(b)} \big(\varphi_i\big)_b(e) \\
					&  \stackrel{\text{(AC)}}{\simeq} \sum_{\beta:\prod_{i:I} B_i} \prod_{i:I} \sum_{\substack{e:P_i(\beta(i)) \\ d:Q_i(\beta(i))}} \big(\psi_i\big)_{\beta(i)}(d) \longrightarrow_{R_i(\beta(i))} \big(\varphi_i\big)_{\beta(i)}(e) \\
					& \stackrel{\text{(\ref{eq:fam-sliced-comma})}}{\simeq} \sum_{\beta:\prod_{i:I} B_i} 
					\sum_{\substack{\sigma:\prod_{i:I} P_i(\beta(i)) \\	\tau:\prod_{i:I} Q_i(\beta(i))}} \big( \prod_{i:I} \psi(\tau) \big) \to_{ \big( \prod_{i:I} R_i(\beta(i))\big) }  \big( \prod_{i:I} \varphi_i(\sigma)\big) \\
					&  \stackrel{\text{(AC)}}{\simeq} \left(\prod_{i:I} \psi_i \right) \downarrow_{\left(\prod_{i:I} B_i \right)} \left(\prod_{i:I} \varphi_i \right)	
				\end{align*}
				This yields the desired fibered equivalence.
			\end{proof}
			
			\begin{corollary}[Products of commas in a slice]\label{lem:prod-comma-slice}
				Fix a base Rezk type $B$ be and an indexing type $I$. Given for $i:I$ an isoinner fibration $\pi_i:E_i \fibarr B$ consider a cospan of isoinner fibrations $\psi_i: F_i \to E_i \leftarrow G_i: \varphi$. Then we have a fibered equivalence:
				\[\begin{tikzcd}
					{ \times_{i:I}^B \big( \relcomma{B}{\psi_i}{\varphi_i}  \big)} && {\relcomma{B}{\left( \times_{i:I}^B \varphi_i \right)}{\left(\times_{i:I}^B \psi_i \right)}} \\
					{\prod_{i:I} F_i \times_B E_i} && {\left(\prod_{i:I} F_i\right) \times_{I \to B} \left(\prod_{i:I} E_i\right)} \\
					& B
					\arrow[two heads, from=1-1, to=2-1]
					\arrow["\simeq", from=1-1, to=1-3]
					\arrow["\simeq", from=2-1, to=2-3]
					\arrow[two heads, from=1-3, to=2-3]
					\arrow[two heads, from=2-1, to=3-2]
					\arrow[two heads, from=2-3, to=3-2]
				\end{tikzcd}\]
			\end{corollary}

			\begin{proposition}[Fibered (LARI) adjunctions are preserved by sliced products]\label{prop:fib-lari-pres-by-sl-prod}
				For an indexing type $I$ and a base Rezk type $B$, families of fibered (LARI) adjunctions between Rezk types over $B$ are closed under taking sliced products, \ie: Given a family of isoinner fibrations over $B$ together with a fibered (LARI) adjunction as below left, the 
				induced maps on the right make up a fibered (LARI) adjunction as well:
				\[\begin{tikzcd}
					{E_i} && {F_i} & {\times_{i:I}^B E_i} && {\times_{i:I}^B F_i} \\
					& B &&& B
					\arrow[two heads, from=1-3, to=2-2]
					\arrow[two heads, from=1-1, to=2-2]
					\arrow[""{name=0, anchor=center, inner sep=0}, "{\varphi_i}"{description}, curve={height=8pt}, from=1-1, to=1-3]
					\arrow[""{name=1, anchor=center, inner sep=0}, "{\psi_i}"{description}, curve={height=14pt}, from=1-3, to=1-1]
					\arrow[squiggly, from=1-3, to=1-4]
					\arrow[two heads, from=1-6, to=2-5]
					\arrow[two heads, from=1-4, to=2-5]
					\arrow[""{name=2, anchor=center, inner sep=0}, "{\times_{i:I}^B \varphi_i}"{description}, curve={height=8pt}, from=1-4, to=1-6]
					\arrow[""{name=3, anchor=center, inner sep=0}, "{\times_{i:I}^B \psi_i}"{description}, curve={height=14pt}, from=1-6, to=1-4]
					\arrow["\dashv"{anchor=center, rotate=-90}, draw=none, from=1, to=0]
					\arrow["\dashv"{anchor=center, rotate=-90}, draw=none, from=3, to=2]
				\end{tikzcd}\]
			\end{proposition}
			
			\begin{proof}
				Given a family of fibered adjunctions as indicated amounts to a family of fibered equivalences, themselves fibered over $B$, for $i:I$:
				\[\begin{tikzcd}
					{\psi_i \downarrow_B E_i} && {\varphi_i \downarrow_B E_i} \\
					& {F_i \times_B E_i} \\
					& B
					\arrow[two heads, from=1-1, to=2-2]
					\arrow[two heads, from=1-3, to=2-2]
					\arrow[two heads, from=1-1, to=3-2]
					\arrow[two heads, from=1-3, to=3-2]
					\arrow[""{name=0, anchor=center, inner sep=0}, curve={height=-12pt}, from=1-1, to=1-3]
					\arrow[""{name=1, anchor=center, inner sep=0}, curve={height=-6pt}, from=1-3, to=1-1]
					\arrow[two heads, from=2-2, to=3-2]
					\arrow["\simeq"{description}, draw=none, from=0, to=1]
				\end{tikzcd}\]
				Taking the dependent product over $i:I$ produces a fibered equivalence, itself fibered over $B^I$. Pullback along $\cst: B \to B^I$ yields the sliced products and again preserves the fibered equivalence:
				\[\begin{tikzcd}
					{	{\times_{i:I}^B \psi_i \downarrow_B E_i}} &&& {\prod_{i:I} \psi_i \downarrow_B E_i } & {} \\
					& {{\times_{i:I}^B F_i \downarrow_B \varphi_i }} & {} && {\prod_{i:I} F_i \downarrow_B \varphi_i} & {} \\
					& {{\times_{i:I}^B F_i \times_B E_i}} &&& {\prod_{i:I} F_i \times_B E_i} \\
					B &&& {B^I}
					\arrow[two heads, from=1-1, to=4-1]
					\arrow[from=1-1, to=1-4]
					\arrow["\cst"', from=4-1, to=4-4]
					\arrow[two heads, from=1-1, to=3-2]
					\arrow[two heads, from=3-2, to=4-1]
					\arrow[two heads, from=3-5, to=4-4]
					\arrow["\simeq"{description}, dashed, from=1-1, to=2-2]
					\arrow[two heads, from=2-2, to=3-2]
					\arrow["\simeq"{description}, from=1-4, to=2-5]
					\arrow[two heads, from=2-5, to=3-5]
					\arrow[two heads, from=1-4, to=4-4]
					\arrow["\lrcorner"{anchor=center, pos=0.125}, draw=none, from=3-2, to=4-4]
					\arrow["\lrcorner"{anchor=center, pos=0.125}, draw=none, from=2-2, to=3-5]
					\arrow["\lrcorner"{anchor=center, pos=0.125}, draw=none, from=1-1, to=2-5]
					\arrow[two heads, from=1-4, to=3-5, crossing over]
					\arrow[from=3-2, to=3-5, crossing over]
					\arrow[from=2-2, to=2-5, crossing over]
				\end{tikzcd}\]
				Since sliced products canonically commute with both sliced commas and fiber products, this gives a fibered equivalence
				\[{\scriptsize
					\begin{tikzcd}
					{\big( \times_{i:I}^B \psi_i  \big) \downarrow_{I \to B} \big( \times_{i:I}^B  E_i \big)} && {\big(\times_{i:I}^B F_i \big) \downarrow_{I \to B} \big( \times_{i:I}^B  \varphi_i\big)} \\
					\\
					& {\big(\times_{i:I}^B F_i \big) \times_{I \to B} \big(\times_{i:I}^B E_i \big)}
					\arrow[""{name=0, anchor=center, inner sep=0}, curve={height=-12pt}, from=1-1, to=1-3]
					\arrow[""{name=1, anchor=center, inner sep=0}, curve={height=-12pt}, from=1-3, to=1-1]
					\arrow[two heads, from=1-1, to=3-2]
					\arrow[two heads, from=1-3, to=3-2]
					\arrow["\simeq"{description}, Rightarrow, draw=none, from=0, to=1]
				\end{tikzcd} }\]
				which exactly yields the desired fibered adjunction of the sliced products.
			\end{proof}


%

\section*{Acknowledgments}
I am grateful to the US Army Research Office for the support of some stages of this work under MURI Grant W911NF-20-1-0082. I also thank the MPIM~Bonn for its hospitality and financial support during some work on this project. I wish to thank Ulrik Buchholtz, Emily Riehl, and Thomas Streicher for many helpful discussions, valuable feedback, and steady guidance. An anonymous referee has provided detailed and valuable corrections and suggestions to significantly improve the presentation of the paper, which I am highly grateful for.
I am thankful to Tim Campion and Sina Hazratpour for further discussions.

Furthermore, I am indebted to Ulrik Buchholtz for his collaboration on synthetic fibered $\inftyone$-category theory that made the work at hand possible in the first place.

\phantomsection%
\nocite{AFfib,GHT17,RV2cat,RVexp,RVscratch,RVyoneda,rasekh2021cartesian,LiBint,clementino2020lax,hermida1992fibred,rezk2017stuff,kock2013local,BorHandb2,BarwickShahFib,JoyQcat}

\printbibliography[heading=bibintoc]

\end{document}